\documentclass[11pt, a4paper,leqno]{amsart}
\usepackage{amsmath,amsthm,amscd,amssymb,amsfonts, amsbsy}
\usepackage{latexsym}
\usepackage{exscale}

\usepackage{pgf}

\calclayout
\allowdisplaybreaks

\newtheorem{theorem}{Theorem}[section]
\newtheorem{proposition}{Proposition}[section]
\newtheorem{lemma}{Lemma}[section]

\theoremstyle{definition}
\newtheorem{definition}{Definition}
\newtheorem{remark}{Remark}[section]

\numberwithin{equation}{section}

\newcommand{\beq}{\begin{equation}}
\newcommand{\bea}[1]{\begin{array}{#1} }
\newcommand{\eeq}{ \end{equation}}
\newcommand{\ea}{ \end{array}}

\newcommand{\ran}{\rangle}
\newcommand{\lan}{\langle}

\newcommand{\om}{\omega}

\newcommand{\sem}{\setminus}

\newcommand{\De}{\Delta}

\newcommand{\lsim}{\lesssim}

\def \rnn {{\mathbb {R}}^{N+1}}

\def \g {{\gamma}}

\def \t {{\tau}}

\def \d {{\delta}}

\def\Xint#1{\mathchoice
{\XXint\displaystyle\textstyle{#1}}%
{\XXint\textstyle\scriptstyle{#1}}%
{\XXint\scriptstyle\scriptscriptstyle{#1}}%
{\XXint\scriptscriptstyle%
\scriptscriptstyle{#1}}%
\!\int}
\def\XXint#1#2#3{{\setbox0=\hbox{$#1{#2#3}{%
\int}$ }
\vcenter{\hbox{$#2#3$ }}\kern-.6\wd0}}
\def\barint{\,\Xint -} 
\def\bariint{\barint_{} \kern-.4em \barint}
\def\bariiint{\bariint_{} \kern-.4em \barint}

\def\mean#1{\mathchoice%
          {\mathop{\kern 0.2em\vrule width 0.6em height 0.69678ex depth -0.58065ex
                  \kern -0.8em \intop}\nolimits_{\kern -0.4em#1}}%
          {\mathop{\kern 0.1em\vrule width 0.5em height 0.69678ex depth -0.60387ex
                  \kern -0.6em \intop}\nolimits_{#1}}%
          {\mathop{\kern 0.1em\vrule width 0.5em height 0.69678ex
              depth -0.60387ex
                  \kern -0.6em \intop}\nolimits_{#1}}%
          {\mathop{\kern 0.1em\vrule width 0.5em height 0.69678ex depth -0.60387ex
                  \kern -0.6em \intop}\nolimits_{#1}}}

\def\vintslides_#1{\mathchoice%
          {\mathop{\kern 0.1em\vrule width 0.5em height 0.697ex depth -0.581ex
                  \kern -0.6em \intop}\nolimits_{\kern -0.4em#1}}%
          {\mathop{\kern 0.1em\vrule width 0.3em height 0.697ex depth -0.604ex
                  \kern -0.4em \intop}\nolimits_{#1}}%
          {\mathop{\kern 0.1em\vrule width 0.3em height 0.697ex depth -0.604ex
                  \kern -0.4em \intop}\nolimits_{#1}}%
          {\mathop{\kern 0.1em\vrule width 0.3em height 0.697ex depth -0.604ex
                  \kern -0.4em \intop}\nolimits_{#1}}}

\newcommand{\aveint}[2]{\mathchoice%
          {\mathop{\kern 0.2em\vrule width 0.6em height 0.69678ex depth -0.58065ex
                  \kern -0.8em \intop}\nolimits_{\kern -0.45em#1}^{#2}}%
          {\mathop{\kern 0.1em\vrule width 0.5em height 0.69678ex depth -0.60387ex
                  \kern -0.6em \intop}\nolimits_{#1}^{#2}}%
          {\mathop{\kern 0.1em\vrule width 0.5em height 0.69678ex depth -0.60387ex
                  \kern -0.6em \intop}\nolimits_{#1}^{#2}}%
          {\mathop{\kern 0.1em\vrule width 0.5em height 0.69678ex depth -0.60387ex
                  \kern -0.6em \intop}\nolimits_{#1}^{#2}}}

\def\eqn#1$$#2$${\begin{equation}\label#1#2\end{equation}}
\def\charfn_#1{{\raise1.2pt\hbox{$\chi
_{\kern-1pt\lower3pt\hbox{{$\scriptstyle#1$}}}$}}}

\def\qq1{q_*}
\def\q2{q_{**}}

\def\osc{\operatorname{osc}}

\newdimen\vintbar
\def\vint{-\kern-\vintbar\int}

\def\A{\mathcal A}

\def\L{\mathcal L}
\def\K{\mathcal K}

\def\L{\mathcal L}

\def\0{\boldsymbol 0}

\newcommand{\R}{\mathbb R}

\newtoks\by
\newtoks\paper
\newtoks\book
\newtoks\jour
\newtoks\yr
\newtoks\pages
\newtoks\vol
\newtoks\publ

\def\name[#1, #2]{#1 #2}
\def\ota{{\hbox{\bf ???}}}
\def\cLear{\by=\ota\paper=\ota\book=\ota\jour=\ota\yr=\ota
\pages=\ota\vol=\ota\publ=\ota}
\def\endpaper{\the\by, \textit{\the\paper},
{\the\jour} \textbf{\the\vol} (\the\yr), \the\pages.\cLear}
\def\endbook{\the\by, \textit{\the\book},
\the\publ, \the\yr.\cLear}
\def\endpap{\the\by, \textit{\the\paper}, \the\jour.\cLear}
\def\endproc{\the\by, \textit{\the\paper}, \the\book, \the\publ,
\the\yr, \the\pages.\cLear}

\renewcommand{\d}{\, \mathrm{d}} 

\begin{document}
\title[Potential theory for operators of Kolmogorov type]{Potential theory for a class of strongly degenerate parabolic operators of Kolmogorov type\\ with rough coefficients}

\address{Malte Litsg{\aa}rd \\Department of Mathematics, Uppsala University\\
S-751 06 Uppsala, Sweden}
\email{malte.litsgard@math.uu.se}

\address{Kaj Nystr\"{o}m\\Department of Mathematics, Uppsala University\\
S-751 06 Uppsala, Sweden}
\email{kaj.nystrom@math.uu.se}

\thanks{K.N was partially supported by grant  2017-03805 from the Swedish research council (VR)}

\author{Malte Litsg{\aa}rd and Kaj Nystr{\"o}m}
\maketitle
\begin{abstract}
\noindent \medskip
 In this paper we develop a potential theory for strongly degenerate parabolic operators of the form
              \begin{eqnarray*}
   \L:=\nabla_X\cdot(A(X,Y,t)\nabla_X)+X\cdot\nabla_{Y}-\partial_t,
    \end{eqnarray*}
    in unbounded domains of the form
         \begin{eqnarray*}
 \Omega=\{(X,Y,t)=(x,x_{m},y,y_{m},t)\in\mathbb R^{m-1}\times\mathbb R\times\mathbb R^{m-1}\times\mathbb R\times\mathbb R\mid x_m>\psi(x,y,y_m,t)\},
    \end{eqnarray*}
     where $\psi$  is assumed to satisfy a
    uniform Lipschitz condition adapted to the dilation structure and the (non-Euclidean) Lie group underlying
    the operator $\L$. Concerning $A=A(X,Y,t)$ we assume that $A$ is bounded, measurable, symmetric and uniformly elliptic (as a matrix in $\mathbb R^{m}$). Beyond the solvability of the Dirichlet problem and other fundamental properties our results include scale and translation invariant boundary comparison principles,
boundary Harnack inequalities and doubling properties of associated parabolic measures.  All of our estimates are translation- and scale-invariant with constants only depending on the constants defining the boundedness and ellipticity of $A$ and the Lipschitz constant of $\psi$. Our results represent a version, for operators of Kolmogorov type with bounded, measurable coefficients, of the
by now classical results of Fabes and Safonov, any several others,  concerning boundary estimates for uniformly parabolic equations in (time-dependent) Lipschitz type domains. \\

\noindent

\medskip

\end{abstract}

    \setcounter{equation}{0} \setcounter{theorem}{0}
\section{Introduction}
The operator
    $$\K:=\nabla_X\cdot \nabla_X+X\cdot\nabla_Y-\partial_t$$
    in $\mathbb R^{N+1}$, $N=2m$, $m\geq 1$, equipped with coordinates $(X,Y,t):=(x_1,...,x_{m},y_1,...,y_{m},t)\in \mathbb R^{m}\times\mathbb R^{m}\times\mathbb R$,  was introduced and studied by
Kolmogorov in a famous note published in 1934 in Annals of Mathematics, see \cite{K}. Kolmogorov noted that $\K$ is an example of a degenerate parabolic
operator having strong regularity properties and he proved that $\K$ has a fundamental solution which is smooth off its diagonal.
As a consequence,
\begin{eqnarray}\label{uu3}
  \K u = f \in C^\infty \quad \Rightarrow \quad u \in C^\infty,
\end{eqnarray}
for every distributional solution of $\K u=f$. Today the property in \eqref{uu3} is stated
\begin{eqnarray}\label{uu2}
\mbox{$\K$ is hypoelliptic}.
\end{eqnarray}
As can be read in the introduction of H{\"o}rmander's monumental paper on the hypoellipticity of operators published in Acta Mathematica in 1967, see \cite{Hormander}, the operator studied by Kolmogorov served as an important model case for H{\"o}rmander when he developed  his theory. Today the Kolmogorov operator, and more general operators of Kolmogorov-Fokker-Planck type with variable coefficients, play central roles in many applications in analysis, physics and finance.

Kolmogorov was originally motivated by statistical physics and he studied $\K$ in the context of stochastic processes. Indeed, the fundamental
solution associated to $\K$ describes the density of the stochastic process
$(X_t,Y_t)$ which solves the Langevin system
\begin{equation}\label{e-langevin}
\begin{cases}
    \d X_t &= \sqrt{2}\d W_t, \\
    \d Y_t &= X_t \d t,
\end{cases}
\end{equation}
where $W_t$ is a $m$-dimensional Wiener process. The system in \eqref{e-langevin} describes the density of a system
with $2m$ degrees of freedom. Given $Z:=(X,Y)\in\mathbb R^{2m}$, $X=(x_1,...,x_m)$ and $Y=(y_{1},...,y_{m})$ are,
respectively, the velocity and the position of the system.

Kinetic theory is concerned with the evolution of a particle
distribution
$$f(X,Y,t):D\times\tilde D\times \mathbb R_+\to\mathbb R,\ D,\ \tilde D\subset \mathbb R^m,$$
subject to geometric restrictions and models for the interactions and collisions between particles.  Generally, assuming no external forces, the evolution of the particle density is described
by the Boltzmann equation
\begin{eqnarray}\label{e-kolm-ndcol}
 \partial_tf+X\cdot\nabla_Yf=Q(f, f).
    \end{eqnarray}
    The left-hand side in \eqref{e-kolm-ndcol} describes the evolution of $f$ under the action of transport, with the free
streaming operator. The right-hand side describes elastic collisions through the nonlinear Boltzmann collision operator.  The Boltzmann equation is an integro-
(partial)-differential equation with non-local operator in the kinetic variable $X$.  The
Boltzmann equation is a fundamental equation in kinetic theory in the sense that it has been derived rigorously, at least in some settings, from microscopic first
principles.

In the case of so called Coulomb interactions the Boltzmann collision operator is
ill-defined and Landau proposed an alternative operator for these  interactions: this operator is now called the Landau or the Landau-Coulomb operator. The operator can be stated as
\begin{eqnarray}\label{e-kolm-ndc1}
 \partial_tf+X\cdot\nabla_Yf=\nabla_X\cdot(A(f)\nabla_Xf+B(f)f),
    \end{eqnarray}
    where
    \begin{equation}\label{e-kolm-ndc2}
     \begin{split}
A(f)(X,Y,t)&:=a_{m,\gamma}\int\limits_{\mathbb R^m}\biggl (I-\frac {X'}{|X'|}\otimes\frac {X'}{|X'|}\biggr )|X'|^{\gamma+2}f(X-X',Y,t)\, \d X',\\
B(f)(X,Y,t)&:=b_{m,\gamma}\int\limits_{\mathbb R^m}|X'|^{\gamma}X'f(X-X',Y,t)\, \d X',
    \end{split}
\end{equation}
    and $\gamma\in[-m,0]$, $a_{m,\gamma}>0$, and obviously the collision term in \eqref{e-kolm-ndc1} has a divergence structure.  The operator in \eqref{e-kolm-ndc1} is a nonlinear drift-diffusion operator with coefficients given by convolution like averages of the unknown.  As mentioned above  the Landau equation is  considered fundamental because of its close link to
the Boltzmann equation for Coulomb interactions.

In the case of long-range interactions, the Boltzmann and Landau-Coulomb operators show local ellipticity provided
the solution enjoys some pointwise bounds on the associated hydrodynamic fields and the local entropy. Indeed, assuming, for all $(Y, t)\in \tilde D\times I$, that
     \begin{eqnarray*}\label{e-kolm-nd21}
     M_1\leq \int_{\mathbb R^m}f(X,Y,t)\, \d X\leq M_0&&\quad (\mbox{Local mass}),\\
    \frac 1 2\int_{\mathbb R^m}f(X,Y,t)|X|^2\, \d X\leq E_0&&\quad (\mbox{Local energy}),\\
       \int_{\mathbb R^m}f(X,Y,t)\ln f(X,Y,t)\, \d X\leq H_0&&\quad (\mbox{Local entropy}),
       \end{eqnarray*}
       one can prove that
           \begin{eqnarray*}\label{e-kolm-nd22}
           0<\lambda I\leq A(f)(X,Y,t)\leq \Lambda I,\ |B(f)(X,Y,t)|\leq \Lambda,
       \end{eqnarray*}
       for $(X,Y,t)\in D\times\tilde D\times I$, i.e., under these assumptions the Landau equation becomes locally uniformly elliptic. As a consequence, and as global well posedness for the Boltzmann
equation and the construction of solutions in the large is an outstanding open problem, the study of conditional regularity for the Boltzmann and Landau equations has become a way to make progress on the regularity issues for these equations. We refer to \cite{ MR554086, MR715658,d15, MR2765747,dvI, PLLCam, review} for more on the connections between Kolmogorov-Fokker-Planck equations, the Boltzmann and Landau equation, statistical physics and conditional regularity. Furthermore, we have learned a lot from the interesting survey of C. Mouhot \cite{Mou}.

       As outlined above, kinetic theory and the idea of conditional regularity is one way to motivate the study of the local regularity of weak solutions to the equation
              \begin{eqnarray}\label{e-kolm-ndfl}
      \nabla_X\cdot(A\nabla_Xu)+B\nabla_Xu+X\cdot\nabla_Yu-\partial_tu=0
    \end{eqnarray}
    assuming  $A$ is measurable, bounded and uniformly elliptic and the starting point for our analysis is the recent results concerning  the local regularity of weak solutions to the equation in \eqref{e-kolm-ndfl} established in \cite{GIMV}. In \cite{GIMV} the authors extended the  De Giorgi-Nash-Moser (DGNM) theory, which in its original form only considers elliptic
or parabolic equations in divergence form, to hypoelliptic equations with rough coefficients including the ones in \eqref{e-kolm-ndfl}. Their result is the correct scale- and translation-invariant estimates for local H{\"o}lder continuity and the Harnack inequality for weak solutions.

The results in \cite{GIMV} represent  an important achievement which paves the way for developments concerning
operators as in \eqref{e-kolm-nd} in several fields of analysis and in the theory of PDEs. In this paper we contribute to the understanding of the fine properties of the Dirichlet problems for a subclass of operators of the form stated in \eqref{e-kolm-ndfl}, we will for simplicity here only consider the case $B\equiv 0$,
    in appropriate domains $\Omega\subset\mathbb R^{N+1}$, and we note that in general there is a rich interplay between the operators considered, applications and geometry. Indeed, as discussed, the Kolmogorov operator, and the more general operators of Kolmogorov-Fokker-Planck type with variable coefficients considered in this paper, play central roles in many applications in analysis, physics and finance and depending on the application different model cases for the local geometry of $\Omega$ may be relevant:
    \begin{eqnarray}\label{dom-mod}
 (i)&&\{(X,Y,t)=(x,x_{m},y,y_{m},t)\in\mathbb R^{N+1} \mid \ x_m>\psi_1(x,Y,t)\},\notag\\
 (ii)&&\{(X,Y,t)=(x,x_{m},y,y_{m},t)\in\mathbb R^{N+1} \mid \ y_m>\psi_2(X,y,t)\},\\
 (iii)&&\{(X,Y,t)=(x,x_{m},y,y_{m},t)\in\mathbb R^{N+1} \mid \ t>\psi_3(X,Y)\}.\notag
\end{eqnarray}
In particular, in finance and in the context of option pricing and associated free boundary problems, case $(i)$ can be relevant. In kinetic theory it is relevant to restrict the particles to a container making case $(ii)$ relevant. Case $(iii)$ captures, as a special case, the initial value or Cauchy problem.

     To be precise, in this paper we  consider solutions to the equation $\L u=0$ in $\Omega$ where $\L$ is the operator
              \begin{eqnarray}\label{e-kolm-nd}
   \L:=\nabla_X\cdot(A(X,Y,t)\nabla_X)+X\cdot\nabla_Y-\partial_t,
    \end{eqnarray}
    in $\mathbb R^{N+1}$, $N=2m$, $m\geq 1$, $(X,Y,t):=(x_1,...,x_{m},y_1,...,y_{m},t)\in \mathbb R^{m}\times\mathbb R^{m}\times\mathbb R$. We assume that $$A=A(X,Y,t)=\{a_{i,j}(X,Y,t)\}_{i,j=1}^{m}$$ is a real-valued, $m\times m$-dimensional, symmetric matrix valued function satisfying
    \begin{eqnarray}\label{eq2}
      \kappa^{-1}|\xi|^2\leq \sum_{i,j=1}^{m}a_{i,j}(X,Y,t)\xi_i\xi_j,\quad \ \ |A(X,Y,t)\xi\cdot\zeta|\leq \kappa|\xi||\zeta|,
    \end{eqnarray}
    for some $\kappa\in [1,\infty)$, and for all $\xi,\zeta\in \mathbb R^{m}$, $(X,Y,t)\in\mathbb R^{N+1}$. We will refer to $\kappa$ as the  constant of $A$. Throughout the paper we will also assume that
    \begin{eqnarray}\label{comp}
    \quad A=A(X,Y,t)\equiv I_m\mbox{ outside some arbitrary but fixed compact subset of $\mathbb R^{N+1}$},
    \end{eqnarray}
    where $I_m$ denotes the $m\times m$ identity matrix, and that
        \begin{eqnarray}\label{eq2+}
    a_{i,j}\in C^\infty(\mathbb R^{N+1})
    \end{eqnarray}
    for all $i,j\in\{1,...,m\}$. The assumptions in \eqref{comp} and \eqref{eq2+} will only be used in a qualitative fashion.  The constants of our quantitative estimates will depend on $m$ and $\kappa$. The assumption \eqref{comp} is only imposed to simplify matters as we will work in unbounded domains. In particular, as our results are local by nature this assumption is a modest constraint. The assumption \eqref{eq2+} simplifies matters concerning the continuous Dirichlet problem. We note that the results in  \cite{GIMV} were derived  for operators including the ones in \eqref{e-kolm-nd} assuming \eqref{eq2} and also assuming, implicitly, \eqref{eq2+}. Naturally, this is not an issue in situations when uniqueness of weak solutions can be ensured. Concerning $\Omega$ we restrict ourselves to case \eqref{dom-mod} $(i)$ and unbounded domains $\Omega\subset\mathbb R^{N+1}$ of the form
\begin{eqnarray}\label{dom-}
 \Omega=\{(X,Y,t)=(x,x_{m},y,y_{m},t)\in\mathbb R^{N+1} \mid \ x_m>\psi(x,y,y_m,t)\},
    \end{eqnarray}
    imposing restrictions on $\psi$ of Lipschitz character. To generalize the program of this paper to the geometrical contexts \eqref{dom-mod} $(ii)$, $(iii)$, are interesting and relevant projects.

     Our main result is a potential theory for operators $\L$ as in \eqref{e-kolm-nd}, assuming only \eqref{eq2}, \eqref{comp} and \eqref{eq2+}, in unbounded  Lipschitz type domains as in \eqref{dom-}. Beyond the solvability of the Dirichlet problem and other fundamental properties our results include scale and translation invariant boundary comparison principles,
boundary Harnack inequalities and doubling properties of associated parabolic measures. All of our estimates are translation- and scale-invariant with constants depending only
$m$, $\kappa$, and the Lipschitz constant of $\psi$. These results are, up to a point, established allowing $A$ and $\psi$ to depend on all variables with $y_m$ included. However, the more refined results established are derived assuming in addition that $A$ as well as $\psi$ are independent of  the variable $y_m$. The reason for this is discussed in detail in the paper but this is a way to handle the rigidity in the underlying Harnack inequality, a rigidity that stems from the subelliptic nature of the operators considered. In the prototype case $A\equiv I_m$, i.e., in the case of the operator $\K$, the corresponding results were established in \cite{NP} but we  also refer to \cite{CNP1}, \cite{CNP2} and \cite{CNP3}, where a  number of important preliminary estimates concerning the boundary behavior of non-negative solutions
to equations of Kolmogorov-Fokker-Planck type in non-divergence form  in Lipschitz type domains were developed. Together, these papers  seem to represent the only previous result of their kind for operators of Kolmogorov type.

To put the importance of our results into perspective it is relevant to outline the progress on the corresponding problems in the case of uniformly parabolic equations in $\mathbb R^{m+1}$, i.e., in the case when all dependence on the variable $Y$ is removed in \eqref{e-kolm-nd}, leaving us with the operator
              \begin{eqnarray}
\nabla_X\cdot(A(X,t)\nabla_X)-\partial_t.
    \end{eqnarray}
   In this case and in the case of (time-dependent) Lipschitz type domains,  scale and translation invariant boundary comparison principles,
boundary Harnack inequalities and doubling properties of associated parabolic measures  were settled in a number of fundamental papers including \cite{FS}, \cite{FSY}, \cite{SY}, \cite{FGS} and \cite{N1}. Subsequently, this type of results have found their applications in several important fields of analysis including the analysis of free boundary problems, see \cite{C1}, \cite{C2} and \cite{ACS} for instance.


     \setcounter{equation}{0} \setcounter{theorem}{0}
    \section{Preliminaries}
\subsection{Group law and metric}  Throughout the paper we will use the notation $(Z,t)=(X,Y,t)=(x,x_m,y,y_m,t)$ and $(z,t)=(x,y,t)$.  The natural family of dilations for $\L$, $(\delta_r)_{r>0}$, on $\R^{N+1}$,
is defined by
\begin{equation}\label{dil.alpha.i}
 \delta_r (X,Y,t) =(r X, r^3 Y,r^2 t),
\end{equation}
for $(X,Y,t) \in \R^{N +1}$,  $r>0$.  Our class of operators  is closed under the  group law
\begin{equation}\label{e70}
 (\tilde Z,\tilde t)\circ (Z,t)=(\tilde X,\tilde Y,\tilde t)\circ (X, Y,t)=(\tilde X+X,\tilde Y+Y-t\tilde X,\tilde t+t),
\end{equation}
where $(Z,t),\ (\tilde Z,\tilde t)\in \R^{N+1}$. Note that
\begin{equation}\label{e70+}
(Z,t)^{-1}=(X,Y,t)^{-1}=(-X,-Y-tX,-t),
\end{equation}
and hence
\begin{equation}\label{e70++}
 (\tilde Z,\tilde t)^{-1}\circ (Z,t)=(\tilde  X,\tilde  Y,\tilde  t)^{-1}\circ (X,Y,t)=(X-\tilde  X,Y-\tilde  Y+(
t-\tilde  t)\tilde  X,t-\tilde  t),
\end{equation}
whenever $(Z,t),\ (\tilde Z,\tilde t)\in \R^{N+1}$.  Given $(Z,t)=(X,Y,t)\in \R^{N+1}$ we let \begin{equation}\label{kolnormint}
  \|(Z, t)\|= \|(X,Y, t)\|:=|(X,Y)|\!+|t|^{\frac{1}{2}},\ |(X,Y)|=\big|X\big|+\big|Y\big|^{1/3}.
\end{equation}
We recall that there exists a positive constant ${c}=c(m)$ such that
\begin{eqnarray}\label{e-ps.tr.in}
 \|(Z,t)^{-1}\|\le {c}  \| (Z,t) \|,\ \|(Z,t)\circ (\tilde Z,\tilde t)\| \le  {c}  (\| (Z,t) \| + \| (\tilde Z,\tilde t)
\|),
\end{eqnarray}
whenever $(Z,t),(\tilde Z,\tilde t)\in \R^{N+1}$. Using \eqref{e-ps.tr.in} it  follows immediately that
\begin{equation} \label{e-triangularap}
    \|(\tilde Z,\tilde t)^{-1}\circ (Z,t)\|\le c \, \|(Z,t)^{-1}\circ (\tilde Z,\tilde t)\|,
\end{equation}
whenever $(Z,t),(\tilde Z,\tilde t)\in \R^{N+1}$. We define
\begin{equation}\label{e-ps.distint}
    d((Z,t),(\tilde Z,\tilde t)):=\frac 1 2\bigl( \|(\tilde Z,\tilde t)^{-1}\circ (Z,t)\|+\|(Z,t)^{-1}\circ (\tilde Z,\tilde t)\|\bigr).
\end{equation}
Using \eqref{e-triangularap} it follows that
\begin{equation}\label{e-ps.dist}
\begin{split}
       c^{-1}\|(\tilde Z,\tilde t)^{-1}\circ (Z,t)\|\leq\ &d((Z,t),(\tilde Z,\tilde t))\leq c \|(\tilde Z,\tilde t)^{-1}\circ (Z,t)\|, \\
   c^{-1}\|(Z,t)^{-1}\circ (\tilde Z,\tilde t)\| \leq\ &d((Z,t),(\tilde Z,\tilde t)) \leq c\|(Z,t)^{-1}\circ (\tilde Z,\tilde t)\|
\end{split}
\end{equation}
with constants of comparison independent of $(Z,t),(\tilde Z,\tilde t)\in \R^{N+1}$. Again using \eqref{e-ps.tr.in} we also see that
\begin{equation} \label{e-triangular}
    d((Z,t),(\tilde Z,\tilde t))\le {c} \bigl(d((Z,t),(\hat Z,\hat t))+d((\hat Z,\hat
t),(\tilde Z,\tilde t))\bigr ),
\end{equation}
whenever $(Z,t),(\hat Z,\hat t),(\tilde Z,\tilde t)\in \R^{N+1}$, and hence $d$ is a symmetric quasi-distance. Based on $d$ we introduce the balls
\begin{equation}\label{e-BKint}
    \mathcal{B}_r(Z,t):= \{ (\tilde Z,\tilde t) \in\mathbb R^{N+1} \mid d((\tilde Z,\tilde t),(Z,t)) <
r\},
\end{equation}
for $(Z,t)\in \R^{N+1}$ and $r>0$.  The measure of the ball $\mathcal{B}_r(Z,t)$, $|\mathcal{B}_r(Z,t)|$, satisfies
\[
c^{-1}r^{\bf q}\leq |\mathcal{B}_r(Z,t)| \leq cr^{\bf q},\ {\bf q}:=4m+2,
\]
independent of $(Z,t)$.

\subsection{Geometry}
A function $\psi(x,y,y_m,t):\mathbb R^{m-1}\times \mathbb R^{m-1}\times\mathbb R\times\mathbb R\to\mathbb R$ satisfying
\begin{equation}\label{def:LipK}
    |\psi(z,y_m,t)-\psi(\tilde z, \tilde y_m,\tilde t)| \leq M(\|(\tilde z,\tilde t)^{-1}\circ (z,t)\| + |y_m-\tilde y_m + (t-\tilde t)\psi(\tilde z,\tilde y_m,\tilde t)|^{1/3}),
\end{equation}
for some constant $M\in (0,\infty)$, will be referred to as a Lipschitz function. Given a Lipschitz  function $\psi$ we say that
\begin{equation}\label{def:Omega}
    \Omega := \lbrace (Z,t)=(x,x_m,y,y_m,t)\in\R^{N+1}\mid x_m >\psi(x,y,y_m,t)\rbrace
\end{equation}
is an unbounded Lipschitz  domain with constant $M$, or  simply a Lipschitz  domain with constant $M$.

Several of our scale- and translation-invariant estimates will be formulated using certain reference points which we now introduce. Given $\rho>0$ and $\Lambda>0$ we let
\begin{equation}\label{pointsref2}
\begin{split}
A_{\rho,\Lambda}^+&:= \left(0,\Lambda\rho,0,-\tfrac 2 3\Lambda\rho^3,\rho^2\right)\in\mathbb R^{m-1}\times\mathbb
R\times\mathbb R^{m-1}\times\mathbb R
\times\mathbb R,\\
A_{\rho,\Lambda}&:=\left(0,\Lambda\rho,0,0,0\right)\in\mathbb R^{m-1}\times\mathbb R\times\mathbb R^{m-1}\times\mathbb
R\times\mathbb R,\\
A_{\rho,\Lambda}^-&:=\left(0,\Lambda\rho,0,\tfrac 2 3\Lambda\rho^3,-\rho^2\right)\in\mathbb R^{m-1}\times\mathbb
R \times\mathbb R^{m-1}\times\mathbb R
\times\mathbb R.
\end{split}
\end{equation}
Given $(Z_0,t_0)\in\mathbb R^{N+1}$ we let $$A_{\rho,\Lambda}^\pm(Z_0,t_0):=(Z_0,t_0)\circ A_{\rho,\Lambda}^\pm,\quad A_{\rho,\Lambda}(Z_0,t_0):=(Z_0,t_0)\circ A_{\rho,\Lambda}.$$

\subsection{Weak solutions} Consider $U_X\times U_Y\times J\subset\mathbb R^{N+1}$ with $U_X\subset\mathbb R^{m}$, $U_Y\subset\mathbb R^{m}$ being bounded domains, i.e, open, connected and bounded sets, and $J=(a,b)$ with $-\infty<a<b<\infty$. Then
    $u$ is said to be a weak solution to the equation
                  \begin{eqnarray}\label{e-kolm-nd-}
   \L u=\nabla_X\cdot(A(X,Y,t)\nabla_Xu)+X\cdot\nabla_Yu-\partial_tu=0,
    \end{eqnarray}
    in $U_X\times U_Y\times J\subset\mathbb R^{N+1}$ if
                      \begin{eqnarray}\label{weak1}
u\in L_{Y,t}^2(U_Y\times J,H_X^1(U_X)),
    \end{eqnarray}
    and
                          \begin{eqnarray}\label{weak2}
                          -X\cdot\nabla_Yu+\partial_tu\in  L_{Y,t}^2(U_Y\times J,H_X^{-1}(U_X)),
    \end{eqnarray}
    and if $\L u=0$ in the sense of distributions, i.e.,
\begin{eqnarray}\label{weak3}
 \iiint_{}\ \bigl(A(X,Y,t)\nabla_Xu\cdot \nabla_X\phi+(X\cdot \nabla_Y\phi)u-u\partial_t\phi\bigr )\, \d X \d Y \d t=0,
\end{eqnarray}
 whenever $\phi\in C_0^\infty(U_X\times U_Y\times J)$. Similarly we say that $u$ is a weak supersolution, $\L u\leq 0$ for short, if for all $\phi\in C_0^\infty(U_X\times U_Y\times J)$ such that $\phi\geq 0$, we have
    \begin{eqnarray}\label{superweak}
 \iiint_{}\ \bigl(-A(X,Y,t)\nabla_Xu\cdot \nabla_X\phi-(X\cdot \nabla_Y\phi)u+u\partial_t\phi\bigr )\, \d X \d Y \d t \leq 0.
\end{eqnarray}
Further, $u$ is a weak subsolution if
  $-u$ is a weak supersolution.

We say that $u$ is a weak solution
to the equation $\L u=0$ in $\Omega$ if $u$ is a weak solution to $\L u=0$ in $U_X\times U_Y\times J\subset\mathbb R^{N+1}$, where $U_X\subset\mathbb R^{m}$, $U_Y\subset\mathbb R^{m}$ are  bounded domains, and $J=(a,b)$ with $-\infty<a<b<\infty$, whenever $U_X\times U_Y\times J$ is compactly contained in $\Omega$. Weak super- and subsolutions are defined analogously.

\section{Statement of main results}
We first prove the following three theorems giving the solvability of the continuous Dirichlet problem, a  H{\"o}lder continuity estimate up to the boundary and an estimate usually referred to as the Carleson estimate. Our standing assumptions concerning $A$ is that $A$ satisfies \eqref{eq2} with constant $\kappa$ and that \eqref{comp} and \eqref{eq2+} hold.

\begin{theorem}\label{thm:dp} Let $\Omega\subset\mathbb R^{N+1}$ be a Lipschitz domain with constant $M$. Given $\varphi\in
C_0(\partial\Omega)$, there exists
a  unique  weak solution  $u=u_\varphi$, $u\in C(\bar \Omega)$, to the Dirichlet problem
\begin{equation} \label{e-bvpuu}
\begin{cases}
	\L u = 0  &\text{in} \ \Omega, \\
      u = \varphi  & \text{on} \ \partial \Omega.
\end{cases}
\end{equation}
Furthermore, there exists, for every $(Z, t)=(X,Y,t)\in \Omega$, a unique probability
measure  $\omega(Z,t,\cdot)$ on $\partial\Omega$ such that
\begin{eqnarray}  \label{1.1xxuu}
u(Z,t)=\iint_{\partial\Omega}\varphi(\tilde Z,\tilde t)\d \omega(Z,t,\tilde Z,\tilde t).
\end{eqnarray}
The measure $\omega(Z,t,E)$ is referred to as the parabolic measure associated to $\L$ in $\Omega$ and at $(Z, t)\in \Omega$ and of $E\subset\partial\Omega$.
\end{theorem}

\begin{theorem}\label{lem4.5-Kyoto1}
Let $\Omega\subset\mathbb R^{N+1}$ be a Lipschitz domain with constant $M$.
Let $(Z_0,t_0)\in\partial\Omega$ and $r>0$. Let $u$ be a weak  solution of $\L u=0$ in $\Omega\cap \mathcal{B}_{2r}(Z_0,t_0) $, vanishing continuously on  $\partial\Omega\cap \mathcal{B}_{2r}(Z_0,t_0) $. Then,
there exists a constant $c=c(m,\kappa,M)$, $1\leq c<\infty$, and $\alpha
=\alpha (m,\kappa,M)\in (0,1)$, such that
\begin{equation}
|u(Z,t)|\leq c\biggl (\frac{d((Z,t),(Z_{0},t_{0}))}{r}\biggr )^{\alpha
}\sup_{\Omega\cap \mathcal{B}_{2r}(Z_0,t_0)}|u|
\end{equation}%
whenever $(Z,t)\in \Omega\cap \mathcal{B}_{r/c}(Z_0,t_0)$.
\end{theorem}

\begin{theorem}\label{thm:carleson} Let $\Omega\subset\mathbb R^{N+1}$ be a Lipschitz domain with constant $M$. There exist positive $\Lambda=\Lambda(m,M)$ and $c=c(m,\kappa,M)$, $1\leq c<\infty$, such that the following holds. Let $(Z_0,t_0)\in\partial\Omega$ and $r>0$.  Assume that  $u$ is a non-negative weak solution to $\L u=0$ in $\Omega\cap \mathcal{B}_{2r}(Z_0,t_0)$,
vanishing continuously on $\partial\Omega\cap \mathcal{B}_{2r}(Z_0,t_0)$. Then
\begin{equation*}
u(Z,t)\leq cu(A^+_{\rho,\Lambda}(Z_0,t_0))
\end{equation*}
whenever $(Z,t)\in \Omega\cap \mathcal{B}_{2\rho/c}(Z_0,t_0)$, $0<\rho<r/c$.
\end{theorem}

We emphasize that Theorem \ref{thm:dp}-Theorem \ref{thm:carleson}  are proven assuming only that $\Omega\subset\mathbb R^{N+1}$ is a Lipschitz domain with constant $M$, that
$A$ satisfies \eqref{eq2} with constant $\kappa$, and that \eqref{comp} and \eqref{eq2+} hold. The latter assumptions are only used qualitatively.  However, our next set of results are proven under the following additional structural assumptions:
\begin{equation}\label{struct}
\begin{split}
(i)& \quad A(X,Y,t)=A(x,x_m,y,y_m,t)=A(x,x_m,y,t),\\
(ii)& \quad \psi(x,y,y_m,t)=\psi(x,y,t),
\end{split}
\end{equation}
whenever $(x,x_{m},y,y_{m},t)\in\mathbb R^{N+1}$. That is, in the following both $A$ and $\psi$ are assumed to be independent of the variable $y_m$. Using this additional assumption we prove the following theorems.
\begin{theorem}\label{thm:back} Let $\Omega\subset\mathbb R^{N+1}$ be a Lipschitz domain with constant $M$ and assume in addition \eqref{struct}.  There exist positive $\Lambda=\Lambda(m,M)$ and $c=c(m,\kappa,M)$, $1\leq c<\infty$, such that the following holds. Let $(Z_0,t_0)\in\partial\Omega$ and $r>0$.  Assume that  $u$ is a non-negative solution to $\L u=0$ in $\Omega\cap \mathcal{B}_{2r}(Z_0,t_0)$,
vanishing continuously on $\partial\Omega\cap \mathcal{B}_{2r}(Z_0,t_0)$. Let $\rho_0=r/c$,
 \begin{eqnarray}
 m^+=u(A_{\rho_0,\Lambda}^+(Z_0,t_0)),\ m^-=u(A_{\rho_0,\Lambda}^-(Z_0,t_0)),
 \end{eqnarray}
and assume that $m^->0$. Then there exist constants
$c_1=c_1(m,\kappa,M)$,
$1\leq c_1<\infty$, $c_2=c_2(m,\kappa,M, m^+/m^-)$,
$1\leq c_2<\infty$,  such that if we let $\rho_1=\rho_0/c_1$, then
\begin{equation*}
u(Z,t)\leq c_2u(A_{\rho,\Lambda}(\tilde Z_0,\tilde t_0)),
\end{equation*}
whenever $(Z,t)\in \Omega\cap \mathcal{B}_{\rho/c_1}(\tilde Z_0,\tilde t_0)$, for some  $(\tilde Z_0,\tilde t_0)\in
\partial\Omega\cap \mathcal{B}_{\rho_1}(Z_0,t_0)$, and  $0<\rho<\rho_1$.
\end{theorem}

\begin{theorem}\label{thm:quotients}
Let $\Omega\subset\mathbb R^{N+1}$ be a Lipschitz domain with constant $M$ and assume in addition \eqref{struct}. There exist positive $\Lambda=\Lambda(m,M)$ and $c=c(m,\kappa,M)$, $1\leq c<\infty$, such that the following holds. Let $(Z_0,t_0)\in\partial\Omega$ and
 $r>0$.  Assume that $u$ and $v$ are non-negative solutions to
 $\L u=0$ in $\Omega$,  vanishing continuously on $\partial\Omega\cap \mathcal{B}_{2r}(Z_0,t_0)$. Let $\rho_0=r/c$,
 \begin{equation}\label{singa1u1}
 \begin{split}
 m_1^+=& \, v(A_{\rho_0,\Lambda}^+(Z_0,t_0)),\ m_1^-= \,v(A_{\rho_0,\Lambda}^-(Z_0,t_0)),\\
 m_2^+=&\, u(A_{\rho_0,\Lambda}^+(Z_0,t_0)),\ m_2^-=\, u(A_{\rho_0,\Lambda}^-(Z_0,t_0)),
  \end{split}
 \end{equation}
and assume $m_1^->0$, $m_2^->0$. Then there exist constants
$c_1=c_1(m,\kappa,M)$, $$c_2=c_2(m,\kappa,M, m_1^+/m_1^-,m_2^+/m_2^-),$$
$1\leq
c_1,c_2<\infty$, $\sigma=\sigma(m,\kappa,M, m_1^+/m_1^-,m_2^+/m_2^-)$, $\sigma\in (0,1)$,   such that if we let $\rho_1=\rho_0/c_1$, then
$$\biggl |\frac { v ( Z,t ) }{ u ( Z,t ) }-\frac { v ( \tilde Z,\tilde t ) }{ u ( \tilde Z,\tilde t ) }\biggr |\leq c_2\biggl (\frac {d((Z,t),(\tilde Z,\tilde t))}{\rho}\biggr
)^\sigma\frac{v(A_{\rho,\Lambda}(\tilde Z_0,\tilde t_0))}{u(A_{\rho,\Lambda}(\tilde Z_0,\tilde t_0))},$$
whenever $(Z,t), (\tilde Z,\tilde t)\in \Omega\cap \mathcal{B}_{\rho/c_1}(\tilde Z_0,\tilde t_0)$, for some  $(\tilde Z_0,\tilde t_0)\in
\partial\Omega\cap \mathcal{B}_{\rho_1}(Z_0,t_0)$, and $0<\rho<\rho_1$.
\end{theorem}

\begin{theorem}\label{thm:doub} Let $\Omega\subset\mathbb R^{N+1}$ be a Lipschitz domain with constant $M$ and assume in addition \eqref{struct}.  Then there exist
 positive $\Lambda=\Lambda(m,M)$,  $c=c(m,\kappa,M)$,  $1\leq c<\infty$, such that the following is true. Let $(Z_0,t_0)\in\partial\Omega$, $0<\rho_0<\infty$. Then
\begin{eqnarray*}
\omega\bigl (A_{c\rho_0,\Lambda}^+(Z_0,t_0), \partial\Omega\cap\mathcal{B}_{2\rho}(\tilde Z_0,\tilde t_0)\bigr )\leq
c\omega\bigl (A_{c\rho_0,\Lambda}^+(Z_0,t_0), \partial\Omega\cap\mathcal{B}_{\rho}(\tilde Z_0,\tilde t_0)\bigr )
\end{eqnarray*}
for all balls $\mathcal{B}_{\rho}(\tilde Z_0,\tilde t_0)$, $(\tilde Z_0,\tilde t_0)\in\partial\Omega$, such that $\mathcal{B}_{\rho}(\tilde Z_0,\tilde t_0)\subset \mathcal{B}_{4\rho_0}(Z_0,t_0)$.
\end{theorem}

As mentioned before,  in the prototype case $A\equiv I_m$, i.e., in the case of the operator $\K$, Theorems \ref{thm:back}-\ref{thm:doub} are proved in \cite{NP}, and our
Theorems \ref{thm:back}-\ref{thm:doub} represent far reaching generalizations of the results in \cite{NP}. Here it is also fair to refer to \cite{CNP1}, \cite{CNP2} and \cite{CNP3} for some relevant estimates in the context of operators in non-divergence form. Compared to these previous results, the proofs presented here consistently have to take into account that in this paper the coefficients $\{a_{i,j}\}$ are, from a quantitative perspective, only assumed to be bounded, measurable, symmetric and uniformly elliptic.

The results in \cite{NP} are established assuming that $\Omega\subset\mathbb R^{N+1}$ is a Lipschitz domain with constant $M$ and that \eqref{struct} $(ii)$ holds (in \cite{NP} obviously \eqref{struct} $(i)$ is satisfied). In particular, for reasons that are explained in detail in \cite{NP} the results, including the translation
invariant doubling property of parabolic measure, were derived using the assumption that the defining function for $\Omega$ in \eqref{dom-}, $\psi$, was assumed to be independent of the variable $y_m$. This assumption gave the authors  a crucial additional
degree of freedom at their disposal when building Harnack chains to connect points: they could freely connect points in the $x_m$ variable,
taking  geometric restrictions into account, accepting that the path in the $y_m$ variable will most likely not end up
in `the right spot'. This possibility to conduct translations in the $y_m$ variable is also reason why we in Theorems \ref{thm:back}-\ref{thm:doub} assume \eqref{struct}.

\subsection{Structure of the paper} The rest of the paper is organized as follows. In Section \ref{sec2+} we establish crucial local estimates such as a maximum principle, energy estimates, interior regularity estimates, Harnack inequalities and related estimates. We also introduce, and state estimates of, the fundamental solution. Section \ref{sec5} is devoted to the proof of Theorem \ref{lem4.5-Kyoto1} which is based on estimates introduced in Section \ref{sec2+}. In Section \ref{sec6} we discuss the Dirichlet problem and prove Theorem \ref{thm:dp}. We also prove a maximum principle on unbounded Lipschitz domains. In Section \ref{sec7} we prove Theorem \ref{thm:carleson}.
In Section \ref{sec8} we introduce the Green function and discuss its relation to the parabolic measure associated to $\L$ and $\Omega$.
In Section \ref{sec9} we prove a weak comparison principle close to the boundary and discuss some consequences of it and the structural assumption \eqref{struct}. Note that up until and including Lemma \ref{compprinciple}, assumption \eqref{struct} is not used. Furthermore, assumption \eqref{struct} is only used explicitly in the proofs of Lemma \ref{T:backprel} and Lemma \ref{lem4.6+}, and implicitly in subsequent statements based on Lemma \ref{T:backprel}.
In Section \ref{sec10} we prove Theorem \ref{thm:back}. It is worth noting that Lemma \ref{T:backprel} (and therefore assumption \eqref{struct}) appears to be crucial to the proof. In Section \ref{sec11} we prove Theorem \ref{thm:quotients}. In Section \ref{sec12} we discuss properties of the parabolic measure associated to $\L$ and $\Omega$. In particular we prove Theorem \ref{thm:doub} and prove estimates of the related kernel function.\\

\noindent
{\bf Convention concerning constants.} Throughout the paper we will use following conventions. By $c$ we will denote a constant satisfying $1\leq c<\infty$, and $c$ will at most
depend on  $m$, $\kappa$ and $M$ unless otherwise stated.  We write that $c_1\lesssim c_2$ if $c_1/c_2$ is bounded from above by a positive constant depending at most on $m$, $\kappa$ and $M$. We write $c_1\approx c_2$ if $c_1\lesssim c_2$ and $c_2\lesssim c_1$.

\section{Basic principles and estimates}\label{sec2+}
Throughout the paper we will use the notation
\begin{align*}
    Q_r&:=\{(X,Y,t) \mid |x_i|<r,\ |y_i|<r^3,\ |t|<r^2\},\\
    Q_{M,r}&:=\{(X,Y,t) \mid |x_i|<r,\ i=1,...,m-1,\ |x_m|<4Mr,\ |y_i|<r^3,\ |t|<r^2\},
\end{align*}
where $r>0$, and given $(Z_0,t_0)=(X_0,Y_0,t_0)\in\mathbb R^{N+1}$ we let $Q_r(Z_0,t_0)=(Z_0,t_0)\circ Q_r$, $Q_{M,r}(Z_0,t_0)=(Z_0,t_0)\circ Q_{M,r}$. Note that given $M\geq 0$, there exists $c=c(m,M)\geq 1$  such that
\begin{equation}\label{intextball}
\mathcal{B}_{r/c}(Z_0,t_0)\subseteq Q_{M,r}\subseteq\mathcal{B}_{cr}(Z_0,t_0),
\end{equation}
for every $(Z_0,t_0)\in\R^{N+1}$ and $r>0$. Given $\psi$ as in \eqref{def:LipK}, $\Omega$ as in \eqref{def:Omega},  and $(Z_0,t_0)\in \partial\Omega$ we let
\begin{equation}\label{def.Omega.Delta.fr}
\begin{split}
    &\Omega_{r}(Z_0,t_0):=Q_{M,r}(Z_0,t_0)\cap\{(X,Y,t) \mid  \psi(x,Y,t)<x_m<4Mr+\psi(x_0,Y_0,t_0)\},\\
    &\Delta_{r}(Z_0,t_0):=Q_{r}(Z_0,t_0)\cap\{(X,Y,t) \mid  x_m=\psi(x,Y,t)\}.
\end{split}
\end{equation}
Note that  there exists $c=c(m,M)$, $1\leq c<\infty$, such that
\begin{eqnarray}\label{def.Omega.Delta.frinc}\Omega\cap\mathcal{B}_{r/c}(Z_0,t_0)\subset \Omega_{r}(Z_0,t_0)\subset \Omega\cap\mathcal{B}_{cr}(Z_0,t_0)
\end{eqnarray}
for all $(Z_0,t_0)\in\partial\Omega$ and $r>0$.

Throughout the paper we will, given $(Z_0,t_0)\in\partial\Omega$ and $r>0$, rather consistently use the notation $Q_r(Z_0,t_0)$, $\Omega_{r}(Z_0,t_0)$, and $\Delta_{r}(Z_0,t_0)$ instead of the notation used in statement of our main results:  $\mathcal{B}_r(Z_0,t_0)$, $\Omega\cap \mathcal{B}_r(Z_0,t_0)$, and $\partial\Omega\cap \mathcal{B}_r(Z_0,t_0)$.
Using \eqref{intextball} and \eqref{def.Omega.Delta.frinc} we note that we can readily move between the two different sets of notation.

\subsection{Comparison/maximum principle}
We first prove the weak maximum principle in $\Omega_{r}(Z_0,t_0)$ and we recall that we are qualitatively assuming  \eqref{eq2+}.
\begin{lemma}\label{maxprinciple}
Let $(Z_0,t_0)\in\R^{N+1}$, $r>0$, $\Omega_r:= \Omega_{r}(Z_0,t_0)$.  Let $u\in C^2(\Omega_r)\cap C(\overline{\Omega_r})$ be such that
\begin{equation}
    \begin{cases}
    \L u \geq 0 \quad \text{in } \Omega_{r},\\
    u \leq 0 \quad \text{on } \partial \Omega_{r}.
    \end{cases}
\end{equation}
Then $u\leq 0$ in $\Omega_{r}$.
\end{lemma}
\begin{proof} We can without loss of generality assume that  $(Z_0,t_0)=(0,0)$.  Using (\ref{eq2}) we see that $a_{m,m}(Z,t)\geq \kappa^{-1} > 0$ for all $(Z,t)\in \R^{N+1}$. Using
 \eqref{eq2+} we can rewrite the operator in non-divergence form,
$$\L = \sum_{i,j=1}^m{a_{i,j}(X,Y,t)\partial_{x_i x_j}} + \sum_{i=1}^mb_i(X,Y,t)\partial_{x_j}+X\cdot\nabla_Y-\partial_t,$$
where $$b_i(X,Y,t) := \sum_{j=1}^m{\partial_{x_i}a_{i,j}(X,Y,t)}\in L^\infty(Q_{M,2r}).$$
Assume  that
\begin{equation}\label{contra}\max_{\Omega_r} u > \max_{\partial\Omega_r} u.
\end{equation}
For $\epsilon>0$, put $u_\epsilon = u + \epsilon e^{Kx_m}$  for some constant $K>1$ to be chosen. Then
$u_\epsilon \rightarrow u$ uniformly on $\overline{\Omega_r}$ as $\epsilon\rightarrow 0$. Let $(\hat X_\epsilon, \hat Y_\epsilon, \hat t_\epsilon)$ be such that $u_\epsilon(\hat X_\epsilon, \hat Y_\epsilon, \hat t_\epsilon) = \max_{\overline{\Omega_r}}u_\epsilon$. By compactness we may pick a subsequence such that $(\hat X_\epsilon, \hat Y_\epsilon, \hat t_\epsilon)\rightarrow (\hat X,\hat Y,\hat t)$  and by using the uniform convergence we have
$$u(\hat X,\hat Y,\hat t)=\max_{\overline{\Omega_r}}u.$$
Using \eqref{contra} we have $(\hat X,\hat Y,\hat t)\in\Omega_r$ and, for $\epsilon$ small, $(\hat X_\epsilon,\hat Y_\epsilon,\hat t_\epsilon)\in\Omega_r$.
Note that $$\nabla_{X,Y,t} u_\epsilon(\hat X_\epsilon,\hat Y_\epsilon,\hat t_\epsilon)=0,$$ and that
$$\sum_{i,j=1}^m{a_{i,j}(\hat X_\epsilon,\hat Y_\epsilon,\hat t_\epsilon)\partial_{x_i x_j}u_\epsilon(\hat X_\epsilon,\hat Y_\epsilon,\hat t_\epsilon)} = \mathrm{trace}(A(\hat X_\epsilon,\hat Y_\epsilon,\hat t_\epsilon)\mathrm{H}_X(u_\epsilon(\hat X_\epsilon,\hat Y_\epsilon,\hat t_\epsilon)))\leq 0,$$
where $\mathrm{H}_X(u(X,Y,t))$ denotes the hessian matrix in the $X$-variable of $u$ at $(X,Y,t)$. Thus,
\begin{equation}\label{maxprinproof:eq1}
    \begin{split}
        \L u_\epsilon(\hat X_\epsilon,\hat Y_\epsilon,\hat t_\epsilon) &=
    \mathrm{trace}(A(\hat X_\epsilon,\hat Y_\epsilon,\hat t_\epsilon)\mathrm{H}_X(u_\epsilon(\hat X_\epsilon,\hat Y_\epsilon,\hat t_\epsilon)))\\
    &+ \sum_{i,j=1}^m{b_i}(\hat X_\epsilon,\hat Y_\epsilon,\hat t_\epsilon)\partial_{x_j}u_\epsilon(\hat X_\epsilon,\hat Y_\epsilon,\hat t_\epsilon)\\
    & \quad +\hat{X}_\epsilon\cdot\nabla_Yu_\epsilon(\hat X_\epsilon,\hat Y_\epsilon,\hat t_\epsilon) - \partial_tu_\epsilon(\hat X_\epsilon,\hat Y_\epsilon,\hat t_\epsilon)
    \leq 0.
    \end{split}
\end{equation}
On the other hand,
\begin{equation*}
    \begin{split}
        \L u_\epsilon(\hat X_\epsilon,\hat Y_\epsilon,\hat t_\epsilon) &= \L u(\hat X_\epsilon,\hat Y_\epsilon,\hat t_\epsilon) + \L\left(\epsilon e^{Kx_m}\right)\\
        &= \L u(\hat X_\epsilon,\hat Y_\epsilon,\hat t_\epsilon) + \epsilon K e^{Kx_m}(K a_{m,m}(\hat X_\epsilon,\hat Y_\epsilon,\hat t_\epsilon) + b_m(\hat X_\epsilon,\hat Y_\epsilon,\hat t_\epsilon))\\
        &\geq \epsilon K e^{Kx_m} (K/\kappa - \|b_m\|_{L^\infty(Q_{M,r})}) > 0,
    \end{split}
\end{equation*}
for $K$ large enough. This contradicts (\ref{maxprinproof:eq1}) and hence \eqref{contra}. Hence \eqref{contra} is false and the proof is complete.\end{proof}

\subsection{Energy estimates and local H{\"o}lder continuity}
We state and prove the following energy estimate.
\begin{lemma}\label{lem:energy}
Let $(Z_0,t_0)\in\R^{N+1}$, $r>0$. Let $u$ be a weak solution to the equation $\L u=0$ in $Q_{2r}(Z_0,t_0)$. Then
\begin{equation}
    \iiint_{Q_r(Z_0,t_0)}{|\nabla_X u|^2 \, \d Z\d t} \lesssim{r^{-2}}\iiint_{Q_{2r}(Z_0,t_0)}{|u|^2 \, \d Z\d t}.
\end{equation}
\end{lemma}
\begin{proof} The proof is standard and we note that  can without loss of generality assume that  $(Z_0,t_0)=(0,0)$ and we let $Q_r:=Q_r(0,0)$.  Let $\phi$ be a test function such that
$\phi\in C^\infty_0(Q_{2r})$, $\phi = 1$ in $Q_r$ and $r|\nabla_X\phi|+r^2|(X\cdot\nabla_Y-\partial_t) \phi|\lesssim 1$ in $Q_{2r}$. Using that $\L u=0$ in $Q_{2r}$ and using $\phi^2u$ as a test function, which is allowed due to that we are qualitatively assuming  \eqref{eq2+},  we obtain
\begin{equation}\label{eq:3}
    \iiint_{Q_{2r}}\ \bigl(A(X,Y,t)\nabla_Xu\cdot \nabla_X(\phi^2u)+(X\cdot \nabla_Y(\phi^2u))u-u\partial_t(\phi^2u)\bigr )\, \d Z\d t = 0.
\end{equation}
Manipulating this equality, using \eqref{eq2}, Cauchy-Schwarz and the properties of $\phi$, the lemma readily follows.
\end{proof}

The following two lemmas are proven in \cite{GIMV}.

\begin{lemma}\label{holder} Let $(Z_0,t_0)\in \mathbb R^{N+1}$, $r>0$.
Let $u$ be a weak solution to the equation $\L u=0$ in $Q_{2r}(Z_0,t_0)$. Then there exists  $\alpha\in (0,1)$,  depending only on  $m$ and $\kappa$, so that
$$|u(Z,t)-u(\tilde Z,\tilde t)|\lesssim \left (\frac {d((Z,t),(\tilde Z,\tilde t))}{r}\right )^\alpha\sup_{Q_{2r}(Z_0,t_0)}|u|$$
whenever $(Z,t), (\tilde Z,\tilde t)\in Q_{r}(Z_0,t_0)$.
\end{lemma}

 \begin{lemma}\label{lem:moserlocal} Assume that $\L u=0$ in $Q_{2r}=Q_{2r}(Z_0,t_0)\subset \mathbb R^{N+1}$.  Given $p\in (1,\infty)$ there exists a constant $c=c(m,\kappa,p)$,
$1\leq c<\infty$, such that
\begin{eqnarray}
\sup_{Q_r}\ |u|\leq c\biggl ( \bariiint_{Q_{2r}}\ |u|^p\, \d Z \d t\biggr )^{1/p}.
\end{eqnarray}
\end{lemma}

\subsection{Harnack's inequality and Harnack chains}\label{Harnack chains}
To state the Harnack inequality we introduce some further notation. We let
\begin{eqnarray}\label{pastcyl}
   \quad\quad Q_r^-:=\{(X,Y,t) \mid |x_i|<r,\ |y_i|<r^3,\ -r^2<t\leq 0\},\quad Q_r^-(Z_0,t_0)=(Z_0,t_0)\circ Q_r^-,
\end{eqnarray}
for $(Z_0,t_0)\in\R^{N+1}$. The following Harnack inequality is proved in \cite{GIMV}.
\begin{lemma}\label{harnack} {There exist constants $c=c(m,\kappa)>1$ and $\alpha, \beta, \gamma, \theta \in (0,1)$, with
$0 < \alpha < \beta < \gamma < \theta^2$, such that the following is true. Assume $u$ is a non-negative solution to
$\L u=0$ in $Q_r^-(Z_0,t_0)$ for some $r>0$, $(Z_{0},t_{0})\in\rnn$. Then,
$$
    \sup_{\widetilde Q^-_{r}(Z_0, t_0)} u \leq c\inf_{\widetilde Q^+_{r}(Z_0, t_0)} u,
$$
}
where
\begin{equation*}
\begin{split}
       \widetilde Q^+_{r}(Z_0, t_0) & = \big\{ (x,t) \in Q_{\theta r}^-(Z_0,t_0) \mid t_0 - \alpha r^2 \le t \le t_0
\big\},\\
   \widetilde Q^-_{r}(Z_0, t_0) & = \big\{ (x,t) \in Q_{\theta r}^-(Z_0,t_0) \mid t_0 - \gamma r^2 \le t \le t_0 -\beta
r^2  \big\}.
\end{split}
\end{equation*}
\end{lemma}

We remark that the constants $\alpha, \beta, \gamma, \theta$ appearing in the above lemma can not be chosen arbitrarily and this is in contrast to, for example, the case of uniformly parabolic equations.

\begin{definition}\label{def:admissible}
A path $\gamma: [0,T] \to \rnn$ is called
\emph{admissible} if it is absolutely continuous and satisfies
\begin{equation}\label{e-control.pb}
   \frac{d}{d\tau}\gamma(\tau)=\sum_{j=1}^m \omega_j(\tau)\partial_{x_j}(\gamma(\tau))+\lambda(\tau)\bigl(\sum_{k=1}^m{x_k\partial_{y_k}(\gamma(\tau))-\partial_t(\gamma(\tau))}\bigr),
   \,\text{for a.e.}\,\tau\in [0,T],
\end{equation}
where $\omega_j \in L^2([0,T])$, for $j=1,\dots, m$, and $\lambda$ is a non-negative measurable function. We say that
$\gamma$ connects $(Z,t)=(X,Y,t)\in\mathbb R^{N+1}$ to $(\tilde Z,\tilde t)=(\tilde X,\tilde Y,\tilde t)\in\mathbb
R^{N+1}$, $\tilde t<t$, if $\g(0)= (Z,t)$ and $\g(T)= (\tilde Z,\tilde t)$.
\end{definition}

\begin{definition}\label{propag}
Given a domain $\Omega\subset\mathbb R^{N+1}$, and a point $(Z,t)\in \Omega$, we let
$\mathcal{A}_{(Z,t)}=\mathcal{A}_{(Z,t)}(\Omega)$ denote the closure of the set
\begin{equation*}
\big\{(\tilde Z,\tilde t)\in\Omega \mid \text{there exists an admissible} \ \g: [0,T] \to \Omega, \
 \text{connecting} \  (Z,t) \ \text{to} \ (\tilde Z,\tilde t) \big\}.
\end{equation*}
We will refer to
$\A_{(Z,t)}(\Omega)$ as the \emph{propagation set of the point $(Z,t)$ with respect to $\Omega$}.
\end{definition}

\begin{definition}\label{harnackchain} Let $\Omega\subset\mathbb R^{N+1}$ be a domain. Let $(Z,t)$, $(\tilde Z,\tilde
t)\in \Omega $, $\tilde t<t$, be
given. Let
$\{r_j\}_{j= 1}^k$ be a finite sequence of real positive numbers
and let $\{(Z_j, t_j)\}_{j= 1}^k$ be a sequence of points such that $(Z_1, t_1) = (Z,t)$. Then
$\{\{(Z_j, t_j)\}_{j= 1}^k, \{r_j\}_{j= 1}^k\}$ is said to be a Harnack chain in $\Omega$, connecting $(Z,t)$ to
$(\tilde  Z,\tilde  t)$, if
\begin{equation}\label{e-incl}
\begin{split}
     (i)\quad& Q^-_{r_j}(Z_j, t_j) \subset \Omega,\mbox{ for every $j=1, \dots,k$},\\
(ii)\quad&(Z_{j+1}, t_{j+1}) \in \widetilde Q^-_{r_j}(Z_j, t_j),\mbox{ for every $j=1, \dots,k-1$},\\
(iii)\quad&(\tilde Z, \tilde  t)\in \widetilde Q^-_{r_k}(Z_k, t_k).
\end{split}
\end{equation}
\end{definition}

Let $u$ be a non-negative weak solution to $\L u = 0$ in $\Omega$. Assume that $\{\{(Z_j, t_j)\}_{j= 1}^k, \{r_j\}_{j=
1}^k\}$ is  a Harnack chain in $\Omega$, connecting
 $(\tilde Z,\tilde t)$ to
$(Z,t)$, and let $c$ be the constant appearing in Lemma \ref{harnack}. Then, using Lemma \ref{harnack}, we see
that
\begin{eqnarray}\label{e-incl+}
u(Z_{j+1}, t_{j+1}) \le c u (Z_{j}, t_{j}),\mbox{ for every $j=1, \dots,
k-1,$}
\end{eqnarray}
and hence,
\begin{eqnarray}\label{e-incl++}
u(\tilde Z,\tilde t) \le c u(Z_{k},t_{k})\le c^k u(Z,t).
\end{eqnarray}

 To use Lemma \ref{harnack} efficiently we will build Harnack chains using admissible paths.

\begin{lemma}\label{adda}
Let  $\g(\tau):[0, t-\tilde t]\to\rnn$, $\tilde t<t$, be an admissible path, starting at a point with time coordinate equal to $t$, with $\lambda=\lambda(\tau)\equiv 1$, and let $a,b$ be constants such that
$0\le a < b \le T$, $T:=t-\tilde t$. Then there exist positive constants $h$ and $\eta$, depending only on $m$, such that
\begin{equation}\label{e-lemma}
 \int_a^b\frac {||\omega(\tau)||^2}h \d \tau \le 1 \quad \Rightarrow \quad \gamma(b)\in  Q^-_{\eta r}(\g(a)),\mbox{ where }
 r=\sqrt{b-a}.
\end{equation}
Furthermore, the set $Q^-_{\eta r}(\g(a))\cap\lbrace t=T-b\rbrace$ is contained in $\mathrm{Int}\left(\mathcal{A}_{(\gamma(a)}(Q^-_{r})\right)$.
\end{lemma}
\begin{proof}
The first part of the lemma is a consequence of Lemma 2.2 in \cite{BP}. The second conclusion follows by arguing as in the proof of Proposition 3.2 in \cite{CNP3}.
\end{proof}

\begin{lemma}\label{akkaa-} Let $\Omega\subset\mathbb R^{N+1}$ be a domain. Let $\gamma$, $t$, $\tilde t$,  $\lambda$, $T$, $h$, $\eta$, be as in Lemma \ref{adda} and define  $\{\tau_j\}$ as follows. Let $\tau_0= 0$, and define
$\tau_{j+1}$, for $j\geq 0$, recursively as follows:
\begin{eqnarray*}
(i)&&\mbox{ if }\int_{\tau_j}^{t-\tilde t}
 \frac{||\omega(\tau)||^2}{h} \, \d \tau  > 1\mbox{ then } \tau_{j+1} =\inf  \Big\{ \sigma \in (\tau_j, t-\tilde t ] :
\int_{\tau_j}^{\sigma}
 \frac{||\omega(\tau)||^2}{h} \, \d \tau  > 1\Big\},\notag\\
 (ii)&&\mbox{ if }\int_{\tau_j}^{t-\tilde t}
 \frac{||\omega(\tau)||^2}{h} \, \d \tau \leq 1\mbox{ then } \tau_{j+1}:=t-\tilde t.
\end{eqnarray*}
Let $k$ be smallest index such that $\tau_{k+1}=t-\tilde t$. Define, based on $\{\tau_j\}_{j=0}^{k+1}$,
\begin{equation}\label{a1}
    r_j = \sqrt{\frac{\tau_{j+1} - \tau_j}{\eta^2}}, \qquad j=1, \ldots, k,
\end{equation}
and let $(Z_j, t_j) =\g(\t_j)$ for $j=1, \ldots, k$. Assume that
\begin{eqnarray}\label{e-incl--}
 \g(\tau):[0, t-\tilde t ]\to \Omega,\mbox{ and } Q^-_{r_j}(Z_j, t_j) \subset \Omega,
\end{eqnarray}
for every $j=1, \dots,k$. Then there exists a constant $c=c(m,\kappa)$, $1\leq c<\infty$, such that if
$u$ is a non-negative weak solution to $\L u=0$ in $\Omega$, then
\begin{eqnarray}
u(\tilde Z,\tilde t)\leq c^{\bigl(1+\frac 1 h\int_0^{t-\tilde t} {||\omega(\tau)||^2} \d \tau\bigr )}u(Z,t).
\end{eqnarray}
\end{lemma}
\begin{proof} This can be proved by proceeding as in the proof of Proposition 1.1 in \cite{BP} using  Lemma \ref{harnack} and Lemma \ref{adda}. Indeed, let $\{\{(Z_j, t_j)\}_{j= 1}^k, \{r_j\}_{j= 1}^k\}$ be as in the statement of the lemma. Using the assumption \eqref{e-incl--} and applying Lemma \ref{adda} then yields that $\{\{(Z_j, t_j)\}_{j= 1}^k, \{r_j\}_{j= 1}^k\}$ is a Harnack chain connecting $(\tilde Z,\tilde t)$ to $(Z,t)$ with
$$k\leq 1+\frac 1 h\int_0^{t-\tilde t} {||\omega(\tau)||^2} \d \tau.$$
This gives the inequality of the lemma. \end{proof}

Let
      \begin{equation}\label{def:matrixE}
      B:=\begin{pmatrix}
   0 & I_{m} \\
   0 & 0 \
 \end{pmatrix},\quad E(s)=\exp(-sB^\ast),
\end{equation}
for $s\in\mathbb R$, where $I_m$ and $0$ represents the identity matrix and the zero matrix in $\mathbb R^m$,
respectively. Furthermore, let
\begin{eqnarray}\label{def:matrixC} \mathcal{C}(t)&:=& \int_0^tE(s)\begin{pmatrix}
   I_m & 0 \\
   0 & 0 \
 \end{pmatrix}E^\ast(s) \d s=\begin{pmatrix}
   tI_m & -\frac{t^2} 2I_m \\
   -\frac{t^2} 2I_m & \frac{t^3} 3I_m \
 \end{pmatrix},
\end{eqnarray}
whenever $t\in\mathbb R$. Note that $\det\mathcal{C}(t)=t^{4m}/12$ and that
\begin{eqnarray}\label{def:matrixCinv} (\mathcal{C}(t))^{-1}= 12\begin{pmatrix}
   \frac {t^{-1}}{3}I_m & \frac {t^{-2}}{2}I_m \\
   \frac {t^{-2}}{2}I_m & t^{-3}I_m\
 \end{pmatrix}.
\end{eqnarray}

\begin{lemma}\label{akkaa} Let $\Omega\subset\mathbb R^{N+1}$ be a  domain.
Let $(Z,t)$, $(\tilde Z,\tilde t)\in \Omega $, $\tilde t<t$, be given. Consider the path $\g(\tau)=(\tilde\g(\tau), t-\t):[0, t-\tilde
t]\to\rnn$
where
\begin{equation}\label{e-gamma-t}
    \tilde\g(\t) = E(-\t) \left( Z + \mathcal{C}(\t) \mathcal{C}^{-1}(t-\tilde t ) (E(t-\tilde t  ) \tilde  Z - Z) \right).
\end{equation}
Then $\g(0)=(Z,\tau)$, $\g(t-\tilde t ) =(\tilde Z,\tilde\tau)$ and $(\tilde\g(\tau), t-\t)$ is an {admissible} path. Moreover, the path satisfies \eqref{e-control.pb} with
\begin{equation}\label{e-cost-3+}
    \omega(\tau)=( \omega_1(\tau),.., \omega_m(\tau))= E(\tau)^\ast
    \mathcal{C}^{-1}(t-\tilde t ) (E(t-\tilde t  )\tilde  Z - Z).
\end{equation}
Let  $h$, $\beta$, $\{\tau_j\}$, $\{r_j\}$ and $\{(Z_j,\tau_j)\}$ be as in Lemma \ref{akkaa-}.  Assume that
\begin{eqnarray}\label{e-incl---}
 \g(\tau):[0, t-\tilde t ]\to \Omega,\mbox{ and } Q^-_{r_j}(z_j, t_j) \subset \Omega,
\end{eqnarray}
for every $j=1, \dots,k$. Then there exists a constant $c=c(m,\kappa)$, $1\leq c<\infty$, such that if
$u$ is a non-negative weak solution to $\L u=0$ in $\Omega$, then
\begin{eqnarray}
u(\tilde Z,\tilde t)\leq c^{\bigl(1+\frac 1 h\langle \mathcal{C}^{-1}(t-\tilde t  ) (Z-E(t-\tilde t  )\tilde Z),
Z-E(t-\tilde t  )\tilde Z\rangle\bigr )}u(Z,t).
\end{eqnarray}
\end{lemma}
\begin{proof}
That $\gamma(\tau)$ is an admissible curve satisfying \eqref{e-gamma-t} is shown by direct computation. Using the assumption of the lemma we see that we can apply Lemma \ref{akkaa-}. However, in this case, by a direct computation, we see that
\[
\int_0^{t-\tilde t}{||\omega(\tau)||^2 \d\tau}=
\langle \mathcal{C}^{-1}(t-\tilde t  ) (Z-E(t-\tilde t  )\tilde Z),
Z-E(t-\tilde t  )\tilde Z\rangle,
\]
which proves the result. \end{proof}

\begin{remark}
Just to be clear, to apply Lemma \ref{akkaa}, the critical  assumption to be verified in a specific application is \eqref{e-incl---}.
\end{remark}

Based on the notion of propagation sets the following (general) geometric version of the Harnack inequality can also be proved using Lemma \ref{harnack} and Lemma \ref{akkaa-}.
\begin{lemma} \label{t-1} Let $\Omega\subset\mathbb R^{N+1}$ be a domain and let $(Z_0,t_0) \in \Omega$. Let $K$ be a
compact set contained in the interior of $\mathcal{A}_{(Z_0,t_0)}(\Omega)$. Then there exists a positive constant $c_K=c(m,\kappa,M,K)$, such that
\begin{equation*}
    \sup_K u \le c_K \, u(Z_0,t_0),
\end{equation*}
for every non-negative weak solution $u$ of $\L u = 0$ in $\Omega$.
\end{lemma}
\begin{proof} The lemma can be proved using the same reasoning as in the proof of Lemma \ref{akkaa-}; we refer to \cite{AEP} for details.
\end{proof}

In light of Lemma \ref{t-1} we could pose the following alternative definition of a Harnack chain: $\lbrace \lbrace(Z_j,t_j)\rbrace_{j=1}^k,\lbrace r_j \rbrace_{j=1}^k \rbrace$ is a Harnack chain in $\Omega$, connecting $(Z,t)$ to $(\tilde Z,\tilde t)$, if

\begin{equation}\label{e-incl2}
\begin{split}
     (i)\quad& Q^-_{r_j}(Z_j, t_j) \subset \Omega,\mbox{ for every $j=1, \dots,k$},\\
(ii)\quad&(Z_{j+1}, t_{j+1}) \in K_j \subset \mathrm{Int}(\A_{(Z_j, t_j)}(\Omega)),\mbox{ for every $j=1, \dots,k-1$},\\
(iii)\quad&(\tilde Z, \tilde  t)\in K_k \subset \mathrm{Int}(\A_{(Z_k, t_k)}(\Omega)),
\end{split}
\end{equation}
where $K_j$ is a compact set, for each $j=1,\cdots,k$, and $\mathrm{Int}(\A_{(Z_j, t_j)}(\Omega))$ denotes the interior of the set $\A_{(Z_j, t_j)}(\Omega)$. Note that similarly as before, if $u$ is a non-negative weak solution to $\L u=0$ in $\Omega$ and if $\lbrace \lbrace(Z_j,t_j)\rbrace_{j=1}^k,\lbrace r_j \rbrace_{j=1}^k \rbrace$ is a Harnack chain in the sense above, then by Lemma \ref{t-1}
\begin{eqnarray*}
u(Z_{j+1}, t_{j+1}) \le c_{K_j} u (Z_{j}, t_{j}),\mbox{ for every $j=1, \dots,
k-1,$}
\end{eqnarray*}
and hence
\begin{eqnarray}\label{e-incl+++pd}
u(\tilde Z,\tilde t) \le c u(Z_{k},t_{k})\le c^k u(Z,t),
\end{eqnarray}
where $c:=\max\lbrace c_{K_j} \rbrace_{j=1}^k$.

\subsection{Admissible paths and cones}

\begin{lemma}\label{lem.conecond.0} Let $\Lambda$ be a positive constant. Define 
\begin{equation} \label{e-xtilde}
z_\Lambda=\left(0,\Lambda,0,-\tfrac 2 3\Lambda\right)\in \R^{m-1}\times \R\times\R^{m-1}\times \R.
\end{equation}
 Then, the path $$[0,1] \to \R^{N+1},\: \tau \mapsto \gamma(\tau)=\delta_{1-\tau}(z_\Lambda, 1)$$ is admissible.
\end{lemma}
\begin{proof} Note that by definition
\begin{equation*}
 \gamma(\tau)= \left(0, (1-\tau)\Lambda,0,-\tfrac 2 3(1-\tau)^3\Lambda,(1-\tau)^2\right), \tau\in [0,1].
\end{equation*}
By a direct computation we see that
\begin{equation*}
 \frac{d}{d\tau} \gamma(\tau)=(0,-\Lambda,0,2(1-\tau)^2\Lambda,-2(1-\tau)),\ \tau\in [0,1].
\end{equation*}
In particular,
\begin{equation}\label{e-control.pb+++}
   \frac{d}{d\tau}\gamma(\tau)=
   \sum_{j=1}^m \omega_j(\tau)\partial_{x_j}(\gamma(\tau))+\lambda(\tau)\left(\sum_{k=1}^m{x_k\partial_{y_k}(\gamma(\tau))-\partial_t(\gamma(\tau))}\right),\: \tau\in
[0,1],
\end{equation}
where  $\omega_j\equiv 0$ for $j\in\{1,..,m-1\}$, $\omega_m=-\Lambda$ and $\lambda(\tau)=2(1-\tau)$.
\end{proof}

To continue, we recall the following cones, or cone like objects, introduced in \cite{CNP3}. Given $(Z_0,t_0) \in \rnn$, $\bar Z \in \R^N$,
$\bar t \in \R_+$, consider an open
neighborhood $U \subset \R^N$ of $\bar Z$, and let
\begin{equation} \label{e-tusk}
\begin{split}
 	Z^+_{\bar Z, \bar t, U} (Z_0,t_0) & = \big\{ (Z_0,t_0) \circ \delta_s (Z,\bar t) \mid Z \in U, \, 0 < s \le 1 \big\},\\
 	Z^-_{\bar Z, \bar t, U} (Z_0,t_0) & = \big\{ (Z_0,t_0) \circ \delta_s (Z,-\bar t) \mid Z \in U, \, 0 < s \le 1 \big\}.
\end{split}
\end{equation}
Given $\rho>0$ and
$\Lambda>0$, recall the points $A_{\rho,\Lambda}^+$, $A_{\rho,\Lambda}$, $A_{\rho,\Lambda}^-$, introduced in
\eqref{pointsref2}. In addition we  introduce
\begin{equation}\label{pointsref2a}
\begin{split}
    \tilde A_{\rho,\Lambda}^+&=(0,-\Lambda\rho,0,\frac 2 3\Lambda\rho^3,\rho^2),\\
    \tilde A_{\rho,\Lambda}^-&=(0,-\Lambda\rho,0,-\frac 2 3\Lambda\rho^3,-\rho^2),
\end{split}
\end{equation}
and
\begin{eqnarray}\label{pointsref2apa}
\tilde A_{\rho,\Lambda}^\pm(Z_0,t_0)=(Z_0,t_0)\circ \tilde A_{\rho,\Lambda}^\pm,
\end{eqnarray}
whenever $(Z_0,t_0)\in \mathbb R^{N+1}$. Let the points $z_{\rho,\Lambda}^\pm$, $\tilde z_{\rho,\Lambda}^\pm$  be defined
through the relations
\begin{equation}
        A_{\rho,\Lambda}^\pm = (z_{\rho,\Lambda}^\pm,\rho^2),\ \tilde A_{\rho,\Lambda}^\pm = (\tilde z_{\rho,\Lambda}^\pm,\rho^2).
\end{equation}
Given $\eta$, $0<\eta\ll 1$, $\Lambda$, and $\rho>0$ we let
$$U_{\rho,\eta,\Lambda}^\pm:=\mathcal{B}_{\eta \rho}((z_{\rho,\Lambda}^\pm,0))\cap\{(Z,t)\mid t=0\}.$$
Then, based on \eqref{e-tusk}, we define
\begin{equation}\label{conesha}
\begin{split}
    C_{\rho,\eta,\Lambda}^\pm(Z_0,t_0)&=Z^\pm_{A^\pm_{\rho,\Lambda}, U_{\rho,\eta,\Lambda}^\pm}(Z_0,t_0),\\
\tilde C_{\rho,\eta,\Lambda}^\pm(Z_0,t_0)&=Z^\pm_{\tilde A^\pm_{\rho,\Lambda}, U_{\rho,\eta,\Lambda}^\pm}(Z_0,t_0).
\end{split}
\end{equation}
The sets \begin{eqnarray}\label{coneshapa}
C_{\rho,\eta,\Lambda}^\pm(Z_0,t_0),\ \tilde C_{\rho,\eta,\Lambda}^\pm(Z_0,t_0),
\end{eqnarray}
will be referred to as cones with vertex at $(Z_0,t_0)$ as, for $\eta$ small, they represent cones around admissible paths passing through
$(Z_0,t_0)$ as well as the points $A_{\rho,\Lambda}^\pm(Z_0,t_0)$, and $\tilde A_{\rho,\Lambda}^\pm(Z_0,t_0)$.

\subsection{Cones in Lipschitz domains} The standing assumption in all lemmas stated in this subsection is that $\Omega\subset\mathbb R^{N+1}$ is a Lipschitz domain with
Lipschitz constant $M$, and that $\Omega_{r}=\Omega_{r}(\hat Z_0,\hat t_0)$, $\Delta_{r}=\Delta_{r}(\hat Z_0,\hat t_0)$ for some $(\hat Z_0,\hat t_0)\in\partial\Omega$ fixed, and for $r>0$.

\begin{lemma}\label{coneconditions} Consider $\Omega_{2r}$. There exist
 $\Lambda=\Lambda(m,M)$, $1\leq \Lambda<\infty$, and  $c_0=c_0(m,M)$,
$1\leq c_0<\infty$, such that the following is true. Let  $\rho_0=r/c_0$, consider $(Z_0,t_0)\in
\Delta_{\rho_0}$ and $0<\rho<\rho_0$. Then
there exists $\eta=\eta(m,M)$,
$0<\eta\ll 1$, such that if we introduce
$C_{\rho,2\eta,\Lambda}^\pm(Z_0,t_0)$,  $\tilde C_{\rho,2\eta,\Lambda}^\pm(Z_0,t_0)$, as in \eqref{conesha}, then
\begin{equation}\label{conc}
\begin{split}
(i)&\quad C_{\rho,2\eta,\Lambda}^\pm(Z_0,t_0)\subset \Omega_{r}, \\
(ii)&\quad \tilde C_{\rho,2\eta,\Lambda}^\pm(Z_0,t_0)\subset \mathbb R^{N+1}\setminus \Omega_{r}.
\end{split}
\end{equation}
\end{lemma}
\begin{proof} This is a consequence of  Lemma 4.4 in \cite{CNP3}.\end{proof}

\begin{lemma}\label{coneconditions-} Consider $\Omega_{2r}$. There exist
 $\Lambda=\Lambda(m,M)$, $1\leq \Lambda<\infty$, and  $c_0=c_0(m,M)$,
$1\leq c_0<\infty$, such that the following is true. Let  $\rho_0=r/c_0$, consider $(Z_0,t_0)\in
\Delta_{\rho_0}$, $0<\rho<\rho_0$, and
let $A_{\rho,\Lambda}^\pm(Z_0,t_0)$, $\tilde A_{\rho,\Lambda}^\pm(Z_0,t_0)$,  be the reference points introduced. Then
\begin{equation}\label{akka-}
\begin{split}
    A_{\rho,\Lambda}^\pm(Z_0,t_0) &\in\Omega_{r},\\
    \tilde A_{\rho,\Lambda}^\pm(Z_0,t_0) &\in \mathbb R^{N+1}\setminus \Omega_{r},
\end{split}
\end{equation}
and
\begin{eqnarray}\label{akka}
(i)&&1\approx d(P_{\rho,\Lambda}(Z_0,t_0),(Z_0,t_0))/\rho,\notag\\
(ii)&&1\lesssim d(P_{\rho,\Lambda}(Z_0,t_0),\Delta_{2r})/\rho,
\end{eqnarray}
whenever $P_{\rho,\Lambda}(Z_0,t_0)\in\{A_{\rho,\Lambda}^\pm(Z_0,t_0), \tilde A_{\rho,\Lambda}^\pm(Z_0,t_0)\}$.
Furthermore, the paths
\begin{eqnarray}\label{admi}
\gamma^+(\tau)=A^+_{(1-\tau)\rho,\Lambda}(Z_0,t_0),\ \gamma^-(\tau)=A^-_{(1-\tau)\rho,\Lambda}(Z_0,t_0),
\tau\in [0,1],
\end{eqnarray}
are admissible paths.
\end{lemma}
\begin{proof} Note that \eqref{akka-} follows immediately from Lemma \ref{coneconditions}. Furthermore, \eqref{akka} is a consequence of Lemma 4.4 in \cite{CNP3}. That the paths in \eqref{admi}
are admissible follows from Lemma \ref{lem.conecond.0}.
\end{proof}


 \begin{lemma}\label{l-coord}  Consider $\Omega_{2r}$. There exist
 $\Lambda=\Lambda(m,M)$, $1\leq \Lambda<\infty$, and  $c_0=c_0(m,M)$,
$1\leq c_0<\infty$, such that the following is true. Let $\rho_0=r/c_0$, $\rho_1=\rho_0/c_0$, assume
$(Z,t)\in \Omega_{\rho}$, $0<\rho<\rho_1$, and let $d=d((Z,t),\Delta_{2r})$. Then there exist
$(Z_0^\pm,t_0^\pm)\in \Delta_{c_0\rho}$ and
$\rho^\pm$ such that
 \begin{equation*}
    (Z, t) = A^\pm_{\rho^\pm,\Lambda}(Z_0^\pm,t_0^\pm)\mbox{ and } \rho^\pm/d\approx 1.
 \end{equation*}
\end{lemma}

\begin{remark}\label{remnot}
 Consider $\Omega_{2r}$. From now on we will let $\Lambda=\Lambda(m,M)$, $1\leq \Lambda<\infty$, $c_0=c_0(m,M)$, $1\leq c_0<\infty$, and $\eta=\eta(m,M)$, $0<\eta\ll 1$,
be such that Lemma \ref{coneconditions} and  Lemma \ref{coneconditions-}  hold whenever $(Z_0,t_0)\in \Delta_{\rho_0}$
and $0<\rho<\rho_0$, $\rho_0=r/c_0$, and such that Lemma \ref{l-coord} holds whenever $(Z,t)\in \Omega_{\rho}$,
$0<\rho<\rho_1$, $\rho_1=\rho_0/c_0$.
\end{remark}

\subsection{The Harnack inequality in cones} The standing assumption in all lemmas stated in this subsection is that $\Omega\subset\mathbb R^{N+1}$ is a Lipschitz domain with
Lipschitz constant $M$, and that $\Omega_{r}=\Omega_{r}(\hat Z_0,\hat t_0)$, $\Delta_{r}=\Delta_{r}(\hat Z_0,\hat t_0)$ for some $(\hat Z_0,\hat t_0)\in\partial\Omega$ fixed, and for $r>0$. Furthermore, $\Lambda$, $c_0$, $\eta$, $\rho_0=r/c_0$, $\rho_1=\rho_0/c_0$, are chosen in accordance with Remark \ref{remnot}.

\begin{lemma}\label{lem4.7bol+}  Consider $\Omega_{2r}$. Let
$\delta$, $0<\delta<1$, be a degree of freedom. There exists  $c=c(m,\kappa,M,\delta)$, $1\leq c<\infty$, such that
following holds. Assume that $u$ is a non-negative weak solution to $\L u=0$ in $\Omega_{2\rho_0}$, let $(Z_0,t_0)\in
\Delta_{\rho_1}$, and consider $\rho$ such that  $0<\rho<\rho_1$. Then
\begin{equation}\label{coneset}
\begin{split}
    (i)\qquad\sup_{\mathcal{B}_{\rho/c}(A_{\delta\rho,\Lambda}^+(Z_0,t_0))}u&\leq
c \inf_{\mathcal{B}_{\rho/c}(A_{\rho,\Lambda}^+(Z_0,t_0))}u,\\
(ii)\qquad\inf_{\mathcal{B}_{\rho/c}(A_{\delta\rho,\Lambda}^-(Z_0,t_0))}u&\geq
c^{-1} \sup_{\mathcal{B}_{\rho/c}(A_{\rho,\Lambda}^-(Z_0,t_0))} u,
\end{split}
\end{equation}
and
\begin{equation}\label{coneset+}
\begin{split}
    (i')\qquad A_{\delta\rho,\Lambda}^+(\tilde Z_0,\tilde t_0)&\in
\mathcal{B}_{\rho/c}(A_{\delta\rho,\Lambda}^+(Z_0,t_0)),\\
(ii')\qquad A_{\delta\rho,\Lambda}^-(\tilde Z_0,\tilde t_0)&\in \mathcal{B}_{\rho/c}(A_{\delta\rho,\Lambda}^-(Z_0,t_0)),
\end{split}
\end{equation}
whenever $(\tilde Z_0,\tilde t_0)\in  \Delta_{\rho/c}(Z_0,t_0)$.
\end{lemma}

\begin{proof}
Let $\delta$ as in the statement of the lemma be given. By construction and Lemma \ref{coneconditions} there exists a constant $\tilde c=\tilde c(m,M,\delta)$ such that
\begin{equation}\label{ballinconeindom}
    \mathcal{B}_{\rho/\tilde c}(A^\pm_{\delta\rho,\Lambda})\subset C^\pm_{\rho,2\eta,\Lambda}(Z_0,t_0)\subset \Omega_{r}.
\end{equation}
Using translation and dilation invariance, we may assume that $(Z_0,t_0)=(0,0)$ and that $\rho=1$. We need to show then, that there exist $c_1$, $c_2$ and $c_3$, each only depending on $m$, $\kappa$, $M$ and $\delta$, such that
\begin{equation}\label{conesetsimple}
    \sup_{\mathcal{B}_{1/c_1}(A_{\delta,\Lambda}^+)}u\leq
c_2 u(A^+_{1,\Lambda}),\ \inf_{\mathcal{B}_{1/c_1}(A_{\delta,\Lambda}^-)}u\geq
c_2^{-1} u(A^-_{1,\Lambda}),
\end{equation}
and
\begin{equation}\label{conesetsimple2}
    A_{\delta,\Lambda}^+(\tilde Z_0,\tilde t_0)\in
\mathcal{B}_{1/c_1}(A_{\delta,\Lambda}^+),\ A_{\delta,\Lambda}^-(\tilde Z_0,\tilde t_0)\in \mathcal{B}_{1/c_1}(A_{\delta,\Lambda}^-),
\end{equation}
whenever $(\tilde Z_0,\tilde t_0)\in\Delta_{1/c_3}$.
Note that using \eqref{ballinconeindom} and Lemma \ref{coneconditions-} we see that
\[
A^+_{\delta,\Lambda}\in \text{Int}\left(\mathcal{A}_{A^+_{1,\Lambda}}\left( C^+_{1,2\eta,\Lambda}(0,0) \right) \right),
\]
and hence there exists a constant $\hat c=\hat c(m,M,\delta)$ such that
\[
\mathcal{B}_{1/\hat c}(A^+_{\delta,\Lambda}) \subset \text{Int}\left(\mathcal{A}_{A^+_{1,\Lambda}}\left( C^+_{1,2\eta,\Lambda}(0,0) \right) \right).
\]
Similarly,
\[
A^-_{\delta,\Lambda}\in \text{Int}\left(\mathcal{A}_{(\tilde Z,\tilde t)}\left( C^-_{1,2\eta,\Lambda}(0,0) \right) \right),
\]
when $(\tilde Z,\tilde t)\in\mathcal{B}_{1/\hat c'}(A^-_{\delta,\Lambda})$. Now \eqref{conesetsimple} follows by putting $c_1=\max\lbrace \hat c,\hat c' \rbrace$ and applying Lemma \ref{t-1}.
Finally, as $A^\pm_{\delta,\Lambda}\in \mathcal{B}_{1/c_1}(A^\pm_{\delta,\Lambda})$,  and using from continuity of the map
\[
(\hat Z_0,\hat t_0) \mapsto A^\pm_{\delta,\Lambda}(\hat Z_0,\hat t_0),
\]
we see that \eqref{conesetsimple2} holds.
\end{proof}

\begin{lemma}\label{lem4.7rho}  Consider $\Omega_{2r}$. There
exists $\gamma=\gamma(m,\kappa,M)$, $0<\gamma<\infty$, such that the following holds. Assume that $u$ is a non-negative weak solution to $\L u=0$ in $\Omega_{2\rho_0}$, let $(Z_0,t_0)\in
\Delta_{\rho_1}$ and consider $\rho$, $\tilde \rho$, $0<\tilde \rho \leq \rho < \rho_1$. Then
\begin{equation}
\begin{split}
    u(A_{\tilde \rho,\Lambda}^+(Z_0,t_0))&\lesssim (\rho/\tilde \rho)^\gamma u(A_{\rho,\Lambda}^+(Z_0,t_0)),\\
    u(A_{\tilde \rho,\Lambda}^-(Z_0,t_0))&\gtrsim(\tilde \rho/\rho)^\gamma u(A_{\rho,\Lambda}^-(Z_0,t_0)).
\end{split}
\end{equation}
\end{lemma}

\begin{proof}
By constructing Harnack chains along the paths given in \eqref{admi} and applying Lemma \ref{lem4.7bol+}, the lemma follows. See Lemma 4.3 in \cite{CNP3} for further details.
\end{proof}

\begin{lemma}\label{lem4.7}  Consider $\Omega_{2r}$. There
exist  $c=c(m,\kappa,M)$, $1\leq c<\infty$, and  $\gamma=\gamma(m,\kappa,M)$, $0<\gamma<\infty$, such that the following holds. Assume that $u$ is a non-negative weak solution to $\L u=0$ in $\Omega_{2\rho_0}$ and let $(Z_0,t_0)\in
\Delta_{\rho_1}$. Then
\begin{equation}
\begin{split}
    u(Z,t)&\lesssim(\rho/d)^\gamma u(A_{\rho,\Lambda}^+(Z_0,t_0)),\\
    u(Z,t)&\gtrsim (d/\rho)^\gamma u(A_{\rho,\Lambda}^-(Z_0,t_0)),
\end{split}
\end{equation}
whenever $(Z,t)\in \Omega_{2\rho/c}(Z_0,t_0)$, $0<\rho<\rho_1$, with $d=d((Z,t),\Delta_{2r})$.
\end{lemma}

\begin{proof}
This lemma is proved by combining Lemma \ref{lem4.7rho} with Lemma \ref{l-coord} and applying Lemma \ref{lem4.7bol+}. We refer to the proof of Lemma 3.10 in \cite{NP} for further details.
\end{proof}

\subsection{Fundamental solutions and estimates thereof}  In this subsection we introduce a fundamental solution to $\L$. The adjoint operator to $\L$ is defined as
\begin{equation}\label{PDEagg}
\L^\ast:=\nabla_X\cdot(A(X,Y,t)\nabla_X)-X\cdot\nabla_Y+\partial_t.
\end{equation}
\begin{definition}\label{fund}
A fundamental solution for $\L$ is a continuous and positive function $\Gamma=\Gamma(Z,t,\tilde Z,\tilde t)$, defined for
$\tilde t<t$ and $Z,\tilde Z\in\R^{N}$, such that
\begin{itemize}
  \item[(i)] $\Gamma( \cdot,\cdot, \tilde Z,\tilde t)$ is a weak solution of $\L u=0$ in $\mathbb R^N\times (\tilde t,\infty)$ and
  $\Gamma(Z,t,\cdot,\cdot)$ is a weak solution of $\L^{*} u=0$ in $\mathbb R^N\times (-\infty,t)$,
  \item[(ii)] for any bounded function $\phi\in C(\R^{N})$ and $Z,\tilde Z\in\R^{N}$, we have
\begin{align}
  \lim_{(Z,t)\to(\tilde Z,\tilde t)\atop t>\tilde t}u(Z,t)=\phi(\tilde Z), 
  \qquad \lim_{(\tilde Z,\tilde t)\to(Z,t)\atop t>\tilde t}v(\tilde Z,\tilde t)=\phi(Z), 
\end{align}
where
  \begin{align}\label{ae11}
 u(Z,t):=\iint_{\R^{N}}\Gamma(Z,t,\tilde Z,\tilde t)\phi(\tilde Z)\, \d\tilde Z,\qquad 
 v(\tilde Z,\tilde t):=\iint_{\R^{N}}\Gamma(Z,t,\tilde Z,\tilde t)\phi(Z)\, \d Z.
\end{align}
\end{itemize}
\end{definition}

\begin{remark}
Note that for any $\phi\in C^\infty_0(\R^{N+1})$ the following identities hold
\begin{equation*}
\begin{split}
     \iiint_{}\ \bigl(A(X,Y,t)\nabla_X \Gamma( \cdot,\cdot, \tilde Z,\tilde t) \cdot \nabla_X\phi+  \Gamma( \cdot,\cdot, \tilde Z,\tilde t)(X\cdot \nabla_Y\phi-\partial_t\phi)\bigr )\, \d X \d Y \d t=\phi(\tilde Z,\tilde t),\\
     \iiint_{}\ \bigl(A(\tilde X,\tilde Y,\tilde t)\nabla_{\tilde X} \Gamma(Z,t,\cdot,\cdot) \cdot \nabla_{\tilde X}\phi-  \Gamma(Z,t,\cdot,\cdot)(\tilde X\cdot \nabla_{\tilde Y}\phi-\partial_{\tilde t}\phi)\bigr )\, \d \tilde X \d \tilde Y \d \tilde t=\phi(Z,t).
\end{split}
\end{equation*}
\end{remark}

\begin{remark}
The functions in \eqref{ae11} are weak solutions of the following backward and forward Cauchy problems,
  $$
  \begin{cases}
    \L u(Z,t)= 0,\ &(Z,t)\in \mathbb R^N\times (\tilde t,\infty), \\
    u(Z,\tilde t)= \phi(Z), & Z \in\R^{N},
  \end{cases}\qquad
  \begin{cases}
    \L^{*}v(\tilde Z,\tilde t)= 0,\ & (\tilde Z,\tilde t)\in \mathbb R^N\times (-\infty,t), \\
    v(\tilde Z,t)= \phi(\tilde Z) & \tilde Z \in\R^{N}.
  \end{cases}
  $$
\end{remark}

Let $B$ and $E$ be as defined in \eqref{def:matrixE} and let $C$ be defined as in
\eqref{def:matrixC}. Recall \eqref{def:matrixCinv}. Using this notation, an explicit fundamental solution to the constant coefficient operator
\begin{equation}\label{constcoeff}
    \L^\lambda := \frac{\lambda}{2}\nabla_X\cdot\nabla_X + X\cdot\nabla_Y -\partial_t,
\end{equation}
with pole at $(\tilde Z,\tilde t)$, $\Gamma^\lambda(\cdot,\cdot,\tilde Z,\tilde t)$, can be defined by
\begin{equation}\label{ff}
 \Gamma^\lambda(Z,t,\tilde Z,\tilde t):= \Gamma^\lambda(Z-E(t-\tilde t)\tilde Z,t-\tilde t, 0,0) 
\end{equation}
where
\begin{align}\label{ffff}
     \Gamma^\lambda(Z,t,0,0)&=\frac {1}{(2\pi\lambda)^{m}\sqrt{\det \mathcal{C}(t)}}\ e^{\left(-\frac{1}{2\lambda}
\lan\mathcal{C}(t)^{-1}Z,Z\ran\right)},\quad \text{if }t>0,\\
    \Gamma^\lambda(Z,t,0,0)&=0,\quad \text{if } t\leq 0.
\end{align}
Here $ \lan\cdot,\cdot\ran$ denotes the standard inner product on $ \mathbb R^{N}$.

\begin{lemma}\label{lem_fsolbounds}
  Assume that $A$ satisfies \eqref{eq2} and \eqref{eq2+}. Then there exists a fundamental solution $\Gamma(Z,t,\tilde{Z},\tilde{t})$ to $\L$ in the sense of Definition \ref{fund}. Furthermore, there exist positive constants $\lambda^+$, $\lambda^-$,  depending only on $m$ and  $\kappa$, such that
 \begin{equation}\label{fsolbounds}
    \Gamma^{\lambda^-}(Z,t,\tilde{Z},\tilde{t})\lesssim\Gamma(Z,t,\tilde{Z},\tilde{t})\lesssim\Gamma^{\lambda^+}(Z,t,\tilde{Z},\tilde{t})
 \end{equation}
 for all $(Z,t)$, $(\tilde{Z},\tilde{t})$ with $t>\tilde t$.
\end{lemma}
\begin{proof} We refer to \cite{P1, DiFP, DeM} for the existence of the  fundamental solution for $\L$ under the additional condition that the coefficients are
H\"older continuous. See also \cite{LPP}.
\end{proof}

Using Lemma \ref{lem_fsolbounds}, and the arguments in Subsection 2.2 in \cite{NP}, we have the upper bound
\begin{equation}\label{fundsolbd}
    \Gamma(Z,t,\tilde Z,\tilde t) \lesssim \frac{1}{d((Z,t),(\tilde Z,\tilde t))^{\mathbf{q}-2}},
\end{equation}
for all $(Z,t)$, $(\tilde{Z},\tilde{t})$ with $t>\tilde t$.

\section{Boundary H{\"o}lder continuity: Proof of Theorem \ref{lem4.5-Kyoto1}}\label{sec5}
Theorem \ref{lem4.5-Kyoto1} is an immediate consequence of Lemma \ref{boundaryholder+} stated and proved below. To  start the argument towards
the proof of Lemma \ref{boundaryholder+} we first need to introduce some additional notation.

Recall the reference points defined in \eqref{pointsref2a} and, in particular,
\begin{eqnarray*}
\tilde A_{\rho,\Lambda}^-&=&(0,-\Lambda\rho,0,-\frac 2 3\Lambda\rho^3,-\rho^2)\in\mathbb R^{m-1}\times\mathbb R\times\mathbb R^{m-1}\times\mathbb R \times\mathbb R.
\end{eqnarray*}
 Recall that if $(Z_0,t_0)\in\partial\Omega$, then $\tilde A_{\rho,\Lambda}^-(Z_0,t_0)=(Z_0,t_0)\circ \tilde A_{\rho,\Lambda}^-$ is a reference point in $\mathbb R^{N+1}\setminus\Omega$ and into past relative to $t_0$.  We furthermore let
\begin{eqnarray}\label{flatcubes}
    \hat Q_r:=\{(X,Y) \mid |x_i|<r,\ |y_i|<r^3\},
\end{eqnarray}
for $r>0$. We also let
\[
    \hat Q_r(Z,t) := ((Z,t)\circ Q_r) \cap \lbrace (\tilde Z, \tilde t) \mid \tilde t=t \rbrace,
\]
whenever $(Z,t)\in\R^{N+1}$.
We recall the cylinders $Q^-_r$ introduced in \eqref{pastcyl}, and we, in addition, introduce
\begin{eqnarray*}\label{pastcyl2}
    Q_{r_1,r_2}^-:=\hat Q_{r_1}\times\{t \mid  -r_2^2<t\leq 0\},\quad Q_{r_1,r_2}^-(Z,t)=(Z,t)\circ Q_{r_1,r_2}^-.
\end{eqnarray*}
The  following is the key lemma proved in this section.
\begin{lemma}\label{boundaryholderkey}
Let $(Z_0,t_0)\in\partial\Omega$ and $r>0$. Then there exists a constant $K=K(m,\kappa,M)\gg 1$, such that the following is true.
Let $u$ be a non-negative weak solution of $\L u=0$ in $\Omega\cap Q_{Kr,2r}^-(Z_0,t_0) $, vanishing continuously on  $\partial\Omega\cap Q_{Kr,2r}^-(Z_0,t_0) $. Then there exists $\theta=\theta(m,\kappa,M,K)$, $0<\theta<1$, such that
\begin{align*}
\sup_{\Omega\cap Q_{r/K}^-(Z_0,t_0)}u \le \theta \sup_{\Omega\cap Q_{Kr,2r}^-(Z_0,t_0)}u.
\end{align*}
\end{lemma}
\begin{proof} We first note that we can without loss of generality assume that $(Z_0,t_0)=(0,0)$ and $r=1$. Let $K\gg 1$ be a constant to be
fixed below. We let $\phi _{1}\in C_{0}^{\infty }(\mathbb R^N)$ be such
that $0\leq \phi _{1}\leq 1$, and $\phi _{1}\equiv 1$ on $$\hat Q_{(K+1)}\setminus \hat Q_{(K-1 )},$$ and such that  $\phi _{1}\equiv 0$ on
$$ \hat Q_{(K-2 )}\cup (\mathbb R^N\setminus \hat Q_{(K+2 )}).$$
Similarly, we let $\phi _{2}\in C_{0}^{\infty }(\mathbb R^N)$ be such
that $0\leq \phi _{2}\leq 1$, $\phi _{2}\equiv 1$ on $$\hat Q_{(K+1 )}\setminus (\mathcal{B}_{2/M}(\tilde A_{2,\Lambda}^-)\cap\{t \mid t=-4\}),$$
and such that  $\phi _{2}\equiv 0$ on $$ (\mathcal{B}_{1/M}(\tilde A_{2,\Lambda}^-)\cap\{t \mid t=-4\})\cup (\mathbb R^N\setminus \hat Q_{(K+2 )}).$$ Note that by
construction we have $(\mathcal{B}_{2/M}(\tilde A_{2,\Lambda}^-)\cap\{t \mid t=-4\})\subset \mathbb R^{N+1}\setminus\bar{\Omega}$. Using $%
\phi _{1}$ and $\phi _{2}$ we let $\Phi _{1}(Z,t)$ and $\Phi _{2}(Z,t)$ be the solutions to the Cauchy  problem
for $\mathcal {L}$ with data
$\phi _{1}$ and $\phi _{2}$, respectively, on $\{t=-4\}$. Hence $\mathcal {L}\Phi _{1}(Z,t)=0=\mathcal {L}\Phi _{2}(Z,t)$ whenever $(Z,t)\in \R^{N+1}$, $t> -4$ and we can represent $\Phi _{1}(Z,t)$ and $\Phi _{2}(Z,t)$ using the fundamental solution. The remainder of the proof will consist of showing the validity of four claims from which the lemma will follow.

\noindent {\textit{Claim 1: There exists a constant $c=c(m,\kappa,M)\geq 1$ such that
\begin{equation}
\label{jul1}
1\leq c\Phi _{1}(Z,t),
\end{equation}
whenever $(Z,t)\in
\partial Q_{{K},2}^{-}\cap
\{(Z,t)\mid -4<t<0\}.$}}

\noindent \textit{Proof of the claim.} Using Lemma \ref{lem_fsolbounds} we see that solution to the Cauchy problem with initial data $\phi_1$ can be represented as
\[
\Phi_1(Z,t) = \iint_{\mathbb{R}^N}\Gamma(Z,t,\tilde Z, -4)\phi_1(\tilde Z)\d\tilde Z,
\]
where $\Gamma(\cdot,\cdot,\tilde Z,\tilde t)$ is the fundamental solution for $\mathcal{L}$ with pole at $(\tilde Z,\tilde t)$. Note that by Lemma \ref{lem_fsolbounds} we have the lower bound
\begin{equation}
    c\Gamma^\lambda(Z,t,\tilde Z, -4) \leq \Gamma(Z,t,\tilde Z, -4)
\end{equation}
for some $c=c(m,\kappa)>0$, where $\Gamma^\lambda(\cdot,\cdot,\tilde Z,\tilde t)$, $\lambda:=\lambda^-$, is the fundamental solution for the constant coefficient operator defined in \eqref{constcoeff}.
Hence,
\[
\Phi_1(Z,t) \geq c\iint_{\mathbb{R}^N}\Gamma^\lambda(Z,t,\tilde Z, -4)\phi_1(\tilde Z)\d\tilde Z.
\]
Let $\delta>0$ some small fixed constant and consider the shifted cube $$\hat Q_\delta (Z,-4) = ((Z,-4)\circ Q_\delta) \cap \lbrace (Z,t) \mid t=-4 \rbrace.$$ Note that $\hat Q_\delta (Z,-4)\subset (\hat Q_{K+1}\backslash \hat Q_{K-1})\times \lbrace (Z,t) \mid t=-4 \rbrace$.
Using this together with the fact that $\Gamma^\lambda$ is non-negative, smooth and not identically zero in $\hat Q_\delta (Z,-4)$, we have
\[
\iint_{\mathbb{R}^N}\Gamma^\lambda(Z,t,\tilde Z, -4)\phi_1(\tilde Z)\d\tilde Z \geq \iint_{\hat Q_\delta (Z,-4)}\Gamma^\lambda(Z,t,\tilde Z, -4)\d\tilde Z \gtrsim 1.
\]
The completes the proof of the claim.\qed

To simplify notation a little, we will in the sequel write
\begin{equation}
\Psi:=\sup_{\Omega\cap Q_{K,2}^-}u.  \label{supM}
\end{equation}%
With this notation, we see that by using (\ref{jul1}) and the maximum principle on $\Omega\cap Q_{K,2}^-$ we obtain
\begin{equation}
u(Z,t)\leq c\Psi\Phi _{1}(Z,t)+\Psi\Phi _{2}(Z,t),  \label{compky4}
\end{equation}%
when $(Z,t)\in\Omega\cap Q_{K,2}^-$, and thus, in particular, when $(Z,t)\in \Omega\cap Q_{1}^-$.

\noindent {\textit{Claim 2:  If $(Z,t)\in \Omega\cap Q_{1}^-$, then there exist $c\geq 1$, and an integer $\eta \gg 1$,  both independent of ${K}$, such that
\begin{equation}\label{estimateclaim2}
\Phi _{1}(Z,t)\leq c e^{-c^{-1}{K}^{2}}{K}^{\eta }.
\end{equation}}}

\noindent
\textit{Proof of the claim.} Again using Lemma \ref{lem_fsolbounds}, with $\lambda:=\lambda^+$, we have
\begin{align*}
    \Phi_1(Z,t) &= \iint_{\mathbb{R}^N}{\Gamma(Z,t,\tilde Z,-4)\phi_1(\tilde Z) \d\tilde Z}\\
    &\leq
    \iint_{\hat{Q}_{K+2}\backslash \hat{Q}_{K-2}}{\Gamma(Z,t,\tilde Z,-4) \d\tilde Z}\\
    &\lesssim  \iint_{\hat{Q}_{K+2}\backslash \hat{Q}_{K-2}}{\Gamma^\lambda (Z,t,\tilde Z,-4) \d\tilde Z}\\
    &\lesssim \frac{1}{\lambda^m(t+4)^{2m}}
    \iint_{\hat{Q}_{K+2}\backslash \hat{Q}_{K-2}}{e^{-\frac{1}{2\lambda}\left( \frac{1}{t+4}|X-\tilde X|^2 + \frac{3}{(t+4)^3}\left| 2(Y-\tilde Y) + (t+4)(X+\tilde X) \right|^2 \right)} \d\tilde Z}.
\end{align*}
We now consider now the two cases,
\begin{align*}
        |X-\tilde X|^2\geq m|K-3|^2, \,\text{and}\, |X-\tilde X|^2< m|K-3|^2.
\end{align*}
\begin{figure}
    \centering
    \includegraphics{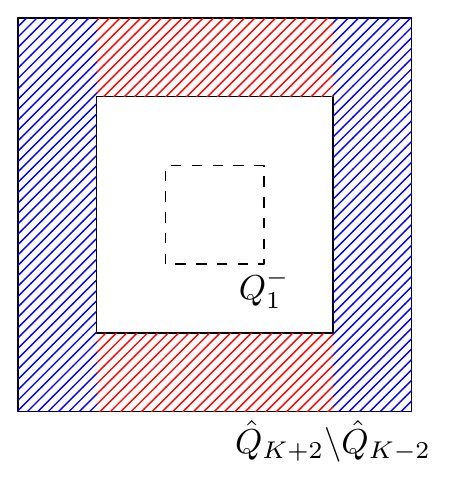}
    \caption{The shaded regions represent the two cases in the proof of Claim 2.}
    \label{fig:boxcases}
\end{figure}
In the first case we immediately see that
\begin{align}\label{est}
e^{-\frac{1}{2\lambda}\left( \frac{1}{t+4}|X-\tilde X|^2 + \frac{3}{(t+4)^3}\left| 2(Y-\tilde Y) + (t+4)(X+\tilde X) \right|^2 \right)} \leq e^{-cK^2}.
\end{align}
In the second case we note that $|Y-\tilde Y|^2\gtrsim mK^6$, for $K$ large enough, and due to the geometry of $\hat{Q}_{K+2}\backslash \hat{Q}_{K-2}$. Also, $|X+\tilde X|^2\lesssim K^2$ and using that  \begin{equation*}
    \begin{split}
        \left| 2(Y-\tilde Y) + (t+4)(X+\tilde X) \right|^2 &= 3|Y-\tilde Y|^2-(t+4)^2|X+\tilde X|^2+|Y-\tilde Y +2(t+4)(X+\tilde X)|^2\\
        &\lesssim 3|Y-\tilde Y|^2-(t+4)^2|X+\tilde X|^2 + \left(|Y-\tilde Y| + 2(t+4)|X+\tilde X| \right)^2,
    \end{split}
\end{equation*}
we can also in the second case conclude the validity of \eqref{est} for some $c>0$ independent of $K$. This together with the elementary fact that there exists $\eta>1$ (independent of $K$) such that $|\hat{Q}_{K+2}\backslash \hat{Q}_{K-2}|\leq K^\eta$, yields that
\[
\Phi_1(Z,t) \leq c \iint_{\hat{Q}_{K+2}\backslash \hat{Q}_{K-2}}{e^{-\frac{1}{c} K^2} \d\tilde Z}\leq c e^{-c^{-1}K^2}K^\eta,
\]
which is the statement of the claim.
\qed

To be able to estimate $\Phi _{2}(Z,t)$ we  write
\[
\Phi _{2}=1-\hat{\Phi}_{2}
\]
where now, in particular, $\hat{\Phi}_{2}$ is a non-negative function
such that
\[
\hat{\Phi}_{2}(Z,-4)=1,
\]
whenever $(Z,-4)\in(\mathcal{B}_{1/M}(\tilde A_{2,\Lambda}^-)\cap\{t \mid t=-4\})$.

\noindent {\textit{Claim 3: There exist $\epsilon>0 $ small, and $c\geq 1$, both depending only on $m$, $\kappa$ and $M$, such that
\begin{equation}\label{estimate}
\hat{\Phi}_{2}(Z,t)\geq c^{-1},
\end{equation}
whenever $(Z,t)\in (\mathcal{B}_{(1-\epsilon)/M}(\tilde A_{2(1-\epsilon),\Lambda}^-)\cap\{t \mid t=-4(1-\epsilon)^2\})$.}}

\noindent
\textit{Proof of the claim.} As in the proof of Claim 1 above we fix a small $\delta>0$. Let $\epsilon>0$ be such that $$\hat{Q}_\delta (Z,-4) \subset (\mathcal{B}_{1/M}(\tilde A_{2,\Lambda}^-)\cap\{t \mid t=-4\}),$$ when $(Z,t)\in (\mathcal{B}_{(1-\epsilon)/M}(\tilde A_{2(1-\epsilon),\Lambda}^-)\cap\{t \mid t=-4(1-\epsilon)^2\})$. Now, since $\hat{\phi}_2=1-\phi_2$ is bounded on $\mathbb{R}^N$, we can represent $\hat{\Phi}_2$ as
\[
\hat{\Phi}_2(Z,t) = \iint_{\mathbb{R}^N}{\Gamma(Z,t,\tilde Z,-4)\hat{\phi}_2(\tilde Z) \d\tilde Z}
\]
and the claim follows by essentially the same argument as in the proof of Claim 1.
\qed

\noindent {\textit{Claim 4: There exists $K_0=K_0(m,\kappa,M)\gg 1$ such that if $K\geq K_0$ then
\begin{equation}\label{estimate+}
\hat{\Phi}_{2}(Z,t)\geq \tilde c^{-1},
\end{equation}
whenever $(Z,t)\in\Omega\cap Q_{1/K}^-$, for a constant $\tilde c=\tilde c(m,\kappa,M)\geq 1$.}}

\noindent
\textit{Proof of the claim.} The idea is to choose an admissible curve $\gamma(\tau)$ connecting the origin to $\tilde A_{2,\Lambda}$ and construct a Harnack chain along it. Consider the curve $$\gamma(\tau) = \tilde A^-_{2\tau,\Lambda}.$$
A direct calculation shows that
$$\frac{d}{d\tau}{\gamma}(\tau) = \sum_{j=1}^m{\omega_j(\tau)\partial_{x_j}\gamma(\tau)} + \lambda(\tau)\left(\sum_{k=1}^m{x_k\partial _{y_k}\gamma(\tau)}-\partial_t\gamma(\tau)\right),$$
with $\omega_m=-2\Lambda$, $\omega_j=0$ for $j=1,\cdots,m-1$, and $\lambda(\tau)=8\tau$. Hence $\gamma(\tau)$ is admissible. Furthermore $\gamma(0)=(0,0)$ and $\gamma(1)=\tilde A_{2,\Lambda}$. Hence, using Lemma \ref{akkaa-} we deduce that
\begin{equation}\label{eq:supepsilon}
    \sup_{\mathcal{B}_{\tilde\epsilon} (\tilde A^-_{2(1-{\tilde\epsilon}),\Lambda})} \hat{\Phi}_2 \leq C_{\tilde\epsilon} \hat{\Phi}_2(0,0),
\end{equation}
for some $C_{\tilde\epsilon}>0$ and ${\tilde\epsilon}>0$ small. It follows from (\ref{estimate}) that
$$\hat{\Phi}_2(0,0)\geq c^{-1}.$$
Using that $\gamma$ is admissible it can be shown that for $\delta>0$ small, $\tilde{\gamma}(\tau) = \tilde A^-_{2\tau,\Lambda}(Z_0,t_0)$ is admissible whenever $(Z_0,t_0)\in Q^-_{\delta}$. Pick then $\delta>0$ small enough so that $\tilde{\gamma}$ is admissible and
$$\tilde{\gamma}(1)\in (\mathcal{B}_{(1-\epsilon)/M}(\tilde A_{2(1-\epsilon),\Lambda}^-)\cap \mathcal{B}_{\tilde\epsilon} (\tilde A^-_{2(1-{\tilde\epsilon}),\Lambda})),$$
whenever $(Z_0,t_0)\in Q^-_{\delta}$. The claim now follows by picking $K_0>1/\delta$.
\qed

As a consequence of the last claim we have
\begin{equation}\label{estimate++}
\Phi _{2}(Z,t)=1-\hat{\Phi}_{2}(Z,t)\leq \left( 1-\tilde c^{-1}\right)
\end{equation}
whenever $(Z,t)\in\Omega\cap Q_{1/K}^-$.
We now put the estimates together and we can conclude that we have proved that if $K\geq K_0$, then
\begin{equation*}
u(Z,t)\leq c\Psi\Phi _{1}(Z,t)+\Psi\Phi _{2}(Z,{t})\leq \Psi(ce^{-c^{-1}{K}^{2}}{K}^{\eta }+\left( 1-\tilde{c}^{-1}\right)
),
\end{equation*}
whenever $(Z,t)\in\Omega\cap Q_{1/K}^-$ as we see from \eqref{compky4}, \eqref{estimateclaim2} and \eqref{estimate++}. Given $\tilde{c}$, we choose ${K}\geq K_0$ so that
\[
c e^{-c^{-1}{K}^{2}}{K}^{\eta
}\leq \frac{1}{2\tilde c},
\]
and we let $\theta = 1-(2\tilde{c})^{-1} <1$. Put together we see that
\begin{equation}
u(Z,t)\leq \theta \Psi,\label{compky10b}
\end{equation}%
whenever  $(Z,t)\in \Omega\cap Q_{1/K}^-$.
This completes the proof of the lemma.
\end{proof}

\begin{lemma}\label{boundaryholder_giventheta}
Let $(Z_0,t_0)\in\partial\Omega$ and $r>0$. Let $u$ be a non-negative weak solution of $\L u=0$ in $\Omega\cap Q_{2r}(Z_0,t_0) $, vanishing continuously on  $\partial \Omega\cap Q_{2r}(Z_0,t_0)$. Let $0<\theta<1$ be given. Then there exists a constant $c=c(m,\kappa,M,\theta)$, $1\leq c <\infty$, such that
\begin{equation*}
    \sup_{\Omega\cap Q_{r/c}(Z_0,t_0)}u \leq \theta \sup_{\Omega\cap Q_{2r}(Z_0,t_0)} u.
\end{equation*}
\end{lemma}
\begin{proof}
The lemma is an immediate consequence of Lemma \ref{boundaryholderkey}.
\end{proof}

\begin{lemma}\label{boundaryholder}
Let $(Z_0,t_0)\in\partial\Omega$ and $r>0$. Let $u$ be a non-negative weak solution of $\L u=0$ in $\Omega\cap Q_{2r}(Z_0,t_0) $, vanishing continuously on  $\partial \Omega\cap Q_{2r}(Z_0,t_0)$. Then,
there exist a constant $c=c(m,\kappa,M)$, $1\leq c<\infty $, and $\alpha
=\alpha (m,\kappa,M)\in (0,1)$, such that
\begin{equation*}
u(Z,t)\leq c\biggl (\frac{d((Z,t),(Z_{0},t_{0}))}{r}\biggr )^{\alpha
}\sup_{\Omega\cap Q_{2r}(Z_{0},t_{0})}u
\end{equation*}%
whenever $(Z,t)\in \Omega\cap Q_{r/c}(Z_{0},t_{0})$.
\end{lemma}
\begin{proof} The lemma follows immediately from Lemma \ref{boundaryholder_giventheta}.
\end{proof}

\begin{lemma}\label{boundaryholder+}
Let $(Z_0,t_0)\in\partial\Omega$ and $r>0$. Then there exists a constant $K=K(m,\kappa,M)\gg 1$, such that the following is true.
Let $u$ be a non-negative weak solution of $\L u=0$ in $\Omega\cap Q_{Kr,2r}^-(Z_0,t_0) $, vanishing continuously on  $\partial\Omega\cap Q_{Kr,2r}^-(Z_0,t_0) $. Then there exists $\theta=\theta(m,\kappa,M,K)$, $0<\theta<1$, such that
\begin{align*}
\sup_{\Omega\cap Q_{r/K}^-(Z_0,t_0)}u^\pm \le \theta \sup_{\Omega\cap Q_{Kr,2r}^-(Z_0,t_0)}u^\pm,
\end{align*}
where $u^{+}(Z,t)=\max \{0,u(Z,t)\}$, $u^{-}(Z,t)=-\min \{0,u(Z,t)\}$.
\end{lemma}
\begin{proof}
We first prove the lemma for $u^{+}$. In this case the argument is essentially the same as that in the proof of
Lemma \ref{boundaryholderkey}. In particular, if we let
\begin{equation*}
\Psi^{+}=\sup_{\Omega\cap Q_{{K}r,2r}^{-}(Z_{0},t_{0})}u^{+},
\end{equation*}%
then we see that (\ref{compky4}) still holds but with $\Psi$ replaced by $\Psi^{+}$. Furthermore, simply repeating the argument in Lemma \ref{boundaryholderkey}
we deduce
\begin{equation*}
u({Z},{t})\leq \theta \Psi^{+},
\end{equation*}%
whenever $({Z},{t})\in \Omega\cap Q_{r/K}^{-}(Z_{0},t_{0})$. Obviously this completes the proof for $u^{+}$.
Concerning the same estimate for $u^{-}$ we see, by analogy, that
\begin{equation}
-u(\hat{Z},{t})\leq \theta \Psi^{-},\label{compky10jajb}
\end{equation}%
where
\[
\Psi^{-}=\sup_{\Omega\cap Q_{{K}%
r,2r}^{-}(Z_{0},t_{0})}(-u)=\sup_{\Omega\cap Q_{{K}%
r,2r}^{-}(Z_{0},t_{0})}u^{-},
\]
whenever $({Z},{t})\in \Omega\cap Q_{r/K}^{-}(Z_{0},t_{0})$ and
from (\ref{compky10jajb}) we deduce Lemma \ref{boundaryholder+} for $u^{-}$. This
completes the proof of the lemma.
\end{proof}

\section{The Dirichlet problem: proof of Theorem \ref{thm:dp}}\label{sec6}
We here consider the well-posedness of the Dirichlet problem with continuous boundary data for the operator $\L$ in domains $\Omega_r(Z_0,t_0)$ introduced in \eqref{def.Omega.Delta.fr}. We recall that
in Definition 3 in \cite{NP} it was introduced what we here, as in \cite{NP}, refer to as the Kolmogorov boundary of $\Omega_r(Z_0,t_0)$, denoted
$\partial_K \Omega_r(Z_0,t_0)$. The notion of the Kolmogorov boundary replaces the notion of the parabolic boundary used
in the context of uniformly parabolic equations and by definition $\partial_K \Omega_r(Z_0,t_0)\subset \partial\Omega_r(Z_0,t_0)$ is the set of all points on the topological
boundary of $\Omega_r(Z_0,t_0)$ which is {contained} in the closure of the propagation set of at least one interior point in
$\Omega_r$. The Kolmogorov boundary $\partial_K \Omega_r(Z_0,t_0)$ is the largest subset of the topological boundary of
$\Omega_r(Z_0,t_0)$ on which we can attempt to impose boundary data if we want to construct non trivial solutions to the Dirichlet problem in $\Omega_r(Z_0,t_0)$ for the operator $\L$.  The notion
of regular points on $\partial\Omega_r(Z_0,t_0)$ for
the Dirichlet problem only makes sense for points on the Kolmogorov boundary. Based on this we consider the well-posedness of the boundary value problem
\begin{equation} \label{e-bvplocal}
\begin{cases}
	\L u = 0  &\text{in} \ \Omega_r(Z_0,t_0), \\
      u = \varphi & \text{on} \ \partial \Omega_r(Z_0,t_0).
\end{cases}
\end{equation}
where $\varphi\in C(\mathbb R^{N+1})$. The boundary data should be understood as only imposed on the Kolmogorov boundary. Indeed, we define solutions to \eqref{e-bvplocal} as follows.
\begin{definition}\label{defDP} Let $\varphi\in C(\mathbb R^{N+1})$. We say that $u$ is a solution to the Dirichlet problem in \eqref{e-bvplocal} if $u$ is a weak solution to
$\L u=0$ in $\Omega_r(Z_0,t_0)$, if $u$ is continuous on the closure of $\Omega_r(Z_0,t_0)$ and if  $u=\varphi$ on $\partial_K \Omega_r(Z_0,t_0)$.
\end{definition}

   To prove solvability of the  Dirichlet problem in  \eqref{e-bvplocal}, as defined in Definition \ref{defDP}, we will make use of our qualitative assumption in \eqref{eq2+}. We remark that the  assumption in \eqref{eq2+} can be removed once uniqueness of solutions to \eqref{e-bvplocal} can be established under the assumption
   that coefficients have no smoothness beyond being bounded and measurable. Indeed, in this case, and by an approximation argument, is suffices to consider the Dirichlet problem in \eqref{e-bvplocal} for the regularized operator
\begin{eqnarray}
   \L_{\epsilon}:=\nabla_X\cdot(A^\epsilon(X,Y,t)\nabla_X)+X\cdot\nabla_{Y}-\partial_t,
    \end{eqnarray}
    where $\epsilon>0$ is small and $A^\epsilon$ is a regularization of $A$ constructed by a group convolution of $A$ with respect to an approximation of the identity with parameter $\epsilon$.

    \begin{lemma}\label{compkol+} Assume \eqref{eq2} and \eqref{eq2+}. Let $\varphi\in C(\mathbb R^{N+1})$. Then the Dirichlet problem in \eqref{e-bvplocal} has a unique solution $u$ and
$$||u||_{L^\infty(\Omega_r(Z_0,t_0))}\leq ||\varphi||_{L^\infty(\partial_K\Omega_r(Z_0,t_0))}.$$
\end{lemma}
\begin{proof} As $A$ is smooth we can freely switch between considering $\L$ as an operator in divergence form and as an operator in non-divergence form. In non-divergence form we have
\begin{eqnarray}\label{non-div}
   \L:=\sum_{i,j=1}^m a_{i,j}\partial_{x_ix_j}+\sum_{i,j=1}^m \partial_{x_i}a_{i,j}\partial_{x_j}+X\cdot\nabla_{Y}-\partial_t,
    \end{eqnarray}
    with $\partial_{x_i}a_{i,j}$ locally bounded. The lemma now follows from the methods employed in \cite{Manfredini}. In particular, it is enough to prove existence of barrier functions at each $(Z,t)\in\partial_K \Omega_r(Z_0,t_0)$. In the following we can, without loss of generality, assume that  $(Z_0,t_0)=(0,0)$. We introduce the sets
\begin{eqnarray*}
S^\pm_{1,i} &=& \Omega_r \cap \lbrace (X,Y,t)\mid x_i=\pm r \rbrace,\quad i=1,\cdots,m-1,  \\
S^\pm_{2,i} &=& \Omega_r \cap \lbrace (X,Y,t)\mid \pm x_i>0,\, y_i=\pm r^3 \rbrace,\quad i=1,\cdots,m,\\
S_3 &=& \Omega_r \cap \lbrace (X,Y,t)\mid t=-r^2 \rbrace\\
S_4 &=& \Omega_r \cap	\lbrace (X,Y,t) \mid x_m=4Mr \rbrace,
\end{eqnarray*}
and we note that
$$ \partial_K \Omega_r(Z_0,t_0)=\Delta_r\cup_iS^\pm_{1,i}\cup_iS^\pm_{2,i}\cup S_3\cup S_4.$$
As points of the sets $S^\pm_{1,i}$ and $S_4$ are non-characteristic for the operators these points are regular for the Dirichlet problem. Indeed for a point $(\hat Z,\hat t)\in S^\pm_{1,i}\cup S_4$, the function
	\[
	w(Z,t) = e^{-K |\nu|^2}-e^{-K|(Z,t)-(\hat Z,\hat t)-\nu|^2},
	\]
	where $\nu$ is an exterior normal and $K\gg 1$ is large enough, is a barrier at $(\hat Z,\hat t)$.
	For points $(\hat Z,\hat t)\in S^\pm_{2,i}$, the function $$w(Z,t) = \pm ((\hat y)_i-y_i)$$ is a barrier. For the set $S_3$ a barrier is constructed analogously, in particular the function
	\[
	w(Z,t) = t + r^2
	\]
	is a barrier at $(\hat Z, -r^2)\in S_3$. Finally, consider  $(\hat Z,\hat t)\in \Delta_r$. It follows from Lemma \ref{coneconditions} that there exist $\eta=\eta(m,M)$, $\Lambda=\Lambda(m,M)$, and $0<\rho<r$ such that the cone $\tilde C_{\rho,\eta,\Lambda}^-(\hat Z,\hat t)$, as defined in \eqref{conesha}, satisfies
	\[
	    \tilde C_{\rho,\eta,\Lambda}^-(\hat Z,\hat t) \in \R^{N+1}\setminus\Omega_r.
	\]
	This implies that $(\hat Z,\hat t)$ satisfies the assumptions of Theorem 6.3 in \cite{Manfredini} and thus is a regular point. In particular, the set of regular points coincides with $\partial_K \Omega_r(Z_0,t_0)$. The inequality
$$||u||_{L^\infty(\Omega_r(Z_0,t_0))}\leq ||\varphi||_{L^\infty(\partial_K\Omega_r(Z_0,t_0))}$$
is a consequence of the maximum principle of Lemma \ref{maxprinciple}.
\end{proof}

\subsection{Proof of Theorem \ref{thm:dp}}
Let $\Omega\subset\mathbb R^{N+1}$ be a Lipschitz domain with defining function $\psi$ and  constant $M$. Let $\varphi\in
C(\partial\Omega)\cap L^\infty(\partial\Omega)$ be such that $\varphi(Z,t)\to 0$  as $||(Z,t)||\to \infty$. To prove Theorem \ref{thm:dp} we first need to prove that there exists
a unique  solution  $u=u_\varphi$, $u\in C(\bar \Omega)$, to the Dirichlet problem
\begin{equation}
\begin{cases}
	\L u = 0  &\text{in} \ \Omega, \\
      u = \varphi  & \text{on} \ \partial \Omega.
\end{cases}
\end{equation}
The uniqueness part of this statement is a consequence of Lemma \ref{maxprinciple_unbounded} stated and proved below. To prove existence we note that we can without loss of generality assume that $(0,0)\in\partial\Omega$, and that $\varphi\geq 0$. Given $\varphi$, let $u_k$, for $k\geq 1$, be the unique weak solution to $\mathcal{L}u=0$ in $\Omega_k=\Omega\cap Q_k(0,0)$ with boundary values
\[
    u_k(Z,t)=\varphi(Z,t)\phi\left(\frac{||(Z,t)||}{k}\right),\quad \text{when}\, (Z,t)\in\Delta_k,
\]
and $u=0$ on $\partial\Omega\setminus\Delta_k$. Here, $\phi$ is a continuous decreasing function on $[0,\infty)$ such that $0\leq\phi\leq 1$, $\phi(s) = 1$ for $0\leq s\leq 1/2$, and $\phi(s) = 0$ for $s > 3/4$. Existence and uniqueness of $u_k$ follows from Lemma \ref{compkol+}. By construction $0\leq u_k\leq u_{k+1}\leq ||\varphi||_{L^\infty(\partial\Omega_k\cap\partial\Omega)}$ in $\Omega_k$ and we deduce, using the maximum principle and the Harnack inequality that
$$\sup_{\Omega_l}|u_k-u_j|\lesssim(u_k-u_j)(A^+_{cl,\Lambda}), \quad \mbox{if $k>j\gg l$},$$
for some constant $c=c(m,\kappa,M)$.
In particular, $u$ can be constructed as the monotone and uniform limit of $\{u_k\}$ as $k\to \infty$  on the closure of $\Omega_l$ for each $l\geq 1$.
Furthermore and similarly, by the maximum principle and the Riesz representation theorem we deduce that
\[
u(Z,t)=\iint_{\partial\Omega}\varphi(\tilde Z,\tilde t)\, \d\omega(Z,t,\tilde Z,\tilde t),
\]
for all $(Z,t)\in\Omega$, where $\{\omega(Z,t,\cdot)\mid (Z,t)\in \Omega\}$ is a  family of Borel regular probability measures on $\partial\Omega$. This finishes the proof of Theorem \ref{thm:dp}. \qed

We next prove the following version of the weak maximum principle in unbounded Lipschitz domains used in the proof of the uniqueness part of Theorem \ref{thm:dp}.
\begin{lemma}\label{maxprinciple_unbounded}
 Let $\Omega$ be a Lipschitz domain. Let $u\in C^2(\Omega)\cap C(\overline{\Omega})$ be such that
\begin{equation}\label{maxpr_unbd_assmp}
    \begin{cases}
    \L u = 0 \quad \text{in } \Omega,\\
    u \leq 0 \quad \text{on } \partial \Omega.
    \end{cases}
\end{equation}
Then $u\leq 0$ in $\Omega$.
\end{lemma}

\begin{proof}
 We can without loss of generality assume that $(0,0)\in\partial\Omega$. Consider $Q_\rho=Q_\rho(0,0)$, $\Omega_\rho=\Omega\cap Q_\rho$, for $\rho>0$. Let $R>0$ and
  $$M:=\max_{\overline{\Omega_{2R}}} u.$$
  We let $\tilde A$ be a  smooth matrix-valued function such that
\[
\tilde A (Z,t) =
\begin{cases}
    A(Z,t),\quad &\mbox{if}\: (Z,t)\in Q_{R}\\
    I_m,\quad &\mbox{if}\:(Z,t)\in \R^{N+1}\backslash Q_{2R}.
\end{cases}
\]
Then, using \eqref{comp} we see that if $R$ is large enough then  $\tilde A\equiv A$ on $\mathbb R^{N+1}$. We fix $R$ so large that this holds. Then $\tilde{A}$ is constant and equal to the $m\times m$ identity matrix outside the cube $Q_{2R}$.  $\tilde A$ defines the operator $\tilde \L$ which coincides with $\L$ on $\mathbb R^{N+1}$ and in particular with $\K$ outside $Q_{2R}$. Let
$\epsilon>0$  be given and assume that $u\geq 2\epsilon$ at some point in $\Omega$. Note that the maximum principle that we are to prove is known to hold for $\K$. Therefore, using Theorem \ref{lem4.5-Kyoto1} and estimates for the fundamental solution, see Lemma \ref{lem_fsolbounds}, we see that there exists $0<\delta=\delta(\epsilon,M)\ll 1$
such that $ u<\epsilon$ in $\Omega \backslash  U_{R,\delta}$, where
\[
U_{R,\delta} := \Omega_{R/\delta}\cap \lbrace (Z,t)=(x,x_m,y,y_m,t)\in\R^{N+1}\mid x_m >\psi(x,y,y_m,t)+\delta\rbrace.
\]
Hence $u\geq 2\epsilon$ at some point in $U_{R,\delta}$. However, applying  Lemma \ref{maxprinciple} we see
$u\leq \epsilon$ in $U_{R,\delta}$, which yields a contradiction. Hence $u \leq 0$ in $\Omega$ and the proof is complete.
\end{proof}

Recall that the operator adjoint to $\L$ is
\begin{equation}
    \L^\ast=\nabla_X\cdot(A(X,Y,t)\nabla_X)-X\cdot\nabla_Y+\partial_t.
\end{equation}
In the case of the adjoint operator $\L^\ast$ we denote the associated Kolmogorov boundary of $\Omega_{r}$
by $\partial_K^\ast\Omega_{r}$.  Lemma \ref{compkol+} and Theorem \ref{thm:dp} then apply to $\L^\ast$ subject to the natural
modifications. In particular, the following is true.
\begin{theorem}\label{thm:dp+} Let $\Omega\subset\mathbb R^{N+1}$ be a Lipschitz domain with constant $M$. Given $\varphi\in
C_0(\partial\Omega)$, there exists
a  unique  weak solution  $u=u_\varphi$, $u\in C(\bar \Omega)$, to the Dirichlet problem
\begin{equation*}
\begin{cases}
	\L^\ast u = 0  &\text{in} \ \Omega, \\
      u = \varphi  & \text{on} \ \partial \Omega.
\end{cases}
\end{equation*}
Furthermore, there exists, for every $(Z, t)=(X,Y,t)\in \Omega$, a unique probability
measure  $\omega^\ast(Z,t,\cdot)$ on $\partial\Omega$ such that
\begin{eqnarray*}
u(Z,t)=\iint_{\partial\Omega}\varphi(\tilde Z,\tilde t)\d \omega^\ast(Z,t,\tilde Z,\tilde t).
\end{eqnarray*}
The measure $\omega^\ast(Z,t,E)$ is referred to as the parabolic measure associated to $\L^\ast$ in $\Omega$ and at $(Z, t)\in \Omega$ and of $E\subset\partial\Omega$.
\end{theorem}

\begin{definition}\label{koldef} Let $(Z,t)\in \Omega$. Then $\omega(Z,t,\cdot)$ is referred to as the
parabolic or Kolmogorov measure
associated to $\L$ relative to $(Z,t)$ and $\Omega$,  and  $\omega^\ast(Z,t,\cdot)$ is referred to as the (adjoint) parabolic or Kolmogorov measure
associated to $\L^\ast$ relative to $(Z,t)$ and $\Omega$.
\end{definition}

\section{Proof of theorem \ref{thm:carleson}}\label{sec7}
To prove Theorem \ref{thm:carleson} it suffices, by a simple argument as in the proof of Theorem 1.1 in \cite{CNP3}, to prove the following proposition.

\begin{proposition}\label{gensalsa-0} $\Omega\subset\mathbb R^{N+1}$ is a Lipschitz domain with
Lipschitz constant $M$. Let $(Z_0,t_0)\in\partial\Omega$, $r>0$, and let  $\Lambda$, $c_0$, $\eta$, $\rho_0$, $\rho_1$, be chosen in accordance with Remark \ref{remnot}.  Assume that $u$ is a non-negative weak solution to $\L u=0$ in  $\Omega_{2r}(Z_0,t_0)$  vanishing continuously on $\Delta_{2r}(Z_0,t_0)$.  Then
\begin{equation*}
\sup_{\Omega_{2\rho/c}(Z_0,t_0)}u\lesssim u(A^+_{\rho,\Lambda}(Z_0,t_0)),
\end{equation*}
for all $0<\rho<\rho_1$.
\end{proposition}
\begin{proof}
Without loss of generality we may assume that $(Z_0,t_0) = (0,0)$, $r=1$. Fix $0<\rho<\rho_1$. By Lemma \ref{coneconditions-}, we have for, any $0<\tau<1$, that $A^+_{(1-\tau)\rho,\Lambda}=A^+_{(1-\tau)\rho,\Lambda}(0,0)$ is a point on an admissible path starting at $A^+_{\rho,\Lambda}=A^+_{\rho,\Lambda}(0,0)$. Moreover, $A^+_{(1-\tau)\rho,\Lambda}$ is an interior point of the propagation set $\mathcal{A}_{A^+_{\rho,\Lambda}}\big(C^+_{\rho,\eta,\Lambda}(0,0)\big)$ as defined in Definition \ref{def:admissible}. Hence, for $\tau\in (0,1/4]$ fixed, there exists $\epsilon=\epsilon(\tau)>0$ such that
\[
    K:=\overline{Q_\epsilon\left( A^+_{(1-\tau)\rho,\Lambda} \right)} \subset\text{Int}\left(\mathcal{A}_{A^+_{\rho,\Lambda}}\big(C^+_{\rho,\eta,\Lambda}(0,0)\big) \right).
\]
By Lemma \ref{t-1} there then exists a constant $c_K=c_K(m,\kappa,M,\tau)$ such that
\[
\sup_K u \leq c_K u(A^+_{\rho,\Lambda}).
\]
Note that we can, due to linearity of $\L$, without loss of generality assume that $c_K u(A^+_{\rho,\Lambda})=1$ and hence
\begin{equation}\label{carlesonpf_supu}
    \sup_K u \leq 1.
\end{equation}
Furthermore, using the continuity of the function $(Z,t)\mapsto A^+_{(1-\tau)\rho,\Lambda}(Z,t)$ we can conclude there exists $\epsilon_1=\epsilon_1(m,M,\tau)$, $0<\epsilon_1<1$ such that
\begin{equation}\label{rho1}
    A^+_{(1-\tau)\rho,\Lambda}(Z,t)\in K,\quad \text{whenever }(Z,t)\in\Delta_{\epsilon_1}.
\end{equation}

 To proceed we fix $0<\theta <1$ to satisfy $0 < \theta < {\bf c}^{-\gamma}$, where $\bf c$ is the constant in \eqref{e-triangular},
and $\gamma$ is as in Lemma \ref{lem4.7}. We then choose $(\epsilon_0,\sigma,\lambda)$ subject to the restriction
\begin{equation} \label{eps0-lambda}
\begin{split}
	\epsilon_0 & < \min\bigg\{ \rho_1,\, \epsilon_1,\, (1-\tau)\rho,\, \frac{1}{c_M^2\mathbf{c}} ,\, \frac{1}{\tilde c(1-\tau)\rho c_0} \bigg\},\\
    \sigma &< \min\bigg\{1,\, \frac{1}{2 c_0 c_M {\bf c}},\, \frac{c_M}{c_0} \bigg\}, \\
	\lambda & > \max\bigg\{ 1,\, c_1(2c_M \mathbf{c} c)^\gamma ,\, c_1 \bigg(\frac{{c_M \mathbf{c}^3}\big(1+ 2c\big) 2 c_0}{\epsilon_0(1-\bf c \theta^{\frac{1}{\gamma}})}\bigg)^{\!\!\gamma} \bigg\},
\end{split}
\end{equation}
where $\tilde c$ is the constant appearing in Lemma \ref{l-coord}, $c_M$ is the constant appearing in \eqref{def.Omega.Delta.frinc}, $\epsilon_1$ is as in \eqref{rho1}, $c_1$ is the constant $c$ in Lemma \ref{lem4.7}, and $c$ is the constant appearing in Lemma \ref{boundaryholder_giventheta}.

Suppose now that there exists a point $(Z_1,t_1) \in \Omega_{\frac{\sigma \epsilon_0}{c_M}}=\Omega_{\frac{\sigma \epsilon_0}{c_M}}(0,0)$ such that
\begin{equation}\label{carlesonpf_assumption}
    u(Z_1,t_1) > \lambda.
\end{equation}
The idea of the argument is to, based on the assumption in \eqref{carlesonpf_assumption}, derive contradiction to the assumption that $u$ is continuous up to the boundary.  Note that \eqref{def.Omega.Delta.frinc} and the choice of $\sigma$ in \eqref{eps0-lambda} imply that
\begin{equation}\label{e.z.1}
    (Z_1,t_1) \in \Omega_{ \frac{\epsilon_0}{c_0}}\cap Q_{\sigma\epsilon_0}.
\end{equation}
To complete the argument by contradiction it suffices to prove the following claim.

\noindent
\textit{Claim: If \eqref{carlesonpf_assumption} holds, then there exists a sequence of points $\lbrace Z_j,t_j \rbrace_{j=1}^\infty \subset\Omega_{\frac{\epsilon_0}{c_0}}$, such that
\[
u(Z_j,t_j)>\lambda\theta^{1-j},
\]
and such that
\begin{equation}\label{carlesonpf_tothebdry}
    d((Z_j,t_j),\Delta_{\epsilon_0})\rightarrow 0,\text{ as }j\rightarrow\infty.
\end{equation}}

\noindent
\textit{Proof of the claim.} We are going to use induction to prove that there exists a sequence $\lbrace Z_j,t_j \rbrace_{j=1}^\infty$ such that for every $j\in\mathbb{N}$, we have
\begin{equation}\label{carlesonpf_induction}
  (Z_j,t_j)\in \Omega_{\frac{\epsilon_0}{c_0}} \text{ and } u(Z_j,t_j)>\lambda\theta^{1-j}.
\end{equation}
To proceed by induction, we first note that by our choice of $(Z_1,t_1)$, \eqref{carlesonpf_induction} holds for $j=1$. Next, assume that \eqref{carlesonpf_induction} holds for $j=k$. Using Lemma \ref{l-coord} and that $\epsilon_0 c_0^{-1}<\rho_1$, we deduce that there exist $(\hat Z_k,\hat t_k)\in\Delta_{\epsilon_0}$ and $\hat\rho_k < \frac{c\epsilon_0}{c_0(1-\tau)\rho}<1$ such that $(Z_k,t_k)=A^+_{\hat\rho_k(1-\tau)\rho,\Lambda}(\hat Z_k,\hat t_k)$. Note that $A^+_{(1-\tau)\rho}(\hat Z_k,\hat t_k)\in K$, by \eqref{rho1} since $\epsilon_0<\epsilon_1$. Using Lemma \ref{lem4.7} we see that
\begin{equation}\label{carlesonpf_ineq1}
    \lambda\theta^{1-k}<u(Z_k,t_k)=u(A^+_{\hat\rho_k(1-\tau)\rho,\Lambda}(\hat Z_k,\hat t_k)) \leq c_1\left( \frac{(1-\tau)\rho}{d} \right)^\gamma u(A^+_{(1-\tau)\rho,\Lambda}(\hat Z_k,\hat t_k)),
\end{equation}
where $d=d((Z_k,t_k),\partial\Omega)$. In particular, using that $d\leq d((Z_k,t_k),(\hat Z_k,\hat t_k))$, \eqref{carlesonpf_ineq1}, \eqref{carlesonpf_supu}, and that $(1-\tau)\rho<1$, we see that
\begin{equation}\label{e.rhok}
    \rho_k := d((Z_k,t_k),(\hat Z_k,\hat t_k)) < c_1^\frac{1}{\gamma}\lambda^{-\frac{1}{\gamma}}\theta^{\frac{k-1}{\gamma}}.
\end{equation}
We now want to apply Lemma \ref{boundaryholder_giventheta}, but to do that we first have to show that
\begin{equation}\label{carlesonpf_ballinbox}
    \partial\Omega \cap Q_{2c\rho_k}(\hat Z_k,\hat t_k) \subset \Delta_{2}.
\end{equation}
By using \eqref{e-triangular}, \eqref{e.rhok}, and \eqref{def.Omega.Delta.frinc}, we see that for any $(Z,t)\in Q_{2c\rho_k}(\hat Z_k,\hat t_k)$ we have
\begin{equation*}
    \begin{split}
        d((Z,t),(0,0)) &\leq {\bf c}(d((Z,t),(\hat Z_k,\hat t_k)) + d((\hat Z_k,\hat t_k),(0,0)))\\
        &\leq {\bf c}(2c\rho_k +c_M\epsilon_0) \leq \frac{2}{c_M},
    \end{split}
\end{equation*}
because of the bounds on $\epsilon_0$ and $\lambda$ in \eqref{eps0-lambda}. This together with \eqref{def.Omega.Delta.frinc} proves \eqref{carlesonpf_ballinbox}. Hence, we can use Lemma \ref{boundaryholder_giventheta} to deduce that
\[
\lambda\theta^{1-k} < u(Z_k,t_k) \leq \sup_{\Omega_{\rho_k}(\hat Z_k,\hat t_k)}u\leq \theta \sup_{\Omega_{2c\rho_k}(\hat Z_k,\hat t_k)}u.
\]
In particular, we see that there exists $(Z_{k+1},t_{k+1})\in \Omega_{2c\rho_k}(\hat Z_k,\hat t_k)$ such that
\[
\lambda \theta^{-k}<u(Z_{k+1},t_{k+1}),
\]
which is the second statement of \eqref{carlesonpf_induction}. We need to check that $(Z_{k+1},t_{k+1})\in\Omega_{\frac{\epsilon_0}{c_0}}$.
By repeatedly using the pseudo-triangular inequality \eqref{e-triangular}, we see that
\begin{equation}\label{e-sum.part.}
  d((Z_{k+1},t_{k+1}),(0,0)) \le {\bf c}\bigg( d((Z_1,t_1),(0,0)) + \sum_{j=1}^k{\bf c}^j d((Z_{j+1},t_{j+1}),(Z_j,t_j)) \bigg).
\end{equation}
Now, notice that
\begin{equation*}
 \begin{split}
    d((Z_{j+1},t_{j+1}),(Z_j,t_j)) &\le {\bf c}(d((\hat Z_j,\hat t_j),(Z_j,t_j))+d((Z_{j+1},t_{j+1}),(\hat Z_j,\hat t_j)))\\
    &\le {\bf c}(\rho_j + 2c\rho_j)\\
    &\le {\bf c}(1+2c)c_1^{\frac{1}{\gamma}}\lambda^{-\frac{1}{\gamma}}\theta^{\frac{j-1}{\gamma}},
 \end{split}
\end{equation*}
where the last inequality follows from \eqref{e.rhok}. Plugging this into \eqref{e-sum.part.}, and recalling \eqref{e.z.1} yields
\begin{equation*}
    d((Z_{k+1},t_{k+1}),(0,0))
     \le {\bf c}  \bigg(\sigma \epsilon_0 + {\bf c} (1 +2c)c_1^\frac{1}{\gamma}\!(\lambda \, \theta)^{-\frac{1}{\gamma}}\sum_{j=1}^{\infty} \big( \theta^\frac{1}{\gamma}{\bf c} \big)^{j}  \bigg) < \frac{\epsilon_0}{c_0c_M},
\end{equation*}
due to the choice of $\sigma$ and $\lambda$ in \eqref{eps0-lambda}. Using the above and \eqref{def.Omega.Delta.frinc}, we see that
\begin{equation*}
    (Z_{k+1},t_{k+1}) \in Q_{\frac{\epsilon_0}{c_0c_M}}\cap \Omega_{2} \subseteq \Omega_{ \frac{\epsilon_0}{c_0}}.
\end{equation*}
Hence \eqref{e.rhok}  holds for $j=k+1$ and hence, by induction, for all $j$. Note that \eqref{e.rhok} implies \eqref{carlesonpf_tothebdry}. Thus proof of the claim is complete. \qed

To complete the proof of Proposition \ref{gensalsa-0} we now see that the claim implies a contradiction as, by continuity, $\lim_{j\rightarrow\infty}u(Z_j,t_j)=0$. Hence,
\eqref{carlesonpf_assumption} can not be true and therefore
\[
\sup_{\Omega_{\frac{\sigma \epsilon_0}{c_M}}} u\leq \lambda c_K u(A^+_{\rho,\Lambda}).
\]
This finishes the proof of the proposition.
\end{proof}

\section{Relations for the Kolmogorov measure and the Green function}\label{sec8}
We define the  Green function for $\L$ associated to $\Omega$, with pole at $(\hat Z,\hat t)\in \Omega$, as
\begin{eqnarray}\label{ghh1-}
G (Z,t,\hat Z,\hat t)=\Gamma(Z,t, \hat Z,\hat t)-\iint_{\partial\Omega}
\Gamma(\tilde Z,\tilde t, \hat Z,\hat t)\d\omega(Z,t,\tilde Z,\tilde t),
\end{eqnarray}
where $\Gamma$ is the fundamental solution to the operator $\L$, see Lemma \ref{lem_fsolbounds}. If we instead consider $(Z,t)\in \Omega$ as fixed, then, for $(\hat Z,\hat t)\in
\Omega$,
\begin{eqnarray}\label{ghh1---}
G (Z,t,\hat Z,\hat t)=\Gamma(Z,t, \hat Z,\hat t)-\iint_{\partial\Omega}
\Gamma(Z,t, \tilde Z,\tilde t)\d\omega^\ast(\hat Z,\hat t,\tilde Z,\tilde t),
\end{eqnarray}
where
$\omega^\ast(\hat Z,\hat t,\cdot)$ is the (adjoint) parabolic measure associated to $\L^\ast$ and defined
relative to $(\hat Z,\hat t)$ and $\Omega$.

\begin{lemma}\label{gensalsa-a}
Let $\Omega\subset\mathbb R^{N+1}$ be a Lipschitz domain with  constant $M$. Let
    $\Lambda$, $c_0$, $\eta$,  be in accordance with Remark \ref{remnot}. Let $(Z_0,t_0)\in\partial\Omega$ and $r>0$.  Then
\begin{eqnarray*}
r^{{\bf q}-2}G (Z,t, A_{r,\Lambda}^+(Z_0,t_0))\lesssim\om (Z,t, \De_{r}(Z_0,t_0)),
\end{eqnarray*}
whenever $(Z,t)\in \Omega$, $t\geq 8 r^2+t_0$.
\end{lemma}

\begin{proof}
We can without loss of generality assume that $(Z_0,t_0)=(0,0)$. Hence we want to prove that
\begin{eqnarray*}
r^{{\bf q}-2}G (Z,t, A_{r,\Lambda}^+)\lesssim\om (Z,t, \De_{r} ),
\end{eqnarray*}
whenever $(Z,t)\in \Omega$, $t\geq 8 r^2$.

Let in the following  $(Z,t)\in \Omega$.  By Definition \ref{ghh1-} we have
\begin{equation}\label{ghh1}
    G (Z,t, A_{r,\Lambda}^+)=\Gamma(Z,t, A_{r,\Lambda}^+)
-\iint_{\partial\Omega}
\Gamma(\tilde Z,\tilde t, A_{r,\Lambda}^+)\d\omega (Z,t,\tilde Z,\tilde t).
\end{equation}
 Obviously, we have that 
\begin{eqnarray}\label{ghh2}
G (Z,t, A_{r,\Lambda}^+)\leq \Gamma(Z,t, A_{r,\Lambda}^+),
\end{eqnarray}
whenever $(Z,t)\in \Omega$.
 Let $0<\delta\ll 1$, be a degree of freedom such that $Q_{\delta r}(A_{r,\Lambda}^+)\subset
\Omega$. We introduce the sets
\begin{equation}\label{ghh3}
\begin{split}
S_1&=\{(Z,t)\in \Omega \mid t=r^2\}\setminus Q_{\delta r/2}(A_{r,\Lambda}^+),\\
S_2&=\{(Z,t)\in \Omega\mid t>r^2\}\cap \partial (Q_{\delta r/2}(A_{r,\Lambda}^+)).
\end{split}
\end{equation}
Using \eqref{fundsolbd} and \eqref{ghh2}  we see that there exists $c=c(m,\kappa,\delta)$ such that
\begin{eqnarray}\label{ghh4}
G (Z,t, A_{r,\Lambda}^+)\leq cr^{-({{\bf q}-2)}},
\end{eqnarray}
when $(Z,t)\in S_2$.

Next, let $v(Z,t) = \om ( Z,t, \De_{r} )$ for $(Z,t) \in \Omega$. Then $\L v =
0$ in $\Omega$, $0 \le v(z,t) \le 1$ in $\Omega$ and $v(Z,t) = 1$ in $\De_{r}$. Hence the
function $u(Z,t) = 1 - v(Z,t)$ satisfies the assumptions of Lemma \ref{boundaryholder_giventheta} and it follows that
\begin{eqnarray}\label{ghh6}
\om ( A_{r/c,\Lambda}^+, \De_{r} )\gtrsim 1.
\end{eqnarray}

We now note that if
we choose $\delta=\delta(m,\kappa,M)$ sufficiently small, then $S_2 \subset \mathcal{B}_{r/c}(A_{r,\Lambda}^+)$
where the constant $c=c(m,\kappa,M,\delta)$ is the one appearing in \eqref{coneset} of Lemma \ref{lem4.7bol+}. Then using  \eqref{ghh6}, and apply inequality \emph{(i)} of
\eqref{coneset} to the function $v(Z,t) = \om ( Z,t, \De_{r} )$, to can conclude that
\begin{eqnarray}\label{ghh8}
\om ( Z,t, \De_{r}  )\gtrsim 1\mbox{ whenever }(Z,t)\in S_2.
\end{eqnarray}
Note that $G (Z,t, A_{r,\Lambda}^+)=0$
if $(Z,t)\in S_1$. Hence, from \eqref{ghh4}, \eqref{ghh8}, and from the maximum principle, it follows that
\begin{eqnarray}\label{ghh5}
r^{{\bf q}-2}G (Z,t, A_{r,\Lambda}^+)\lesssim \om ( Z,t, \De_{r}),
\end{eqnarray}
whenever $(Z,t)\in \Omega\cap\{(Z,t)\mid\ t\geq 8r^2\}$. This completes the proof of the lemma.
\end{proof}

\begin{lemma}\label{gensalsa-aEN}
Let $\Omega\subset\mathbb R^{N+1}$ be a Lipschitz domain with  constant $M$. Let $(Z_0,t_0)\in\partial\Omega$ and $r>0$. Then
\[
\iiint_{\Omega_{r}(Z_0,t_0)}|\nabla_{\tilde X} G(Z,t,\tilde Z,\tilde t)|^2 \d\tilde Z \d\tilde t \lesssim \frac{1}{r^2}\iiint_{\Omega_{2r}(Z_0,t_0)}|G(Z,t,\tilde Z,\tilde t)|^2\d\tilde Z \d\tilde t,
\]
whenever $(Z,t)\in \Omega$, $t\geq 8 r^2+t_0$.
\end{lemma}
\begin{proof} We can without loss of generality assume that $(Z_0,t_0)=(0,0)$ and we simply write $\Omega_r=\Omega_r(Z_0,t_0)$. Let $u(\tilde Z,\tilde t):=G(Z,t,\tilde Z,\tilde t)$. We wish to estimate
\[
\iiint_{\Omega_r}{|\nabla_{\tilde X} u(\tilde Z,\tilde t)|^2}\d\tilde Z \d\tilde t.
\]
Given $\epsilon>0$ we let
\[
\Omega_r^\epsilon := \Omega_r\cap \lbrace (\tilde Z,\tilde t)\mid d((\tilde Z,\tilde t),\partial \Omega)>\epsilon\rbrace
\]
and
\[
S_\epsilon := \Omega_{2r}^{\epsilon/2} \setminus \Omega_{2r}^\epsilon.
\]
Let $\phi_r \in C^\infty_0( Q_{2r})$ be such that $\phi_1 \equiv 1$ in $Q_r$ and let $\eta_\epsilon\in C^\infty_0( \Omega_{2r}^{\epsilon/2})$ be such that $\eta_\epsilon \equiv 1$ in $\Omega_{2r}^\epsilon$. We furthermore choose $\phi_r$ and $\eta_\epsilon$ so that
$$r|\nabla_X\phi_r|+r^2|(X\cdot\nabla_Y-\partial_t)\phi_r|\lesssim 1,$$ and
$$\epsilon|\nabla_X\eta_\epsilon|+\epsilon^2|(X\cdot\nabla_Y-\partial_t)\eta_\epsilon|\lesssim 1.$$
Note $\eta_\epsilon$ can easily be constructed by considering $(\R^{N+1},d,\d Z \d t)$ as a homogeneous space and proceeding through a partition of unity associated to Whitney decompositions of $\Omega$, see below. Using this notation we note that
\begin{align*}
\iiint_{\Omega_r}{|\nabla_{\tilde X} u|^2}\d\tilde Z \d\tilde t & \leq \lim_{\epsilon\rightarrow 0} \iiint_{\Omega_{2r}}{|\nabla_{\tilde X} u(\tilde Z,\tilde t)|^2}(\eta_\epsilon^2\phi_r^2)\d\tilde Z \d\tilde t,
\end{align*}
and we compute
\begin{align*}
 \iiint_{\Omega_{2r}}{|\nabla_{\tilde X} u(\tilde Z,\tilde t)|^2}(\eta_\epsilon^2\phi_r^2)\d\tilde Z \d\tilde t&\lsim \iiint_{\Omega_{2r}}{(A\nabla_{\tilde X} u\cdot \nabla_{\tilde X} u) \eta_\epsilon^2\phi_r^2}\d\tilde Z \d\tilde t\\
&=\iiint_{\Omega_{2r}}{A\nabla_{\tilde X} u\cdot \nabla_{\tilde X} (u \eta_\epsilon^2\phi_r^2)}\d\tilde Z \d\tilde t\notag\\
& - \iiint_{\Omega_{2r}}{A\nabla_{\tilde X} u\cdot u\nabla_{\tilde X}  (\eta_\epsilon^2\phi_r^2)}\d\tilde Z \d\tilde t=:I_{1,\epsilon} + I_{2,\epsilon}.
\end{align*}
We first consider $I_{1,\epsilon}$ and we note that $u \eta_\epsilon^2\phi_r^2$ is a valid test function in the weak formulation for  $\L^\ast u=0$ in $\Omega_r$, as
$A$ is assumed to be smooth. Therefore, using that $\L^\ast u=0$ in $\Omega_r$, and by the properties of $\eta_\epsilon$ and $\phi_r$, we have
\begin{equation}\label{I2e}
\begin{split}
    |I_{1,\epsilon}| &\lesssim \iiint_{\Omega_{2r}}{u^2|({\tilde X}\cdot\nabla_{\tilde Y}-\partial_{\tilde t}) (\eta_\epsilon^2\phi_r^2)|\d\tilde Z \d\tilde t}\\
    &\lsim \frac{1}{\epsilon^2}\iiint_{S_\epsilon}{u^2\phi_r^2\d\tilde Z \d\tilde t} + \frac{1}{r^2}\iiint_{\Omega_{2r}}{u^2\d\tilde Z \d\tilde t}.
\end{split}
\end{equation}
Let
\[
J_\epsilon:= \frac{1}{\epsilon^2}\iiint_{S_\epsilon}{u^2\phi_r^2 \d\tilde Z \d\tilde t}.
\]
Now, applying Theorem \ref{lem4.5-Kyoto1}, the adjoint version of Theorem \ref{thm:carleson}, and using \eqref{fundsolbd} and \eqref{ghh2},  we see that
\begin{align*}
    J_\epsilon&\lsim \frac{1}{\epsilon^2}\iiint_{\Omega_{2r}}|G(Z,t,\tilde Z,\tilde t)|^2 \chi_{S_\epsilon}\d\tilde Z \d\tilde t\\
    &\lsim \frac{1}{\epsilon^2}\left(\frac{\epsilon}{r}\right)^\alpha \sup_{(Z_0,t_0)\in\Delta_{2r}} G(Z,t,A^-_{r,\Lambda}(Z_0,t_0))\iiint_{S_\epsilon}G(Z,t,\tilde Z,\tilde t)\d\tilde Z \d\tilde t\\
    &\lsim \frac{1}{\epsilon^2}\left(\frac{\epsilon}{r}\right)^\alpha r^{-{\bf q}+2}\iiint_{S_\epsilon}G(Z,t,\tilde Z,\tilde t)\d\tilde Z \d\tilde t.
\end{align*}
Now, using that $(\R^{N+1},d,\d Z \d t)$ is a homogeneous space, let $\mathcal{W}_\epsilon = \lbrace I_j \rbrace_j$ be a covering of $S_\epsilon$ with Whitney cubes $\{I_j\}$ such that $|I_j| \lsim \epsilon^{{\bf q}}$. Let
$(Z_j,t_j)\in\partial\Omega$ be a point on $\partial\Omega$ closest to $I_j$ as measured by $d$.  For $(\tilde Z,\tilde t)\in I_j$ we have, by Theorem \ref{thm:carleson} and Lemma \ref{gensalsa-a}, that
\[
G(Z,t,\tilde Z,\tilde t) \lsim G(Z,t,A^+_{c\epsilon,\Lambda}(Z_j,t_j))\lsim \epsilon^{-{\bf q}+2}\omega (Z,t,\Delta_{c\epsilon}(Z_j,t_j)),
\]
where $c=c(m,\kappa,M)$, $1\leq c<\infty$. Hence
\begin{align*}
\iiint_{S_\epsilon}G(Z,t,\tilde Z,\tilde t)\d\tilde Z \d\tilde t &\lsim \sum_{I_j\in\mathcal{W}_\epsilon}|I_j|\epsilon^{-{\bf q}+2}\omega (Z,t,\Delta_{c\epsilon} (Z_j,t_j))\\
&\lsim \sum_{I_j\in\mathcal{W}_\epsilon}\epsilon^2\omega (Z,t,\Delta_{c\epsilon} (Z_j,t_j))\lsim \epsilon^2.
\end{align*}
We can conclude that
\[
J_\epsilon \lsim \left(\frac{\epsilon}{r}\right)^\alpha r^{-{\bf q}+2} \rightarrow 0, \text{ as }\epsilon\rightarrow 0,
\]
and that the estimate of $I_{1,\epsilon}$ is complete.

Now we turn our attention to the term $I_{2,\epsilon}$ and in this case we immediately see using Cauchy Schwarz that
\begin{align*}
I_{2,\epsilon}&\lesssim \tilde\epsilon\iiint_{\Omega_{2r}}{|\nabla_{\tilde X} u(\tilde Z,\tilde t)|^2}(\eta_\epsilon^2\phi_r^2)\d\tilde Z \d\tilde t\\
&+\tilde\epsilon^{-1}\biggl (\frac{1}{\epsilon^2}\iiint_{S_\epsilon}{u^2\phi_r^2\d\tilde Z \d\tilde t} + \frac{1}{r^2}\iiint_{\Omega_{2r}}{u^2\d\tilde Z \d\tilde t}\biggr ),
\end{align*}
where $\tilde\epsilon$ is a degree of freedom. Hence we can reuse the estimates starting from \eqref{I2e} to complete the estimate of $I_{2,\epsilon}$ and hence to complete the proof of the lemma.
\end{proof}

\begin{lemma}\label{Representa} Given $\theta\in C_0^\infty(\mathbb R^{N+1})$ we have the representation
formulas
  \begin{equation*}
  \begin{split}
      \theta(Z,t)&=\iint_{\partial\Omega}\theta(\tilde Z,\tilde t)\d\omega(Z,t,\tilde Z,\tilde t)-\iiint_{\Omega} A\nabla_{\tilde X}G (Z,t,\tilde Z,\tilde t )\nabla_{\tilde X}\theta(\tilde Z,\tilde t)
 \d\tilde Z \d\tilde t\notag\\
 &+\iiint_{\Omega} G (Z,t,\tilde Z,\tilde t )(\tilde X\cdot\nabla_{\tilde Y}-\partial_{\tilde t})\theta(\tilde Z,\tilde t)\d\tilde Z \d\tilde t,\notag\\
  \theta(\hat Z,\hat t)&=\iint_{\partial\Omega}\theta(\tilde Z,\tilde t)\d\omega^\ast(\hat Z,\hat t,\tilde Z,\tilde t)-\iiint_{\Omega} A\nabla_{X} G(Z,t,\hat Z,\hat t )\nabla_{X} \theta(Z,t) \d Z \d t\notag\\
  &+\iiint_{\Omega} G(Z,t,\hat Z,\hat t )(-X\cdot\nabla_{Y}+\partial_{t})\theta(Z,t) \d Z \d t,
  \end{split}
\end{equation*}
whenever $(Z,t),\ (\hat Z,\hat t)\in\Omega$.
\end{lemma}
\begin{proof}
We introduce $\hat G(Z,t,\tilde Z,\tilde t) := \Gamma(Z,t,\tilde Z,\tilde t) - V(Z,t,\tilde Z,\tilde t)$, where
\[
V(Z,t,\tilde Z,\tilde t) =
\begin{cases}
\iint_{\partial \Omega}\Gamma(\hat Z,\hat t,\tilde Z,\tilde t) \d \omega(Z,t,\hat Z,\hat t),\quad&\mbox{if $(\tilde Z,\tilde t)\in\Omega$}\\
\Gamma(Z,t,\tilde Z,\tilde t),\quad&\mbox{if $(\tilde Z,\tilde t)\in \R^{N+1}\backslash\Omega$}.
\end{cases}
\]
Using the maximum principle we see that
\[
\Gamma(Z,t,\tilde Z,\tilde t) =
\iint_{\partial\Omega} \Gamma(\hat Z,\hat t,\tilde Z,\tilde t)\d\omega(Z,t,\hat Z,\hat t),
\]
whenever $(Z,t)\in\Omega$ and $(\tilde Z,\tilde t)\in \R^{N+1}\backslash\Omega$.
Then, using Lemma \ref{gensalsa-aEN} we see that
\begin{equation*}
\begin{split}
    &\iiint_{\Omega} A\nabla_{\tilde X}G (Z,t,\tilde Z,\tilde t )\nabla_{\tilde X}\theta(\tilde Z,\tilde t)
        -G (Z,t,\tilde Z,\tilde t )(\tilde X\cdot\nabla_{\tilde Y}-\partial_{\tilde t})\theta(\tilde Z,\tilde t)\d\tilde Z \d\tilde t\\
    &=\iiint_{\R^{N+1}} A\nabla_{\tilde X}\hat G(Z,t,\tilde Z,\tilde t)\nabla_{\tilde X}\theta(\tilde Z,\tilde t)
        -\hat G(Z,t,\tilde Z,\tilde t)(\tilde X\cdot\nabla_{\tilde Y}-\partial_{\tilde t})\theta(\tilde Z,\tilde t)\d\tilde Z \d\tilde t.
\end{split}
\end{equation*}
Using the definition of $\hat G$, properties of $\Gamma$, and Fubini's theorem we see that the integral in the last display equals
\begin{equation*}
    \begin{split}
    -\theta(Z,t) + \iint_{\partial\Omega}\bigg( \iiint_{\R^{N+1}} &A\nabla_{\tilde X}\Gamma(\hat Z,\hat t,\tilde Z,\tilde t)\nabla_{\tilde X}\theta(\tilde Z,\tilde t)\\
    &- \Gamma(\hat Z,\hat t,\tilde Z,\tilde t)(\tilde X\cdot\nabla_{\tilde Y}-\partial_{\tilde t})\theta(\tilde Z,\tilde t)\d\tilde Z \d\tilde t \bigg)\d \omega(Z,t,\hat Z,\hat t)
    \end{split}
\end{equation*}
which, again by properties of $\Gamma$, equals
\begin{equation*}
    -\theta(Z,t) + \iint_{\partial\Omega}\theta(\hat Z,\hat t)\d\omega(Z,t,\hat Z,\hat t).
\end{equation*}
This finishes the proof of the first formula. The proof of the second formula is analogous.
\end{proof}

Using Lemma \ref{Representa} we see, in particular, that
\begin{equation}\label{1.5}
\begin{split}
\iint_{\partial\Omega}\theta(\tilde Z,\tilde t)\d\omega(Z,t,\tilde Z,\tilde t)&=\iiint_{\Omega} A\nabla_{\tilde X}G (Z,t,\tilde Z,\tilde t )\nabla_{\tilde X}\theta(\tilde Z,\tilde t)
 \d\tilde Z \d\tilde t\notag\\
 &-\iiint_{\Omega} G (Z,t,\tilde Z,\tilde t )(\tilde X\cdot\nabla_{\tilde Y}-\partial_{\tilde t})\theta(\tilde Z,\tilde t)\d\tilde Z \d\tilde t,\notag\\
 \iint_{\partial\Omega}\theta(\tilde Z,\tilde t)\d\omega^\ast(\hat Z,\hat t,\tilde Z,\tilde t)&=\iiint_{\Omega} A\nabla_{X} G(Z,t,\hat Z,\hat t )\nabla_{X} \theta(Z,t) \d Z \d t\notag\\
 &-\iiint_{\Omega} G(Z,t,\hat Z,\hat t )(-X\cdot\nabla_{Y}+\partial_{t})\theta(Z,t) \d Z \d t,
\end{split}
 \end{equation}
 whenever $ \theta \in
C_0^\infty ( \mathbb R^{N + 1 } \sem \{ ( Z,t ) \} )$ and $ \theta \in
C_0^\infty ( \mathbb R^{N + 1 } \sem \{ ( \hat Z,\hat t ) \} ) $, respectively.

\begin{lemma}\label{gensalsa-asecondpart}
Let $\Omega\subset\mathbb R^{N+1}$ be a Lipschitz domain with  constant $M$. Let
    $\Lambda$, $c_0$, $\eta$,  be in accordance with Remark \ref{remnot}. Let $(Z_0,t_0)\in\partial\Omega$ and $r>0$.  There exists
$c=c(m,\kappa,M)$, $1\leq c<\infty$, such that
\begin{eqnarray*}
\om ( Z,t, \De_{r/c}(Z_0,t_0))\lesssim r^{{\bf q}-2}G (Z,t,
A_{r,\Lambda}^-(Z_0,t_0)),
\end{eqnarray*}
whenever $(Z,t)\in \Omega$, $t\geq 8 r^2+t_0$.
\end{lemma}
\begin{proof} We can without loss of generality assume that $(Z_0,t_0)=(0,0)$. Hence we want to prove that there exists
$c=c(m,\kappa,M)$, $1\leq c<\infty$, such that
\begin{eqnarray*}
\om ( Z,t, \De_{r/c})\lesssim r^{{\bf q}-2}G (Z,t,
A_{r,\Lambda}^-),
\end{eqnarray*}
whenever $(Z,t)\in \Omega$, $t\geq 8 r^2$.

Let $(Z,t)\in \Omega\cap\{(Z,t)\mid\ t\geq 8r^2\}$ and let $\delta$, $0<\delta\ll 1$, be a degree of freedom to be chosen.
Given $\delta$, we let
 $\theta\in C^\infty(\mathbb R^{N+1})$ be such that $\theta\equiv 1$ on the set  $Q_{\delta r/2}$, $\theta\equiv 0$ on the complement of $Q_{3\delta r/4}$, and such that $$(\delta r)|\nabla_X\theta|+(\delta r)^2|(X\cdot\nabla_Y-\partial_t)\theta|\lesssim 1.$$  By the definition of $\theta$ we have
 \begin{eqnarray}\label{ghh9}
\om ( Z,t, \De_{\delta r/2} )\leq\iint_{\partial\Omega}\theta(\tilde Z,\tilde t)\d\omega (Z,t,\tilde
Z,\tilde t).
\end{eqnarray}
By  the representation formula of Lemma \ref{Representa},
\begin{align*}\label{ghh10}
\theta(Z,t)&=\iint_{\partial\Omega}\theta(\tilde Z,\tilde t)d\omega(Z,t,\tilde Z,\tilde t)-\iiint A\nabla_{\tilde X}G (Z,t,\tilde Z,\tilde t )\nabla_{\tilde X}\theta(\tilde Z,\tilde t)
 \d\tilde Z \d\tilde t\notag\\
 &+\iiint_\Omega G (Z,t,\tilde Z,\tilde t )(\tilde X\cdot\nabla_{\tilde Y}-\partial_{\tilde t})\theta(\tilde Z,\tilde t)\d\tilde Z \d\tilde t.
\end{align*}
By construction $\theta(Z,t)=0$ whenever $(Z,t)\in \Omega\cap\{(Z,t)\mid  t\geq 8 r^2\}$, and hence through the last two displays we deduce that
\begin{equation}\label{ghh13}
\begin{split}
\om ( Z,t, \De_{\delta r/2} )&\leq \bigg{|}\iiint A\nabla_{\tilde X}G (Z,t,\tilde Z,\tilde t )\nabla_{\tilde X}\theta(\tilde Z,\tilde t)
 \d\tilde Z \d\tilde t\bigg{|}\\
 &+\bigg{|}\iiint_\Omega G (Z,t,\tilde Z,\tilde t )(\tilde X\cdot\nabla_{\tilde Y}-\partial_{\tilde t})\theta(\tilde Z,\tilde t)\d\tilde Z \d\tilde t\bigg{|}.
 \end{split}
\end{equation}
In particular, Cauchy-Schwarz, and  \eqref{ghh13} yields
\begin{equation}\label{est14}
\begin{split}
    \om ( Z,t, \De_{\delta r/2} )&\lsim (\delta r)^{\frac{{\bf q}-2}{2}}\left(\iiint_{\Omega_{\delta  r}}{|\nabla_{\tilde X} G(Z,t,\tilde Z,\tilde t)|^2\d\tilde Z \d\tilde t}\right)^{\frac 1 2}\\
    &+(\delta r)^{\frac{{\bf q}-4}{2}}\left(\iiint_{\Omega_{\delta  r}}{|G(Z,t,\tilde Z,\tilde t)|^2\d\tilde Z \d\tilde t}\right)^{\frac 1 2}.
\end{split}
\end{equation}
    Using Lemma \ref{gensalsa-aEN} with $r$ replaced by $\delta r$ we obtain
\[
\iiint_{\Omega_{\delta r}}|\nabla_{\tilde X} G(Z,t,\tilde Z,\tilde t)|^2 \d\tilde Z \d\tilde t \lesssim \frac{1}{(\delta r)^2}\iiint_{\Omega_{2\delta r}}|G(Z,t,\tilde Z,\tilde t)|^2\d\tilde Z \d\tilde t.
\]
Using this inequality, and (\ref{est14}), we have
\begin{align}
    \om ( Z,t, \De_{\delta r/2} )&\lsim (\delta r)^{\frac{{\bf q}-4}{2}}\left(\iiint_{\Omega_{\delta  r}}{|G(Z,t,\tilde Z,\tilde t)|^2\d\tilde Z \d\tilde t}\right)^{\frac 1 2}.
    \end{align}
    To complete the proof  we now use  the adjoint version of Theorem \ref{thm:carleson} and deduce
\[
 \om ( Z,t, \De_{\delta r/2} ) \lesssim (\delta  r)^{{\bf q}-2}G(Z,t,A^-_{c\delta  r,\Lambda}),
\]
and this completes the proof of the lemma if we choose $\delta=\delta(m,\kappa,M)$ small.
\end{proof}

\section{A Weak comparison principle and its consequences}\label{sec9}

In this section we prove the following lemma.
\begin{lemma}\label{lem:compprinciple} Let $\Omega\subset\mathbb R^{N+1}$ be a Lipschitz domain with  constant $M$. Let
$(Z_0,t_0)\in\partial\Omega$, $r>0$ and let $\Lambda$, $c_0$, $\eta$, $\rho_0=r/c_0$, $\rho_1=\rho_0/c_0$ be in accordance with Remark \ref{remnot}.
Assume that $u$ and $v$ are non-negative weak solutions to $\L u=0$ in $\Omega_{2r}(Z_0,t_0)$ and that $u$ and $v$ vanish continuously on $\Delta_{2r}(Z_0,t_0)$. Then
there exists $c=c(m,\kappa,M)$, $1\leq c<\infty$, such that
\begin{equation}\label{compprinciple}
    \frac{v(A^-_{\rho,\Lambda}(Z_0,t_0))}{u(A^+_{\rho,\Lambda}(Z_0,t_0))} \lesssim \frac{v(Z,t)}{u(Z,t)} \lesssim\frac{v(A^+_{\rho,\Lambda}(Z_0,t_0))}{u(A^-_{\rho,\Lambda}(Z_0,t_0))},
\end{equation}
whenever $(Z,t)\in \Omega_{\rho/c}(Z_0,t_0)$ and $0<\rho\leq\rho_1$.
\end{lemma}

\begin{proof} Again we can  without loss of generality assume that $(Z_0,t_0)=(0,0)$ and we let $\Omega_{2r}=\Omega_{2r}(Z_0,t_0)$, $\De_{r}=\De_{r}(Z_0,t_0)$, and $A_{\rho,\Lambda}^\pm=A_{\rho,\Lambda}^\pm(Z_0,t_0)$.

Let $0<\varepsilon\ll1$ be a  degree of freedom only depending on $m$ and $M$ and consider the set
\[
\Delta_{6\varepsilon\rho}\backslash\Delta_{4\varepsilon\rho}.
\]

\noindent
\textit{Claim 1: There exist  $0<\delta\ll1$, only depending on $m$ and $M$, and a set of points $\lbrace (Z_i,t_i) \rbrace_{i=1}^L$ with $(Z_i,t_i)\in \Delta_{6\varepsilon\rho}\backslash\Delta_{4\varepsilon\rho}$ such that
\begin{equation}
    \lbrace \Delta_{\delta\varepsilon\rho}(Z_i,t_i) \rbrace_{i=1}^L\:\text{is a covering of}\:\Delta_{6\varepsilon\rho}\backslash\Delta_{4\varepsilon\rho},
\end{equation}
and such that
\begin{equation}
    \Delta_{\delta\varepsilon\rho/\lambda}(Z_i,t_i)\cap\Delta_{\delta\varepsilon\rho/\lambda}(Z_j,t_j) = \emptyset \: \text{whenever}\: i\neq j,
\end{equation}
for some $\lambda(m,M,{\bf c})\geq 1$, where ${\bf c}$ is the constant from (\ref{e-triangular}). Furthermore, this covering can be constructed so that
\begin{equation}\label{covermeasbound}
    \sum_{i=1}^L{\omega (Z,t,\Delta_{\delta\varepsilon\rho}(Z_i,t_i))} \gtrsim 1,
\end{equation}
whenever
\begin{equation}
    (Z,t) \in \partial \Omega_{5\varepsilon\rho} \cap \overline{\lbrace (Z,t)\in\Omega_{2r}\mid d((Z,t),\Delta_{2r})\leq \delta^2\varepsilon\rho \rbrace}.
\end{equation}}

\noindent
\textit{Proof of the claim. } The claim follows immediately from a standard Vitali covering argument, Lemma \ref{boundaryholder_giventheta} and the same argument as in the proof of (\ref{ghh6}).
\qed

Based on  $\lbrace (Z_i,t_i) \rbrace_{i=1}^L$ we introduce the function
\begin{equation}
    \Psi(Z,t) := \sum_{i=1}^L\omega (Z,t,\Delta_{\delta\varepsilon\rho}(Z_i,t_i))+(\varepsilon\rho)^{{\bf q}-2}G(Z,t,A^-_{K\varepsilon\rho,\Lambda}),
\end{equation}
where $K\gg 1$ is an additional degree of freedom to be chosen. We now partition the boundary of $\Omega_{5\varepsilon\rho}$ into the sets
\begin{equation}
    \begin{split}
        \Sigma_1 &:= \partial \Omega_{5\varepsilon\rho} \cap \overline{\lbrace (Z,t)\in\Omega_{2r}\mid d((Z,t),\Delta_{2r})\leq \delta^2\varepsilon\rho \rbrace},\\
        \Sigma_2 &:= \partial \Omega_{5\varepsilon\rho} \cap \overline{\lbrace (Z,t)\in\Omega_{2r}\mid d((Z,t),\Delta_{2r}) > \delta^2\varepsilon\rho \rbrace}.
    \end{split}
\end{equation}
Note that $\Sigma_1$ is the part of the boundary of $\Omega_{5\varepsilon\rho}$ which is close to $\partial \Omega_{5\varepsilon\rho}\cap \partial\Omega$ and $\Sigma_2$ is the remaining part.

By the construction of the covering above, in particular using (\ref{covermeasbound}), we  see that
\begin{equation}
    \Psi(Z,t) \gtrsim 1,
\end{equation}
whenever $(Z,t)\in \Sigma_1$. To estimate $\Psi$ on $\Sigma_2$ we prove the following claim.

\noindent
\textit{Claim 2: There exist $K\gg 1$, depending at most on $m$, $\kappa$ and $M$, such that
\begin{equation}\label{compprinpf_claim2}
    (\varepsilon\rho)^{{\bf q}-2}G(Z,t,A^-_{K\varepsilon\rho,\Lambda})\geq c^{-1}\:\text{for all}\:(Z,t)\in \Sigma_2.
\end{equation}}

\noindent
\textit{Proof of the claim.} Suppose $(Z,t)\in\Sigma_2$. It follows from Lemma \ref{lem4.7} that there exists $K\gg 1$, depending at most on $m$, $\kappa$ and $M$, such that
\begin{equation}
    (\varepsilon\rho)^{{\bf q}-2}G(Z,t,A^-_{K\varepsilon\rho,\Lambda}) \gtrsim(\varepsilon\rho)^{{\bf q}-2}G(A^-_{K\varepsilon\rho/\lambda,\Lambda},A^-_{K\varepsilon\rho,\Lambda}).
\end{equation}
Therefore the claim follows if we can show that
\begin{equation}\label{greenrefestimate}
    (\varepsilon\rho)^{{\bf q}-2}G(A^-_{K\varepsilon\rho/\lambda,\Lambda},A^-_{K\varepsilon\rho,\Lambda})\gtrsim 1.
\end{equation}
To prove this, let $\tilde G$ be the Green function for the set $Q_{4\varepsilon\rho}(A^-_{k\varepsilon\rho,\Lambda})$. Then by scaling, and the continuity of $G$ close to $A^-_{K\epsilon\rho,\Lambda}$, we can conclude that there exists $\tilde\eta=\tilde\eta(m)$, $0<\tilde \eta\ll 1$, such that
\[
(\varepsilon\rho)^{{\bf q}-2}\tilde G(A^-_{K(1-\tilde\eta)\varepsilon\rho,\Lambda},A^-_{K\varepsilon\rho,\Lambda})\gtrsim 1.
\]
Therefore, by the maximum principle,
\[
(\varepsilon\rho)^{{\bf q}-2} G(A^-_{K(1-\tilde\eta)\varepsilon\rho,\Lambda},A^-_{K\varepsilon\rho,\Lambda})\gtrsim 1.
\]
Finally, by the Harnack inequality, see Lemma \ref{lem4.7rho}, we obtain
\begin{equation}
    (\varepsilon\rho)^{{\bf q}-2}G(A^-_{K\varepsilon\rho/\lambda,\Lambda},A^-_{K\varepsilon\rho,\Lambda})\gtrsim
    (\varepsilon\rho)^{{\bf q}-2} G(A^-_{K(1-\tilde\eta)\varepsilon\rho,\Lambda},A^-_{K\varepsilon\rho,\Lambda})
    \gtrsim 1,
\end{equation}
and the claim is proved.\qed

Using  (\ref{compprinpf_claim2}) we can conclude $\Psi(Z,t)\geq c^{-1}$, when $(Z,t)\in\partial \Omega_{5\varepsilon\rho}$. Furthermore, applying Theorem \ref{thm:carleson} we also see that we can without loss of generality assume that $K$ is such that
\begin{equation}\label{compprinpf_vest1}
    v(Z,t) \lesssim v(A^+_{K\varepsilon\rho,\Lambda}),
\end{equation}
when $(Z,t)\in\Omega_{6\varepsilon\rho}$. Hence, $v$ is  assumed to vanish  on $\Delta_{2r}$ it follows from the maximum principle and (\ref{compprinpf_vest1}) that
\begin{equation}\label{compprinpf_vest}
    v(Z,t) \lesssim v(A^+_{K\varepsilon\rho,\Lambda})\Psi(Z,t),
\end{equation}
whenever $(Z,t)\in\Omega_{5\varepsilon\rho}$.

Having established an estimate on $v$ from above, we now need to establish an estimate on $u$ from below. We proceed by arguing similarly to the proof of Lemma \ref{gensalsa-a}. We define the two sets
\begin{equation}
    \begin{split}
        S_1 &:= \lbrace (Z,t)\in\Omega_{2r}\mid t=-(k\varepsilon\rho)^2 \rbrace \backslash Q_{\delta\rho/2}(A^-_{K\varepsilon\rho,\Lambda}),\\
        S_2 &:= \lbrace (Z,t)\in\Omega_{2r}\mid t>-(k\varepsilon\rho)^2 \rbrace \cap \partial Q_{\delta\rho/2}(A^-_{K\varepsilon\rho,\Lambda}),
    \end{split}
\end{equation}
and as in the proof of Lemma \ref{gensalsa-a}, see \eqref{ghh4}, we obtain
\begin{equation}\label{greenbound}
    c^{-1}(\varepsilon\rho)^{{\bf q}-2}G(Z,t,A^-_{K\varepsilon\rho,\Lambda})\leq 1,\:\text{ for all}\:(Z,t)\in \Omega_{5\varepsilon\rho}.
\end{equation}
Note that we can choose $\delta\ll 1$ so that the closure of $Q_{\delta\rho/2}(A^-_{K\epsilon\rho,\Lambda})$ is contained in the cone $C^-_{K\epsilon\rho,\eta,\Lambda}(0,0)$. Using this,  continuity of $u$, and the maximum principle, we have
\begin{equation}\label{compprinpf_uest1}
    u(Z,t) \gtrsim(\varepsilon\rho)^{{\bf q}-2}G(Z,t,A^-_{K\varepsilon\rho,\Lambda})u(A^-_{K\varepsilon\rho,\Lambda}),
\end{equation}
when $(Z,t)\in\Omega_{5\varepsilon\rho}$.

\noindent
\textit{Claim 3: \begin{equation}\label{compprinpf_claim3}
    (\varepsilon\rho)^{{\bf q}-2}G(Z,t,A^-_{K\varepsilon\rho,\Lambda})\gtrsim \Psi(Z,t),
\end{equation}
whenever $(Z,t)\in\partial \Omega_{\varepsilon\rho}$.}

\noindent
\textit{Proof of the claim.}  By arguing as in the proof of Lemma \ref{gensalsa-asecondpart}  we deduce that
\begin{equation}\label{compprinpf_claim3pf1}
    \omega (Z,t,\Delta_{\delta\varepsilon\rho}(Z_i,t_i))\lesssim (\varepsilon\rho)^{{\bf q}-2}G(Z,t,A^-_{K\varepsilon\rho,\Lambda}),
\end{equation}
whenever $(Z,t)\in\partial \Omega_{\varepsilon\rho}$ and for all $i=1,....,L$. This yields (\ref{compprinpf_claim3}) by construction. \qed

Combining Claim 3 and (\ref{compprinpf_uest1}) we first obtain the estimate
\begin{equation}
    u(Z,t)\gtrsim u(A^-_{K\varepsilon\rho,\Lambda}) \Psi(Z,t),
\end{equation}
when $(Z,t)\in\Omega_{\varepsilon\rho}$. Using this in combination with \eqref{compprinpf_vest} we can conclude that
\begin{equation}
    u(Z,t)\gtrsim \frac{u(A^-_{K\varepsilon\rho,\Lambda})}{v(A^+_{K\varepsilon\rho,\Lambda})} v(Z,t),
\end{equation}
Finally, putting $\varepsilon=\frac{1}{K}$ completes the proof of the lemma in one direction. The proof on the other direction follows by simply interchanging $u$ and $v$.
\end{proof}

\begin{lemma}\label{T:backprel} Let $\Omega\subset\mathbb R^{N+1}$ be a Lipschitz domain with  constant $M$ and assume in addition \eqref{struct}. Let
$(Z_0,t_0)\in\partial\Omega$, $r>0$ and let $\Lambda$, $c_0$, $\eta$, $\rho_0=r/c_0$, $\rho_1=\rho_0/c_0$ be in accordance with Remark \ref{remnot}.
Assume that $u$ is a non-negative weak solutions to $\L u=0$ in $\Omega_{2r}(Z_0,t_0)$ and that $u$  vanishes continuously on $\Delta_{2r}(Z_0,t_0)$. Define
$y_m^0$  through $(Z_0,t_0)=(x^0,x_m^0,y^0,y_m^0,t^0)$. Then
there exists $c=c(m,\kappa,M)$, $1\leq c<\infty$, such that
 \begin{eqnarray}\label{weakcompuu2}
 \frac {u(A_{\rho_0,\Lambda}^-(Z_0,t_0))}{u(A_{\rho_0,\Lambda}^+(Z_0,t_0))}\lesssim \frac
{u(x,x_m,y,y_m^0,t)}{u(x,x_m,y,y_m,t)}\lesssim\frac {u(A_{\rho_0,\Lambda}^+(Z_0,t_0))}{u(A_{\rho_0,\Lambda}^-(Z_0,t_0))},
 \end{eqnarray}
 whenever $(x,x_m,y,y_m,t)\in \Omega_{\rho_1/c}(Z_0,t_0)$.
\end{lemma}

\begin{proof} Again we can  without loss of generality assume that $(Z_0,t_0)=(0,0)$ and we let $\Omega_{2r}=\Omega_{2r}(Z_0,t_0)$, $\De_{r}=\De_{r}(Z_0,t_0)$, and $A_{\rho,\Lambda}^\pm=A_{\rho,\Lambda}^\pm(Z_0,t_0)$.

We fix $0<\varepsilon_0\ll 1$ and consider $\tilde u(Z,t)= u(x,x_m,y,y_m+\varepsilon^3 r^3,t)$, with $\varepsilon\in(-\varepsilon_0,\varepsilon_0)$. Put $$\tilde r=\frac{r(1-\varepsilon)}{4}.$$
Then $\tilde u$ is a solution to $\L u=0$ in $\Omega_{2\tilde r}$, and since $\Omega_{2r}$ is $y_m$-independent $\tilde u$ vanishes continuously on $\Delta_{2\tilde r}$. Using this and that {$\Omega_{2\tilde r}\subset \Omega_{2r}$} we can apply Lemma \ref{lem:compprinciple} to the functions $u$ and $v=\tilde u$ in $\Omega_{2\tilde r}$ to obtain
\begin{equation}\label{cor1}
     \frac {\tilde u(A_{\tilde \rho_1,\Lambda}^-)}{u(A_{\tilde \rho_1,\Lambda}^+)}\lesssim \frac
{\tilde u(Z,t)}{u(Z,t)}\lesssim\frac {\tilde u(A_{\tilde \rho_1,\Lambda}^+)}{u(A_{\tilde \rho_1,\Lambda}^-)},
\end{equation}
for all $(Z,t)\in \Omega_{\tilde \rho_1/c}$, $\tilde \rho_1=\tilde \rho_0/c_0=\tilde r/c_0^2$. As, by definition, $\tilde \rho_1 \leq \tilde \rho_0$, it follows immediately from Lemma \ref{lem4.7bol+} that
\begin{equation}
\begin{split}
    u(A^+_{\tilde \rho_1,\Lambda}) \lesssim c u(A^+_{\tilde \rho_0,\Lambda}),\ u(A^-_{\tilde \rho_1,\Lambda}) \gtrsim u(A^-_{\tilde \rho_0,\Lambda}),
\end{split}
\end{equation}
as the constant appearing in Lemma \ref{lem4.7bol+} depends on $m$ $\kappa$, $M$ and $\delta$, and with $\delta$ in this case fixed as $\delta = 1/c_0$, where $1\leq c_0 <\infty$ is the constant associated to $\Omega_{2\tilde r}$ in accordance with Remark \ref{remnot}.

Next, note that  given a small degree of freedom $0<\delta\ll 1$, if $\varepsilon\leq\varepsilon_0\leq \delta/c_0^2$, then we obviously have
\begin{equation}
    \begin{split}
        A^+_{\tilde \rho_1,\Lambda} + (0,0,0,\varepsilon^3 r^3,0) \in Q_{\delta\tilde \rho_1}(A^+_{\tilde \rho_1,\Lambda}),\\
        A^-_{\tilde \rho_1,\Lambda} + (0,0,0,\varepsilon^3 r^3,0) \in Q_{\delta\tilde \rho_1}(A^-_{\tilde \rho_1,\Lambda}).
    \end{split}
\end{equation}
We can now choose $\delta$, depending on $m$, $\kappa$, $M$ and $c_0$, so that we again can apply  Lemma \ref{lem4.7bol+} to conclude that
\begin{equation}
\begin{split}
\tilde u(A^+_{\tilde \rho_1,\Lambda}) &\leq \sup_{Q_{\delta\tilde \rho_1}(A^+_{\tilde \rho_1,\Lambda})} u \lesssim u(A^+_{\tilde \rho_0,\Lambda}),\\
\tilde u(A^-_{\tilde \rho_1,\Lambda}) &\geq \inf_{Q_{\delta\tilde \rho_1}(A^-_{\tilde \rho_1,\Lambda})} u \gtrsim u(A^-_{\tilde \rho_0,\Lambda}).
\end{split}
\end{equation}
By combining this with \eqref{cor1} the lemma follows.
\end{proof}

\begin{lemma}\label{lem4.6+}  Let $\Omega\subset\mathbb R^{N+1}$ be a Lipschitz domain with  constant $M$ and assume in addition \eqref{struct}. Let
$(\hat Z_0,\hat t_0)\in\partial\Omega$, $r>0$ and let $\Lambda$, $c_0$, $\eta$, $\rho_0=r/c_0$, $\rho_1=\rho_0/c_0$ be in accordance with Remark \ref{remnot}.
Assume that $u$ is a non-negative weak solutions to $\L u=0$ in $\Omega_{2r}(\hat Z_0,\hat t_0)$ and that $u$  vanishes continuously on $\Delta_{2r}(\hat Z_0,\hat t_0)$.  Let
  \begin{eqnarray*}
 m^+=u(A_{\rho_0,\Lambda}^+(\hat Z_0,\hat t_0)),\qquad m^-=u(A_{\rho_0,\Lambda}^-(\hat Z_0,\hat t_0)),
 \end{eqnarray*}
and assume that $m^->0$. Then there exist constants $1\leq c_1(m,\kappa,M)<\infty$ and $1\leq
c_2(m,\kappa,M,m^+/m^-)<\infty$, such that
\begin{equation}\label{lem4.6+statement}
u(A^-_{\rho,\Lambda}(Z_0,t_0))\leq c_2u(A_{\rho,\Lambda}(Z_0,t_0))\leq c_2^2u(A^+_{\rho,\Lambda}(Z_0,t_0)),
\end{equation}
whenever $0<\rho<\rho_1/c_1$ and  $(Z_0,t_0)\in \Delta_{\rho_1}(\hat Z_0,\hat t_0)$.
\end{lemma}

\begin{proof}
We start by noticing that the path
\[
\gamma(\tau) := (Z_0,t_0) \circ (0,\Lambda\rho,0,\tau\Lambda\rho,-\tau),\quad \tau\in\left[0,\rho^2\right]
\]
is admissible and that
\begin{equation*}
    \begin{split}
        \gamma(0) &= A_{\rho,\Lambda}(Z_0,t_0),\\
        \gamma(\rho^2) &= (Z_0,t_0)\circ (0,\Lambda\rho,0,\Lambda\rho^3,-\rho^2).
    \end{split}
\end{equation*}
Furthermore, by (\ref{pointsref2}) and as $\Omega$ is $y_m$-independent we have by construction that
\[
\gamma\left(\left[0,\rho^2\right]\right)\subset \Omega_{2r}(\hat Z_0,\hat t_0).
\]
By constructing a Harnack chain we can deduce, using Lemma \ref{akkaa-}, that
\begin{equation}\label{lem4.6+pf_est1}
    u(\gamma(\rho^2))\lesssim u(A_{\rho,\Lambda}(Z_0,t_0)).
\end{equation}
Noting that $A^-_{\rho,\Lambda}(Z_0,t_0)$ and $\gamma(\rho^2)$ only differ in the $y_m$-coordinate and applying Lemma \ref{T:backprel} we obtain
\begin{equation}
    \frac{m^-}{m^+} \lesssim \frac{u(\gamma(\rho^2))}{u(A^-_{\rho,\Lambda}(Z_0,t_0))} \lesssim\frac{m^+}{m^-},
\end{equation}
when $0<\rho<\rho_1/c$. The above inequality together with (\ref{lem4.6+pf_est1}) yields that
\[
u(A^-_{\rho,\Lambda}(Z_0,t_0))\lesssim\frac{m^+}{m^-} u(A_{\rho,\Lambda}(Z_0,t_0)),
\]
when $0<\rho<\rho_1/c$. This proves the first part of (\ref{lem4.6+statement}), and the second part is proven in analogy. We omit further details.
\end{proof}

\section{Proof of theorem \ref{thm:back}}\label{sec10}
To prove  Theorem \ref{thm:back} it suffices to prove the following lemma.

\begin{lemma}\label{thm:backagain} Let $\Omega\subset\mathbb R^{N+1}$ be a Lipschitz domain with constant $M$ and assume in addition \eqref{struct}.  Let
$(Z_0,t_0)\in\partial\Omega$, $r>0$ and let $\Lambda$, $c_0$, $\eta$, $\rho_0=r/c_0$ be in accordance with Remark \ref{remnot}. Assume that  $u$ is a non-negative solution to $\L u=0$ in $\Omega_{2r}(Z_0,t_0)$,
vanishing continuously on $\Delta_{2r}(Z_0,t_0)$. Let
 \begin{eqnarray}\label{singa1}
 m^+=u(A_{\rho_0,\Lambda}^+(Z_0,t_0)),\ m^-=u(A_{\rho_0,\Lambda}^-(Z_0,t_0)),
 \end{eqnarray}
and assume that $m^->0$. Then there exist constants
$c_1=c_1(m,\kappa,M)$,
$1\leq c_1<\infty$, $c_2=c_2(m,\kappa,M, m^+/m^-)$,
$1\leq c_2<\infty$,  such that if we let $\rho_1=\rho_0/c_1$, then
\begin{equation*}
u(Z,t)\leq c_2u(A_{\rho,\Lambda}(Z_0,t_0)),
\end{equation*}
whenever $(Z,t)\in \Omega_{\rho/c_1}(Z_0,t_0)$, and  $0<\rho<\rho_1$.
\end{lemma}
\begin{proof} We will assume that $(Z_0,t_0)=(0,0)$. To prove the lemma we have to show that there exist constants $1\leq c_1 < \infty$ depending only on $m$, $\kappa$ and $M$, and $1\leq c_2<\infty$ depending only on $m$, $\kappa$, $M$ and $m^+/m^-$, such that
\[
u(Z,t)\leq c_2 u(A_{\rho,\Lambda}),
\]
when $(Z,t)\in \Omega_{\rho/c_1}$ and $0<\rho<\rho_1$. Let $\rho_0$ be defined as in the statement of the lemma, and let $\rho$, $0<\rho<\rho_1$, be fixed. Consider the number
\begin{equation}
    \rho^\ast := \max\lbrace \tilde \rho \mid \rho\leq\tilde \rho \leq \rho_0,\: \phi(\tilde \rho)\geq \phi(\rho) \rbrace,
\end{equation}
where
\begin{equation}
    \phi(\tilde \rho) := \tilde{\rho}^{-\gamma}u(A^+_{\tilde \rho,\Lambda}),
\end{equation}
and where $\gamma$ is the constant in Lemma \ref{lem4.7rho}. Then
\begin{equation}\label{backpf_est1}
    u(A^+_{\rho,\Lambda})\leq \left(\frac{\rho}{\rho^\ast}\right)^{\gamma}u(A^+_{\rho^\ast,\Lambda}).
\end{equation}
From Lemma \ref{lem4.7rho} we obtain
\begin{equation}\label{backpf_est2}
    u(A^-_{\rho^\ast,\Lambda})\leq c\left(\frac{\rho^\ast}{\rho}\right)^\gamma u(A^-_{\rho,\Lambda}).
\end{equation}
We prove the following claim.

\noindent
\textit{Claim: There exists a constant $\mathbf{c} = \mathbf{c} (m,\kappa,M,m^+/m^-)$, $1\leq \bf c <\infty$, such that
\begin{equation}\label{thmbackpf_claim}
    u(A^+_{\rho^\ast,\Lambda})\leq {\bf c}u(A^-_{\rho^\ast,\Lambda}).
\end{equation}}

\noindent
\textit{Proof of the claim. } Let $K\gg1$ be a large degree of freedom. We consider two cases.

\noindent
\textit{Case 1: } $\frac{\rho_0}{8K}<\rho^\ast$. In this case the claim is trivial in the sense that Lemma \ref{lem4.7rho} yields
\[
\frac{u(A^+_{\rho^\ast,\Lambda})}{u(A^-_{\rho^\ast,\Lambda})}\lesssim\left( \frac{\rho_0}{\rho^\ast}\right)^{2\gamma} \frac{u(A^+_{\rho_0,\Lambda})}{u(A^-_{\rho_0,\Lambda})},
\]
so that
\[
u(A^+_{\rho^\ast,\Lambda}) \lesssim \frac{m^+}{m^-} \left( \frac{\rho_0}{\rho^\ast}\right)^{2\gamma}u(A^-_{\rho^\ast,\Lambda}).
\]

\noindent
\textit{Case 2: } $\frac{\rho_0}{8K} \geq \rho^\ast$.
In this case we first note, by the definition of $\rho^\ast$, that  $\rho<\rho^\ast<\rho_0$ and as $\rho^\ast<2K\rho^\ast<\rho_0$ it follows from the definition of $\rho^\ast$ that
\[
u(A^+_{\rho^\ast,\Lambda}) > (2K)^{-\gamma}u(A^+_{2K\rho^\ast,\Lambda}).
\]
Using this  inequality and applying Theorem \ref{thm:carleson} we see that
\begin{equation}\label{backpfclaimest1}
    c^{-1}(2K)^{-\gamma} \sup_{\Omega_{2K\rho^\ast/c}}u \leq u(A^+_{\rho^\ast,\Lambda}),
\end{equation}
for some  $1\leq c < \infty$ depending on $m$, $\kappa$ and $M$. By dilation we can from now on and  without loss of generality assume that $\rho^\ast=1$ and we put $\tilde K:=K/c$. Then \eqref{backpfclaimest1} states that
\begin{equation}\label{backpfclaimest2}
    \tilde{c}^{-1}(2\tilde{K})^{-\gamma} \sup_{U_{2\tilde{K}}}u \leq u(A^+_{1,\Lambda}),
\end{equation}
for some $1\leq \tilde c < \infty$ depending on $m$, $\kappa$ and $M$ and where $U_{R}$ for $0<R<\infty$ is the thin in time  cylinder
\begin{equation}\label{thincylinder}
    U_{R} := \Omega_{R}\cap\lbrace (Z,t)\in\R^{N+1}\mid -4<t<1 \rbrace.
\end{equation}
We may also assume, without loss of generality, that
\[
\sup_{U_{2\tilde K}}u = 1,
\]
which transforms (\ref{backpfclaimest2}) into
\begin{equation}\label{esta}
    u(A^+_{1,\Lambda}) \geq \tilde{c}^{-1}(2\tilde{K})^{-\gamma}.
\end{equation}
We define \begin{equation}\label{thincylsurfs}
    \begin{split}
        \Sigma^-_{\tilde K} &:= (\partial_K U_{2\tilde K}) \cap \lbrace (Z,t)\in\R^{N+1}\mid t=-4 \rbrace,\\
        \Sigma^{\text{int}}_{\tilde K} &:= ((\partial_K U_{2\tilde K})\backslash \Sigma^-_{\tilde K}) \backslash \Delta_{2\tilde K}.
    \end{split}
\end{equation}
Then, given a weak solution $u$ to $\L u=0$ in $\Omega_{2\tilde K}$, vanishing continuously on $\Delta_{2\tilde K}$, we have by Theorem \ref{thm:dp} that if $(Z,t)\in U_{2\tilde K}$, then
\begin{equation}\label{thm3.4pf_repr}
    u(Z,t) = \iint_{\Sigma^-_{\tilde K}}u(\tilde Z,\tilde t)\d\omega(Z,t,\tilde Z,\tilde t) + \iint_{\Sigma^{\text{int}}_{\tilde K}}u(\tilde Z,\tilde t)\d\omega(Z,t,\tilde Z,\tilde t).
\end{equation}
Using \eqref{thm3.4pf_repr}, applying Lemma \ref{important1} stated and proved below, and \eqref{esta} we see that
\begin{equation}
\begin{split}
    u(A^+_{1,\Lambda}) &\leq \iint_{\Sigma^-_{\tilde K}}u(Z,t)\d\omega(A^+_{1,\Lambda},Z,t) + \frac{1}{2}\tilde{c}^{-1}(2\tilde{K})^{-\gamma}\\
    &\leq \iint_{\Sigma^-_{\tilde K}}u(Z,t)\d\omega(A^+_{1,\Lambda},Z,t) + \frac{1}{2}u(A^+_{1,\Lambda}),
\end{split}
\end{equation}
where $\Sigma^-_{\tilde K}$ is as in (\ref{thincylsurfs}). Thus,
\begin{equation}\label{backpfclaimest3}
    u(A^+_{1,\Lambda}) \leq 2\sup_{\Sigma^-_{\tilde K}} u.
\end{equation}
To proceed we use again Theorem \ref{thm:carleson} to see that for each choice of $(\tilde Z,\tilde t)\in \Sigma^-_{\tilde K} \cap \partial\Omega$ there exists a small $\varepsilon>0$, depending only on $m$, $\kappa$ and $M$, such that
\[
\sup_{\Sigma^-_{\tilde K}\cap\Omega_{\varepsilon}(\tilde Z,\tilde t)}u \leq c u(A^+_{c\varepsilon,\Lambda}(\tilde Z,\tilde t)).
\]
This inequality, together with (\ref{backpfclaimest3}) and Lemma \ref{coneconditions-}, yield
\[
u(A^+_{1,\Lambda}) \leq 2c u(\hat Z,\hat t),
\]
for some $(\hat Z,\hat t)\in \tilde{U}_{2\tilde K,\varepsilon}$ where
\begin{equation}\label{thincylthicsurf}
\begin{split}
      \tilde{U}_{2\tilde K,\varepsilon} :=&\: U_{2\tilde K} \cap \lbrace (Z,t)\in\R^{N+1}\mid -4\leq t \leq -4+(c\varepsilon)^2 \rbrace\\ &\cap \lbrace (Z,t)\in\R^{N+1} \mid d((Z,t),\partial\Omega)\geq \varepsilon/c \rbrace.
\end{split}
  \end{equation}
The claim now follows from an application of Lemma \ref{important2} which is stated and proved below. \qed

Using Theorem \ref{thm:carleson}, (\ref{backpf_est1}), (\ref{backpf_est2}) and the claim we see that
\begin{equation}
    \begin{split}
        \sup_{\Omega_{\rho/c}}u &\leq cu(A^+_{\rho,\Lambda}) \leq c\left(\frac{\rho}{\rho^\ast}\right)^{\gamma}u(A^+_{\rho^\ast,\Lambda})\\
        &\leq c{\bf c}\left(\frac{\rho}{\rho^\ast}\right)^{\gamma} u(A^-_{\rho^\ast,\Lambda}) \leq
         c^2 {\bf c} u(A^-_{\rho^\ast,\Lambda}).
    \end{split}
\end{equation}
Existence of the sought constants $c_1$ and $c_2$ now follow from Lemma \ref{lem4.6+} and the proof of the lemma, and therefore Theorem \ref{thm:back}, is complete.
\end{proof}

\subsection{Technical lemmas: Lemma \ref{important1} and Lemma \ref{important2}}

\begin{lemma}\label{important1} Let $\tilde c$ and $\gamma$ be as in \eqref{esta}. Then there exists $\tilde K=\tilde K(m,\kappa,M)$, $1<K<\infty$, such that
$$\iint_{\Sigma^{\text{int}}_{\tilde K}}u(Z,t)\d\omega(A^+_{1,\Lambda}, Z, t) \leq \frac 12  \tilde c^{-1}(2 \tilde K)^{-\gamma}.$$
\end{lemma}
\begin{proof} We have
\begin{equation}
    \iint_{\Sigma^{\text{int}}_{\tilde K}}u(Z,t)\d\omega(A^+_{1,\Lambda}, Z, t) \leq \omega(A^+_{1,\Lambda},\Sigma^{\text{int}}_{K}).
\end{equation}
Let $1\leq\lambda<\tilde K$ be a degree of freedom. We construct similarly to in the proof of Lemma \ref{boundaryholder} a test function $\phi\in C^\infty(\R^{N+1})$, $0\leq\phi\leq 1$, such that
\begin{equation}
    \begin{cases}
        \phi(Z,t) &= 1, \: \text{for }(Z,t)\in Q_{\tilde K+\lambda}\backslash Q_{\tilde K-\lambda} \cap \lbrace (Z,t)\in\R^{N+1}\mid t=-4 \rbrace,\\
        \phi(Z,t) &= 0, \: \text{for }(Z,t)\in Q_{\tilde K-\lambda-1}\cap  \lbrace (Z,t)\in\R^{N+1}\mid t=-4 \rbrace,\\
        \phi(Z,t) &= 0, \: \text{for }(Z,t)\in (\R^{N+1}\backslash Q_{\tilde K+\lambda+1}) \cap \lbrace (Z,t)\in\R^{N+1}\mid t=-4 \rbrace,
    \end{cases}
\end{equation}
and let $\Phi$ be the solution to
\[
\begin{cases}
\L \Phi = 0,\quad \text{in } \lbrace (Z,t)\in\R^{N+1}\mid t>-4  \rbrace,\\
\Phi(Z,-4) = \phi(Z,-4).
\end{cases}
\]
Arguing as in the proof of Claim 1 in the proof of Lemma \ref{boundaryholder}, we see that there exist $1\leq c < \infty$ and $1\leq\lambda <\infty$ both only depending on $m$, $\kappa$ and $M$, such that
\[
\Phi(Z,t) \geq c^{-1},
\]
whenever $(Z,t)\in \Sigma^{\text{int}}_{\tilde K}$.
It follows then by the maximum principle that
\begin{equation}\label{importantpf_est1}
    \iint_{\Sigma^{\text{int}}_{\tilde K}}u(Z,t)\d\omega(A^+_{1,\Lambda}, Z, t) \leq c\Phi(A^+_{1,\Lambda}).
\end{equation}
Furthermore, by noting that
\[
\Phi(A^+_{1,\Lambda}) = \iint_{\R^{N+1}}\Gamma(A^+_{1,\Lambda},\hat Z,-4)\phi(\hat Z,-4)\d\hat Z,
\]
and arguing as in the proof of Claim 2 in the proof of Lemma \ref{boundaryholder}, we may deduce that
\begin{equation}
    \Phi(A^+_{1,\Lambda}) \leq c e^{-c^{-1}{\tilde K}^{2}}{\tilde K}^{\eta},
\end{equation}
with both $1\leq c<\infty$ and $1<\eta <\infty$ independent of $\tilde K$. Combining the above with (\ref{importantpf_est1}) yields
\begin{equation}
    \iint_{\Sigma^{\text{int}}_{\tilde K}}u(Z,t)\d\omega(A^+_{1,\Lambda}, Z, t) \leq ce^{-c^{-1}{\tilde K}^{2}}{\tilde K}^{\eta}.
\end{equation}
Thus the lemma follows by picking $\tilde K$ large enough.
\end{proof}

\begin{lemma}\label{important2}  Let $(Z,t)\in \tilde{U}_{\tilde K,\varepsilon}$ be arbitrary. Then there exists a constant $$c=c(m,\kappa,M,\varepsilon,\tilde K,m^+/m^-),$$ such that
\[
u(Z,t) \leq cu(A^-_{1,\Lambda}).
\]
\end{lemma}
\begin{proof}
This is a consequence of Lemma 6.2 in \cite{NP}. Indeed, in \cite{NP} it is proved, using Lemma \ref{akkaa-} and Lemma \ref{akkaa}, that it is possible to construct an admissible path,
connecting an arbitrary point $(Z,t)\in \tilde{U}_{\tilde K,\varepsilon}$ to $A^-_{1,\Lambda}$, and an associated Harnack chain so that one, in combination with
 Lemma \ref{T:backprel} can prove that
 $$u(Z,t) \leq c(m,\kappa,M,\varepsilon,\tilde K,m^+/m^-)u(A^-_{1,\Lambda}).$$
 We omit further details and refer the reader to Lemma 6.2 in \cite{NP}.
\end{proof}

\section{Proof of theorem \ref{thm:quotients}}\label{sec11}
In this section we give the proof of Theorem \ref{thm:quotients}. To set up, let $\Omega\subset\R^{N+1}$ be a Lipschitz domain with constant $M$ and assume \eqref{struct}. Without loss of generality we assume that $(0,0)\in\partial\Omega$. Let $r>0$ and consider $\Omega_{2r}=\Omega_{2r}(0,0)$. Let $\Lambda$, $c_0$, $\rho_0$ and $\rho_1$ be as in Remark \ref{remnot}.  Assume that $u$ and $v$ are non-negative weak solutions to $\L u=0$ in $\Omega$, vanishing continuously on $\Delta_{2r}=\Delta_{2r}(0,0)$. Let $m^\pm_1$ and $m^\pm_2$ be as in (\ref{singa1u1}) and recall that by Lemma \ref{lem4.7}, we have
\[
\min\lbrace m^-_1,m^-_2 \rbrace>0 \implies \min\lbrace m^+_1,m^+_2 \rbrace > 0.
\]
For any $\tilde r>0$ and $(Z,t)\in\Omega_{2r}\cup \Delta_{2r}$ we will in the following use the notation
\[
\osc_{\tilde r,(Z,t)}(u) := \sup_{\Omega_{2r}\cap Q_{\tilde r}(Z,t)} u - \inf_{\Omega_{2r}\cap Q_{\tilde r}(Z,t)} u
\]
for the oscillation of a function $u$ in the set ${\Omega_{2r}\cap Q_{\tilde r}(Z,t)}$.
Our proof will be based on a decrease of the oscillation
\[
\osc_{\tilde r,(Z,t)}\left( \frac{v}{u} \right).
\]
Note first that by Lemma \ref{lem:compprinciple} and the assumptions on $m^\pm_1$ and $m^\pm_2$, there exists a constant $c=c(m,\kappa,M)$ such that
\[
\osc_{\rho_1/c,(0,0)}\left(\frac{v}{u}\right) \leq c\frac{m^-_1}{m^+_2} - c^{-1}\frac{m^+_1}{m^-_2} < \infty.
\]
Let now $\tilde{r}>0$ be fixed and let $\varepsilon=\varepsilon(m,\kappa,M)$, $0<\varepsilon<1$ be a small degree of freedom to be chosen later. Fix $0<\rho<\varepsilon \tilde{r}$ and $(Z,t)\in\Omega_{\varepsilon \tilde{r}}$. We consider two cases.

\noindent
\textit{Case 1: $Q_{\rho}(Z,t)\subset \Omega_{2r}$}.

In this case we assume that $\rho \leq d((Z,t),\Delta_{2r})$. Consider the function
\begin{equation}\label{quotpfvdef}
\tilde v(\tilde Z,\tilde t) := \left(\osc_{{\rho},(Z,t)}\left(\frac{v}{u}\right)\right)^{-1}\left[ v(\tilde Z,\tilde t) -  \inf_{\Omega_{2r}\cap Q_{\rho}(Z,t)} \left( \frac{v}{u} \right) u(\tilde Z,\tilde t) \right].
\end{equation}
Two immediate consequences of the definition of the above function are
\begin{equation}\label{quotpf_1}
    0\leq \frac{\tilde v(\tilde Z, \tilde t)}{u(\tilde Z,\tilde t)}\leq 1,
\end{equation}
when $(\tilde Z,\tilde t)\in\Omega_{2r}\cap Q_{\rho}(Z,t)$, and
\begin{equation}\label{quotpf_2}
    \osc_{{\rho},(Z,t)}\left(\frac{\tilde v}{u}\right)=1.
\end{equation}
We introduce an additional small degree of freedom $0<\delta<1$ and assume
\begin{equation}\label{quotpf_3}
    \frac{\tilde v((Z,t)\circ (0,-\delta\rho^2))}{u((Z,t)\circ (0,-\delta\rho^2))} \geq \frac{1}{2}.
\end{equation}
By Lemma \ref{harnack}, we see there exists $\tilde \delta =\tilde \delta(m,\delta)$, $0<\tilde \delta<1$, such that
\begin{equation}
    \tilde v((Z,t)\circ (0,-\delta\rho^2)) \leq c\tilde v(\tilde Z,\tilde t),
\end{equation}
and
\begin{equation}
    u(\tilde Z,\tilde t) \leq cu((Z,t)\circ (0,\delta\rho^2)),
\end{equation}
when $(\tilde Z,\tilde t)\in Q_{\tilde \delta\rho}(Z,t)$. Furthermore, using  Theorem \ref{thm:back} we see that
\begin{equation}\label{thm34pfbackward}
    u((Z,t)\circ (0,\delta\rho^2)) \leq c u((Z,t)\circ (0,-\delta\rho^2)).
\end{equation}
We arrive then by using \eqref{quotpf_1}-\eqref{thm34pfbackward} at
\begin{equation}
    \frac{1}{2}\leq \frac{\tilde v((Z,t)\circ (0,-\delta\rho^2))}{u((Z,t)\circ (0,-\delta\rho^2))} \leq c\frac{\tilde v(\tilde Z,\tilde t)}{u(\tilde Z,\tilde t)}
    \leq c,
\end{equation}
for a constant $c=c(m,\kappa,M,m_1^+/m_1^-, m_2^+/m_2^-)$, when $(\tilde Z,\tilde t)\in Q_{\tilde \delta\rho}(Z,t)$ and thus,
\begin{equation}\label{quotpf_4}
    \osc_{{\tilde \delta\rho},(Z,t)}\left(\frac{\tilde v}{u}\right)\leq \theta,
\end{equation}
with $\theta=1-1/(2c)$, in particular $0<\theta<1$. Plugging in the definition of $\tilde v$ into (\ref{quotpf_4}) yields
\[
\left(\osc_{{\rho},(Z,t)}\left(\frac{v}{u}\right)\right)^{-1}\osc_{{\tilde \delta\rho},(Z,t)}\left(\frac{v}{u}\right) \leq \theta
\]
and consequently
\begin{equation}\label{quotpf_5}
    \osc_{{\tilde \delta\rho},(Z,t)}\left(\frac{v}{u}\right) \leq \theta \osc_{{\rho},(Z,t)}\left(\frac{v}{u}\right).
\end{equation}
Note that if (\ref{quotpf_3}) does not hold then we  replace $\tilde v$ by $u-\tilde v$ and note that (\ref{quotpf_1}), (\ref{quotpf_2}) and (\ref{quotpf_3}) holds for $u-\tilde v$, so that we may deduce that
\[
\osc_{{\tilde \delta\rho},(Z,t)}\left(\frac{u-\tilde v}{u}\right)\leq \theta,
\]
and hence that (\ref{quotpf_5}) holds. By iterating (\ref{quotpf_5}) we obtain
\begin{equation}\label{quotpf_decrease1}
    \osc_{{\rho},(Z,t)}\left(\frac{v}{u}\right) \leq \left(\frac{\rho}{\tilde \delta d((Z,t),\Delta_{2r})}\right)^{\sigma_1} \osc_{{d((Z,t),\Delta_{2r})},(Z,t)}\left(\frac{v}{u}\right),
\end{equation}
for some $\sigma_1=\sigma_1(m,\kappa,M,m_1^+/m_1^-, m_2^+/m_2^-)$, $0<\sigma_1<1$.

\vspace{10pt}

\noindent
\textit{Case 2: $Q_{\rho}(Z,t)\cap (\R^{N+1}\backslash\Omega_{2r})\neq\emptyset$}.

In this case we assume that $\rho > d((Z,t),\Delta_{2r})$.
Fix a point $(Z_0,t_0)\in\Delta_{2r}$ such that $$d((Z,t),(Z_0,t_0))=d((Z,t),\Delta_{2r}).$$
Then there exists a constant ${\bf c}={\bf c}(m,\kappa,M)$, $1\leq{\bf c}<\infty$, such that
\begin{equation}\label{quotpf_lessosc}
    Q_{\rho}(Z,t) \subset Q_{2{\bf c}\rho}(Z_0,t_0),
\end{equation}
and thus
\[
\osc_{{\rho},(Z,t)}\left(\frac{v}{u}\right) \leq \osc_{{2{\bf c}\rho},(Z,t)}\left(\frac{v}{u}\right).
\]
Let $\hat c$ be the constant appearing in Lemma \ref{lem:compprinciple} and put $\tilde \rho := 8{\hat c}{\bf c}\rho$. Define now $\tilde v$ as in (\ref{quotpfvdef}) but with $\rho$ replaced by $\tilde \rho$ and $(Z,t)$ replaced by $(Z_0,t_0)$. Consequently, (\ref{quotpf_1}) and (\ref{quotpf_2}) holds with $\rho$ replaced by $\tilde \rho$ and $(Z,t)$ replaced by $(Z_0,t_0)$. Assume now that
\begin{equation}\label{quotpf_halfcond2}
    \frac{\tilde v(A^-_{\tilde \rho/2,\Lambda}(Z_0,t_0))}{u(A^-_{\tilde \rho/2,\Lambda}(Z_0,t_0))} \geq \frac{1}{2}.
\end{equation}
Since $\tilde v$ and $u$ both are solutions to $\L u=0$ in $\Omega_{2r}$ that are non-negative in $\Omega_{2r}\cap Q_{\tilde \rho}(Z_0,t_0)$ and vanish continuously on $\Delta_{2r}$, we may apply Lemma \ref{lem:compprinciple} and \eqref{quotpf_1} to obtain
\begin{equation}
    \frac{\tilde v(A^-_{\tilde \rho/2,\Lambda}(Z_0,t_0))}{u(A^+_{\tilde \rho/2,\Lambda}(Z_0,t_0))} \leq \hat c \frac{\tilde v(\tilde Z,\tilde t)}{u(\tilde Z,\tilde t)} \leq \hat c,
\end{equation}
when $(\tilde Z,\tilde t)\in \Omega_{2r}\cap Q_{2{\bf c}\rho}(Z_0,t_0)$. Arguing as in the proof of Theorem \ref{thm:back}, recall in particular \eqref{thmbackpf_claim}, we conclude that
\begin{equation}
    \frac{\tilde v(A^-_{\tilde \rho/2,\Lambda}(Z_0,t_0))}{u(A^-_{\tilde \rho/2,\Lambda}(Z_0,t_0))} \leq c \frac{\tilde v(A^-_{\tilde \rho/2,\Lambda}(Z_0,t_0))}{u(A^+_{\tilde \rho/2,\Lambda}(Z_0,t_0))},
\end{equation}
where $c=c(m,\kappa,M,m_2^+/m_2^-)$, and hence
\[
\frac{1}{2}\leq c{\hat c}\frac{\tilde v(\tilde Z,\tilde t)}{u(\tilde Z,\tilde t)} \leq c{\hat c},
\]
whenever $(\tilde Z,\tilde t)\in \Omega_{2r}\cap Q_{2{\bf c}\rho}(Z_0,t_0)$. It follows then that
\begin{equation}
    \osc_{{2{\bf c}\rho},(Z_0,t_0)}\left(\frac{\tilde v}{u}\right)\leq \theta,
\end{equation}
where $\theta = 1-1/(2c{\hat c})$, in particular we have again $0<\theta<1$. Rearranging the above expression similarly as in Case 1 and using (\ref{quotpf_lessosc}), we obtain
\begin{equation}\label{quotpf_6}
    \osc_{{\rho},(Z,t)}\left(\frac{v}{u}\right)\leq \theta \osc_{{\tilde \rho},(Z_0,t_0)}\left(\frac{ v}{u}\right).
\end{equation}
Also similarly to Case 1, if (\ref{quotpf_halfcond2}) does not hold we replace $\tilde v$ by $u-\tilde v$ and conclude again that (\ref{quotpf_6}) holds. Iterating (\ref{quotpf_6}) yields
\begin{equation}\label{quotpf_decrease2}
    \osc_{{\rho},(Z,t)}\left(\frac{v}{u}\right) \leq
    \left(\frac{\tilde \rho}{\tilde{r}}\right)^{\sigma_2} \osc_{{\tilde{r}},(Z_0,t_0)}\left(\frac{v}{u}\right),
\end{equation}
for some $\sigma_2=\sigma_2(m,\kappa,M,m_1^+/m_1^-, m_2^+/m_2^-)$, $0<\sigma_2<1$.

Now we combine (\ref{quotpf_decrease1}) and (\ref{quotpf_decrease2}), and put $\sigma = \min\lbrace \sigma_1,\sigma_2 \rbrace$ to arrive at
\begin{equation}
    \osc_{{\rho},(Z,t)}\left(\frac{v}{u}\right) \leq c\left( \frac{\rho}{\tilde{r}} \right)^\sigma \osc_{{\tilde{r}},(Z_0,t_0)}\left(\frac{v}{u}\right)
\end{equation}
whenever $(Z,t)\in \Omega_{\rho}$, $\rho \leq \epsilon \tilde{r}$. Let now $(\tilde Z,\tilde t)\in\Omega_{\rho}$.
From the above inequality it follows by choosing $\varepsilon$ small enough, and applying Lemma \ref{lem:compprinciple} and Theorem \ref{thm:back}, that
\begin{equation}
\begin{split}
    \left| \frac{v(Z,t)}{u(Z,t)} - \frac{v(\tilde Z,\tilde t)}{u(\tilde Z,\tilde t)} \right| &\leq
    \osc_{{d((Z,t),(\tilde Z,\tilde t))},(Z,t)}\left(\frac{v}{u}\right)\\
    &\leq c\left(\frac{d((Z,t),(\tilde Z,\tilde t))}{\tilde{r}}\right)^\sigma \osc_{{\tilde{r}},(0,0)}\left(\frac{v}{u}\right)\\
    &\leq c\left(\frac{d((Z,t),(\tilde Z,\tilde t))}{\tilde{r}}\right)^\sigma \frac{v(A_{\tilde{r},\Lambda})}{u(A_{\tilde{r},\Lambda})}.
\end{split}
\end{equation}
This concludes the proof.

\section{Proof of Theorem \ref{thm:doub}}\label{sec12}
The purpose of this section is to prove Theorem \ref{thm:doub}. The content of the theorem is that if $\Omega$ is a Lipschitz domain and \eqref{struct} holds, then the associated Kolmogorov measure is a doubling measure, with doubling constant depending only on $m$, $\kappa$, and $M$. Theorem \ref{thm:doub} follows immediately from the following lemma as we without loss of generality can assume that $(Z_0,t_0)=(0,0)$.

\begin{lemma}\label{lem.doub}
Let $\Omega$ be a Lipschitz domain with constant $M$ and assume \eqref{struct}. Let $r>0$ and let $\Lambda$ be as in Remark \ref{remnot}. Then there exists a constant $c=c(m,\kappa,M)$, $1\leq c<\infty$, such that
\begin{equation}
    \omega(A^+_{r,\Lambda},\Delta_{2\tilde r}(\tilde Z_0,\tilde t_0)) \lsim \omega(A^+_{r,\Lambda},\Delta_{\tilde r}(\tilde Z_0,\tilde t_0)),
\end{equation}
when $(\tilde Z_0,\tilde t_0)\in\partial\Omega$ and $\Delta_{\tilde r}(\tilde Z_0,\tilde t_0)\subset \Delta_{r/c}$.
\end{lemma}

\begin{proof}
Let $(\tilde Z_0,\tilde t_0)\in\partial\Omega$ and $\tilde r>0$ be such that $\Delta_{\tilde r}(\tilde Z_0,\tilde t_0)\subset \Delta_{ r/{\bf c}},$ where ${\bf c}={\bf c}(m,\kappa,M)$, ${\bf c}>1$ is a large degree of freedom to be chosen.
Then by choosing ${\bf c}$ large enough and applying Lemma \ref{gensalsa-asecondpart}, we obtain
\begin{equation}\label{doubpf_1}
    \omega(A^+_{r,\Lambda},\Delta_{2\tilde r}(\tilde Z_0,\tilde t_0)) \lsim r^{{\bf q}-2}G(A^+_{r,\Lambda},A^-_{2{\tilde c}\tilde r,\Lambda}(\tilde Z_0,\tilde t_0)),
\end{equation}
for some $\tilde c = \tilde c(m,\kappa,M)$, $1\leq\tilde c<\infty$, where $G(A^+_{r,\Lambda},\cdot)$ is the adjoint Green function for $\Omega$ with pole at $A^+_{r,\Lambda}$.

\noindent
\textit{Claim: }It holds that
\begin{equation}\label{lemdoub_claim}
    \frac{G(A^+_{r,\Lambda},A^+_{r/1000,\Lambda})}{G(A^+_{r,\Lambda},A^-_{r/1000,\Lambda})} \approx 1.
\end{equation}

\noindent
To prove the claim we first deduce using the definition of the Green function, the maximum principle and Lemma \ref{lem4.7rho}, and  by a similar argument as in the proof of Claim 2 in the proof of Lemma \ref{lem:compprinciple}, that
\begin{equation}
    \begin{split}
        \hat c^{-1} \leq r^{{\bf q}-2}G(A^+_{r,\Lambda},A^+_{r/1000,\Lambda}) \leq \hat c,\\
        r^{{\bf q}-2}G(A^+_{r,\Lambda},A^-_{r/1000,\Lambda}) \leq \hat c,
    \end{split}
\end{equation}
for some $\hat c=\hat c(m,\kappa,M)$, $1\leq\hat c<\infty$.
We need to establish a lower bound for $G(A^+_{r,\Lambda},A^-_{r/1000,\Lambda})$. Using the adjoint version of Theorem \ref{thm:carleson} we see that
\begin{equation}
    \sup_{\Omega_{\frac{r}{{\bf c}}}} G(A^+_{r,\Lambda},(Z,t))\lsim  G(A^+_{r,\Lambda},A^-_{r/1000,\Lambda}).
\end{equation}
On the other hand, arguing again similarly as in the proof of Claim 2 in the proof of Lemma \ref{lem:compprinciple}, we see that
\begin{equation}
     \sup_{\Omega_{\frac{r}{{\bf c}}}} G(A^+_{r,\Lambda},(Z,t)) \geq G(A^+_{r,\Lambda},A^+_{r/(100{\bf c}),\Lambda}) \gtrsim r^{2-{\bf q}}.
\end{equation}
The claim follows.

Using the claim, the adjoint version of Theorem \ref{thm:back}, and the scale-invariance of Theorem \ref{thm:back} we find that
\begin{equation}\label{doubpf_2}
    G(A^+_{r,\Lambda},A^-_{2\tilde c\tilde r}(\tilde Z_0,\tilde t_0)) \lsim G(A^+_{r,\Lambda},A^+_{2\tilde c\tilde r}(\tilde Z_0,\tilde t_0)),
\end{equation}
whenever $(\tilde Z_0,\tilde t_0)\in\partial\Omega$ and $Q_{\tilde r}(\tilde Z_0,\tilde t_0)\subset Q_{r/{\bf c}}$.
Now an application of the adjoint version of Lemma \ref{lem4.7rho} and Lemma \ref{gensalsa-a} yields
\begin{equation}\label{doubpf_3}
\tilde r^{{\bf q}-2} G(A^+_{r,\Lambda},A^+_{2\tilde c\tilde r}(\tilde Z_0,\tilde t_0)) \lsim \tilde r^{{\bf q}-2} G(A^+_{r,\Lambda},A^+_{\tilde r}(\tilde Z_0,\tilde t_0)) \lsim \omega(A^+_{r,\Lambda},\Delta_{\tilde r}(\tilde Z_0,\tilde t_0)).
\end{equation}
Combining (\ref{doubpf_1}), (\ref{doubpf_2}), and (\ref{doubpf_3}) finishes the proof.
\end{proof}

\subsection{Estimates of the kernel function}
We end this section by proving two further estimates. Lemma \ref{lem4.5-Kyoto1ha} and Lemma \ref{lemmacruc-} are analogues in our setting to Lemma 4.13 and Lemma 4.14 in \cite{N2}, where they are proven for the operator $\K$.

\begin{lemma}\label{lem4.5-Kyoto1ha}Let $\Omega$ be a Lipschitz domain with constant $M$ and assume \eqref{struct}. Let $(Z_0,t_0)\in\partial\Omega$ and $r>0$. Let $\Lambda$ be in accordance with Remark \ref{remnot}. Let $(\tilde Z_0,\tilde t_0)\in\partial\Omega$ and $\tilde r>0$ be such that $Q_{\tilde r}(\tilde Z_0,\tilde t_0)\subset  Q_{r}(Z_0,t_0)$.  Then there exists  $c=c(m,\kappa,M)$,  $1\leq c<\infty$, such that
 \begin{eqnarray}\label{ad}
K(A_{c\tilde r,\Lambda}^+(\tilde Z_0,\tilde t_0),\bar Z,\bar t):=\lim_{\bar r\to 0}\frac{\omega(A_{c\tilde r,\Lambda}^+(\tilde Z_0,\tilde t_0),\Delta_{\bar r}(\bar Z,\bar t))}{\omega(A_{cr,\Lambda}^+(Z_0,t_0),\Delta_{\bar r}(\bar Z,\bar t))}
\end{eqnarray}
exists for a.e. $(\bar Z,\bar t)\in \Delta_{\tilde r}(\tilde Z_0,\tilde t_0)$, and
\begin{eqnarray}\label{ad1}
c^{-1}\leq {\omega(A_{cr,\Lambda}^+(Z_0,t_0), \Delta_{\tilde r}(\tilde Z_0,\tilde t_0))}K(A_{c\tilde r,\Lambda}^+(\tilde Z_0,\tilde t_0),\bar Z,\bar t)\leq c
\end{eqnarray}
whenever $(\bar Z,\bar t)\in \Delta_{\tilde r}(\tilde Z_0,\tilde t_0)$.
\end{lemma}

\begin{proof}
First of all, we notice that \eqref{ad} exists for a.e. $(\bar Z,\bar t)\in \Delta_{\tilde r}(\tilde Z_0,\tilde t_0)$ by Lemma \ref{lem4.7rho}. We have to show that \eqref{ad1} holds. Let $(\bar Z,\bar t)\in \Delta_{\tilde r}(\tilde Z_0,\tilde t_0)$ and let $\bar r < \tilde r$.
We see that as a consequence of Lemma \ref{gensalsa-a} and Lemma \ref{gensalsa-asecondpart} there exists  $c=c(m,\kappa,M)$,  $1\leq c<\infty$ so that
\begin{equation}
    \begin{split}
        \frac{G(A^+_{c\tilde r,\Lambda}(\tilde Z_0,\tilde t_0),A^+_{\bar r,\Lambda}(\bar Z,\bar t))}{G(A^+_{c r,\Lambda}(Z_0,t_0),A^-_{\bar r,\Lambda}(\bar Z,\bar t))}
        &\lesssim \frac{\omega(A^+_{c\tilde r,\Lambda}(\tilde Z_0,\tilde t_0),\Delta_{\bar r}(\bar Z,\bar t))}{\omega(A^+_{c r,\Lambda}(Z_0,t_0),\Delta_{\bar r}(\bar Z,\bar t))}  \\
        &\lesssim \frac{G(A^+_{c\tilde r,\Lambda}(\tilde Z_0,\tilde t_0),A^-_{\bar r,\Lambda}(\bar Z,\bar t))}{G(A^+_{c r,\Lambda}(Z_0,t_0),A^+_{\bar r,\Lambda}(\bar Z,\bar t))}.
    \end{split}
\end{equation}
Next, using the adjoint versions of Theorem \ref{thm:back} and Lemma \ref{lem4.6+}, and arguing as in the proof of Lemma \ref{lem.doub} (see in particular \eqref{doubpf_2}), we deduce that
\begin{equation}
        \frac{G(A^+_{c\tilde r,\Lambda}(\tilde Z_0,\tilde t_0),A_{\bar r,\Lambda}(\bar Z,\bar t))}{G(A^+_{c r,\Lambda}(Z_0,t_0),A_{\bar r,\Lambda}(\bar Z,\bar t))}
        \approx
        \frac{G(A^+_{c\tilde r,\Lambda}(\tilde Z_0,\tilde t_0),A^+_{\bar r,\Lambda}(\bar Z,\bar t))}{G(A^+_{c r,\Lambda}(Z_0,t_0),A^-_{\bar r,\Lambda}(\bar Z,\bar t))}.
\end{equation}
We can then conclude that
\begin{equation}
\frac{\omega(A^+_{c\tilde r,\Lambda}(\tilde Z_0,\tilde t_0),\Delta_{\bar r}(\bar Z,\bar t))}{\omega(A^+_{c r,\Lambda}(Z_0,t_0),\Delta_{\bar r}(\bar Z,\bar t))}
\approx
\frac{G(A^+_{c\tilde r,\Lambda}(\tilde Z_0,\tilde t_0),A_{\bar r,\Lambda}(\bar Z,\bar t))}{G(A^+_{c r,\Lambda}(Z_0,t_0),A_{\bar r,\Lambda}(\bar Z,\bar t))}
\end{equation}
Note that a consequence of the adjoint version of Theorem \ref{thm:quotients} is a boundary Harnack inequality for solutions to the equation $\L^*u=0$. Furthermore, when $u$ and $v$ in the formulation of (the adjoint version of) Theorem \ref{thm:quotients} are the functions $G(A^+_{c\tilde r,\Lambda}(\tilde Z_0,\tilde t_0),\cdot)$ and $G(A^+_{c r,\Lambda}(\tilde Z_0,\tilde t_0),\cdot)$ respectively, we can use Lemma \ref{gensalsa-a} and Lemma \ref{gensalsa-asecondpart} to estimate the quotients $m_1^+/m_1^-$ and $m_2^+/m_2^-$ so that in this particular case we have
\[
c_2(m,\kappa,M,m_1^+/m_1^-,m_2^+/m_2^-) = c_2(m,\kappa,M),
\]
where $c_2$ is as in the formulation of Theorem \ref{thm:quotients}. We hence deduce that there exists some \mbox{$\tilde c=\tilde c(m,\kappa,M)>1$} such that
\begin{equation}
\frac{G(A^+_{c\tilde r,\Lambda}(\tilde Z_0,\tilde t_0),A_{\bar r,\Lambda}(\bar Z,\bar t))}{G(A^+_{c r,\Lambda}(Z_0,t_0),A_{\bar r,\Lambda}(\bar Z,\bar t))}
\approx
\frac{G(A^+_{c\tilde r,\Lambda}(\tilde Z_0,\tilde t_0),A_{\tilde r/\tilde c,\Lambda}(\tilde Z_0,\tilde t_0))}{G(A^+_{c r,\Lambda}(Z_0,t_0),A_{\tilde r/\tilde c,\Lambda}(\tilde Z_0,\tilde t_0))}
\end{equation}
Thus, we see that
\begin{equation}
\frac{\omega(A^+_{c\tilde r,\Lambda}(\tilde Z_0,\tilde t_0),\Delta_{\bar r}(\bar Z,\bar t))}{\omega(A^+_{c r,\Lambda}(Z_0,t_0),\Delta_{\bar r}(\bar Z,\bar t))}
\approx
\frac{G(A^+_{c\tilde r,\Lambda}(\tilde Z_0,\tilde t_0),A_{\tilde r/\tilde c,\Lambda}(\tilde Z_0,\tilde t_0))}{G(A^+_{c r,\Lambda}(Z_0,t_0),A_{\tilde r/\tilde c,\Lambda}(\tilde Z_0,\tilde t_0))}.
\end{equation}
Using arguments similar to those used in Claim 2 in the proof of Lemma \ref{lem:compprinciple}, see in particular \eqref{greenrefestimate}, the (adjoint) Harnack inequality, and arguing as in Lemma \ref{lem.doub}, see \eqref{lemdoub_claim}, and then using Lemma \ref{gensalsa-a}, we see that
\begin{equation}\label{kernelpf_estbelow}
1\lesssim \frac{G(A^+_{c\tilde r,\Lambda}(\tilde Z_0,\tilde t_0),A_{\tilde r/\tilde c,\Lambda}(\tilde Z_0,\tilde t_0))}{G(A^+_{c r,\Lambda}(Z_0,t_0),A_{\tilde r/\tilde c,\Lambda}(\tilde Z_0,\tilde t_0))} \omega(A^+_{c r,\Lambda}(Z_0,t_0),\Delta_{\tilde r/\tilde c}(\tilde Z_0,\tilde t_0)).
\end{equation}
Using \eqref{kernelpf_estbelow} and arguing similarly, wee obtain
\begin{equation}
\frac{G(A^+_{c\tilde r,\Lambda}(\tilde Z_0,\tilde t_0),A_{\tilde r/\tilde c,\Lambda}(\tilde Z_0,\tilde t_0))}{G(A^+_{c r,\Lambda}(Z_0,t_0),A_{\tilde r/\tilde c,\Lambda}(\tilde Z_0,\tilde t_0))}
\approx
\frac{1}{\omega(A^+_{c r,\Lambda}(Z_0,t_0),\Delta_{\tilde r/\tilde c}(\tilde Z_0,\tilde t_0))}.
\end{equation}
Finally, using the obvious relation
\[
\omega(A^+_{c r,\Lambda}(Z_0,t_0),\Delta_{\tilde r/\tilde c}(\tilde Z_0,\tilde t_0)) \leq
\omega(A^+_{c r,\Lambda}(Z_0,t_0),\Delta_{\tilde r}(\tilde Z_0,\tilde t_0))
\]
together with Lemma \ref{lem.doub}, we have
\[
\frac{\omega(A^+_{c\tilde r,\Lambda}(\tilde Z_0,\tilde t_0),\Delta_{\bar r}(\bar Z,\bar t))}{\omega(A^+_{c r,\Lambda}(Z_0,t_0),\Delta_{\bar r}(\bar Z,\bar t))}
\approx
\frac{1}{\omega(A^+_{c r,\Lambda}(Z_0,t_0),\Delta_{\tilde r}(\tilde Z_0,\tilde t_0))}.
\]
Letting $\bar r\rightarrow 0$ concludes the proof.
\end{proof}

\begin{lemma}\label{lemmacruc-} Let $\Omega$ be a Lipschitz domain with constant $M$ and assume \eqref{struct}. Let $(Z_0,t_0)\in\partial\Omega$ and $r>0$. Let $\Lambda$ be in accordance with Remark \ref{remnot}. Let $(\tilde Z_0,\tilde t_0)\in\partial\Omega$ and $\tilde r>0$ be such that $Q_{\tilde r}(\tilde Z_0,\tilde t_0)\subset  Q_{r}(Z_0,t_0)$.  Then there exists  $c=c(m,\kappa,M)$,  $1\leq c<\infty$, such that
 \begin{eqnarray*}
 \frac {\omega(A_{cr,\Lambda}^+(Z_0,t_0),E)}{\omega(A_{cr,\Lambda}^+(Z_0,t_0),\Delta_{\tilde r}(\tilde Z_0,\tilde t_0))}\approx
 \omega(A_{c\tilde r,\Lambda}^+(\tilde Z_0,\tilde t_0),E),
\end{eqnarray*}
whenever $E\subset \Delta_{\tilde r}(\tilde Z_0,\tilde t_0)$.
\end{lemma}
\begin{proof}
Fix some subset $E\subset \Delta_{\tilde r}(\tilde Z,\tilde t)$. By definition we then have
\[
    \omega(A^+_{c\tilde r,\Lambda}(\tilde Z,\tilde t),E) =
    \iint_{E}{ K(A^+_{c\tilde r,\Lambda}(\tilde Z,\tilde t),(\bar Z,\bar t)) \d\omega(A^+_{cr,\Lambda}(Z_0,t_0),(\bar Z,\bar t))}.
\]
Using Lemma \ref{lem4.5-Kyoto1ha}, \eqref{ad1}, we see that
\begin{equation*}
    \begin{split}
        \iint_{E}{ \omega(A_{cr,\Lambda}^+(Z_0,t_0),\Delta_{\tilde r}(\tilde Z_0,\tilde t_0)) K(A^+_{c\tilde r,\Lambda}(\tilde Z,\tilde t),(\bar Z,\bar t)) \d\omega(A^+_{cr,\Lambda}(Z_0,t_0),(\bar Z,\bar t))}\\
        \lesssim \iint_{E}{ \d\omega(A^+_{cr,\Lambda}(Z_0,t_0),(\bar Z,\bar t))}
    \end{split}
\end{equation*}
and hence that
\[
\omega(A_{cr,\Lambda}^+(Z_0,t_0),\Delta_{\tilde r}(\tilde Z_0,\tilde t_0))\omega(A_{c\tilde r,\Lambda}^+(\tilde Z_0,\tilde t_0),E)
\lesssim \omega(A^+_{cr,\Lambda}(Z_0,t_0),E).
\]
Analogously, we see that
\[
\omega(A^+_{cr,\Lambda}(Z_0,t_0),E)\lesssim \omega(A_{cr,\Lambda}^+(Z_0,t_0),\Delta_{\tilde r}(\tilde Z_0,\tilde t_0))\omega(A_{c\tilde r,\Lambda}^+(\tilde Z_0,\tilde t_0),E),
\]
which finishes the proof.
\end{proof}

\end{document}